\documentclass[11pt,a4paper,reqno]{article}
\usepackage{amsmath, amssymb, amsthm}
\usepackage{fullpage}
\usepackage{hyperref}
\usepackage{pst-node}
\usepackage{color}

\usepackage[all]{xy}
\usepackage{tikz-cd}
\tikzcdset{arrow style=tikz, diagrams={>=stealth}}

\numberwithin{equation}{section}

\theoremstyle{plain}
\newtheorem{thm}{Theorem}[section]
\newtheorem{lem}[thm]{Lemma}
\newtheorem{prop}[thm]{Proposition}
 \newtheorem{cor}[thm]{Corollary}
\newtheorem{defi}[thm]{Definition}

\theoremstyle{remark}

\newtheorem{rem}[thm]{Remark}

\newcommand{\ot}{{\otimes}}
\newcommand{\lan}{{\langle}}
\newcommand{\ran}{{\rangle}}
\newcommand{\VL}{{{}^{\vee}\Lambda}}
\newcommand{\LV}{{\cL^{\vee}}}
\newcommand{\UH}{\underline{H}}

\DeclareMathOperator{\Aut}{Aut}
\DeclareMathOperator{\Hom}{Hom}
\newcommand*{\id}{\textup{id}}

\newcommand{\zero}[1]{{#1}{}_{\scriptscriptstyle{(0)}}}
\newcommand{\one}[1]{{#1}{}_{\scriptscriptstyle{(1)}}}
\newcommand{\two}[1]{{#1}{}_{\scriptscriptstyle{(2)}}}
\newcommand{\three}[1]{{#1}{}_{\scriptscriptstyle{(3)}}}
\newcommand{\four}[1]{{#1}{}_{\scriptscriptstyle{(4)}}}
\newcommand{\five}[1]{{#1}{}_{\scriptscriptstyle{(5)}}}
\newcommand{\six}[1]{{#1}{}_{\scriptscriptstyle{(6)}}}

\newcommand{\tuno}[1]{{#1}{}{}^{\scriptscriptstyle{<1>}}}
\newcommand{\tdue}[1]{{#1}{}{}^{\scriptscriptstyle{<2>}}}

\newcommand{\eins}[1]{{#1}{}{}^{\scriptscriptstyle{(1)}}}
\newcommand{\zwei}[1]{{#1}{}{}^{\scriptscriptstyle{(2)}}}

\renewcommand{\o}{{}_{\scriptscriptstyle{(1)}}}
\renewcommand{\t}{{}_{\scriptscriptstyle{(2)}}}
\renewcommand{\th}{{}_{\scriptscriptstyle{(3)}}}
\newcommand{\fo}{{}_{\scriptscriptstyle{(4)}}}
\newcommand{\fiv}{{}_{\scriptscriptstyle{(5)}}}
\newcommand{\si}{{}_{\scriptscriptstyle{(6)}}}
\newcommand{\sev}{{}_{\scriptscriptstyle{(7)}}}
\newcommand{\eig}{{}_{\scriptscriptstyle{(8)}}}
\newcommand{\nin}{{}_{\scriptscriptstyle{(9)}}}
\newcommand{\ten}{{}_{\scriptscriptstyle{(10)}}}

\newcommand{\la}{{\triangleright}}
\newcommand{\ra}{{\triangleleft}}
\newcommand{\eps}{{\varepsilon}}
\DeclareMathOperator{\tens}{\otimes}

\newcommand{\C}{\mathbb{C}}
\newcommand{\tC}{\tilde{C}}

\newcommand{\tZ}{\tilde{Z}}

\newcommand{\B}{\mathcal{B}}
\newcommand{\CB}{\mathcal{B}}
\newcommand{\cL}{\mathcal{L}}
\newcommand{\CL}{\mathcal{L}}
\newcommand{\CG}{\mathcal{G}}
\newcommand{\CR}{\mathcal{R}}
\renewcommand{\H}{\mathcal{H}}
\newcommand{\CH}{\mathcal{H}}
\newcommand{\CO}{\mathcal{O}}

\newcommand{\CM}{\mathcal{M}}
\newcommand{\Q}{\mathcal{Q}}
\newcommand{\R}{\mathcal{R}}

\newcommand{\DYH}{{\ }^{{}^\times}\kern-12pt\CM^H_H}

\newcommand{\adr}{{\textup{Ad}_{R}}}
\newcommand{\ad}{{\textup{ad}}}
\newcommand{\Ad}{{\textup{Ad}}}
\newcommand{\udh}{{\underline{H}}}

\newcommand{\<}{\langle}
\renewcommand{\>}{\rangle}
 
\allowdisplaybreaks

\title{{\bf Bisections and cocycles on Hopf algebroids}}

\author{Xiao Han${}^*$ and Shahn Majid\footnote{x.han@qmul.ac.uk, s.majid@qmul.ac.uk}
\\ \  \\
School of Mathematical Sciences, Queen Mary University of London\\ London E1 4NS, UK}

\begin{document}

\date{}

\maketitle
%\tableofcontents
\parskip = .75 ex
%\allowdisplaybreaks[4]

%\thispagestyle{empty}

\begin{quote} {\bf Abstract} We introduce and study the group  $\CB(\CL)$ of bisections of a Hopf algebroid $\CL$ and show that they form a group crossed module or 2-group with the group $\Aut(\CL)$ of automorphisms. Moreover, the group of vertical bisections turns out to be part of a certain non-Abelian cohomology of which $\CH^2(\CL,B)$ governs cotwisting of a Hopf algebroid with base $B$.  For the Ehresmann-Schauenburg Hopf algebroid $\CL(P,H)$ of a quantum principal bundle or Hopf-Galois extension, $\CB(\CL(P,H))$  reduces to the group $\Aut_H(P)$ of bundle automorphisms and vertical bisections to the group of `gauge transformations' of the bundle. The general  $\CH^2(\CL(P,H),B)$ reduces to a known non-Abelian cohomology in the case where $P$ is a trivial principal bundle or cleft extension. Parallel characterisations are obtained for the bisections  and non-Abelian cohomology of the action Hopf algebroid $B\# H^{op}$ associated to a braided-commutative algebra $B$ in the category of Drinfeld-Yetter modules over a Hopf algebra $H$. Examples include the Heisenberg double or Weyl Hopf algebroid  of a Hopf algebra and a canonical action Hopf algebroid $\underline H\# H^{op}$ when $H$ is coquasitriangular and $\underline{H}$ is its transmutation. \end{quote}

\begin{quote}  2010 Mathematics Subject Classification: 16T05,  20G42, 46L87\end{quote}
\section{\bf Introduction}

Bialgebroids and Hopf algebroids are a natural generalisation of Hopf algebras or quantum groups in the same way that a groupoid generalises a group. The axioms for a (say, left) bialgebroid $\CL$ are by now quite stable \cite{Boehm} and include that $\CL$ is a bimodule over a base `coordinate algebra' $B$ with source and target algebra maps $B,B^{op}\to \CL$ and with $\CL$ a coalgebra in the category of $B$-bimodules. The point of view here is that of noncommutative geometry in that $B$ could be noncommutative or `quantum' and arrows get reversed compared to their classical counterparts in geometry. In mathematical physics, there would typically be a deformation parameter controlling the noncommutativity, but we do not suppose this. For a Hopf algebroid, there are various notions, and here we adopt the weakest one that a left bialgebroid is a left or anti-left Hopf algebroid if certain left or anti-left `Hopf-Galois' maps are bijective. There are also more restrictive notions of Hopf algebroid involving an `antipode' map $S:\CL\to \CL$ but without consensus as to the axioms (notably, assuming that $S$ is an antialgebra map as in  \cite{Boehm} is too restrictive to include important examples, but can be weakened \cite{HM22}).   

The aim of the present work is to move from the axioms of  bialgebroids and Hopf algebroids towards applications,  in our case introducing a notion of `bisections' $\sigma: \CL\to B$ in Definition~\ref{def. left bisection of Hopf algebroids}. We show in Theorem~\ref{thm. group of left bisections} that these form a group $\CB(\CL)$ when $\CL$ is a left Hopf algebroid. Classically, for groupoids, bisections are maps from the base to the total space of the groupoid with certain natural properties \cite{Mac} as recalled in Section~\ref{sec:preclass}; in our algebraic form, the maps now go the other way and the corresponding algebras are allowed to be noncommutative. A main motivation  comes from the theory of principal bundles $P\to M$ with structure Lie group $G$ over a base manifold $M$. The relevant groupoid is the  Ehresmann groupoid $(P\times P)/G$ providing a receptacle for parallel transport (every connection on the principal bundle induces a groupoid map from the path groupoid of $M$ to $(P\times P)/G$ by parallel transport along paths). Its  bisections encode bundle automorphisms of $P$ as a principal bundle and its `vertical bisections' (that are the identity on the base) encode gauge transformations\cite{Mac}. In noncommutative geometry,  `quantum' principal bundles a.k.a. Hopf-Galois extensions have been widely studied and consist (now in a `coordinate algebra' form) of an algebra $P$ on which a Hopf algebra $H$ coacts with fixed subalgebra the base $B$, see  \cite{BrzMa, mont, BM} among many other works. The Ehresmann-Schauenburg Hopf algebroid $\CL(P,H)$ in  \cite{schau} provides the noncommutative version of the Ehresmann groupoid and it is known  \cite{HL20} that the group $\Aut_H(P)$ of invertible $H$-colinear $P$-algebra maps can be expressed in terms this, albeit without a general theory of bisections of bialgebroids or Hopf algebroids as such. We provide the latter and show in Section~\ref{sec:ESbisec} that our notions indeed reduce correctly in this case.  Whereas quantum principal bundles or Hopf-Galois extensions have received a lot of attention, their application to `gauge theory' along the lines in mathematical physics requires first and foremost a better understanding of gauge transformations, which are now expressed in Hopf algebroid terms as vertical bisections. One can also expect that the Hopf algebroid of differential operators  \cite{Xu,Gho} can play a role akin to the path groupoid, hence providing a route to a notion of parallel transport for noncommutative geometry. This remains open, but behaving well with respect to gauge transformations should be a key property just as it is for parallel transport and monodromy in classical geometry. 

Another justification for our definition of $\CB(\CL)$ is that classically \cite[p.~14]{MeZhu} the group of bisections of a groupoid and its group of automorphisms form a group crossed module in the sense of Whitehead. The latter are a concrete way to think about 2-groups defined as 2-categories with one object and all morphisms and 2-morphisms invertible\cite{Beh}. We now find the same for any left Hopf algebroid, constructing the required group homomorphism $\CB(\CL)\to \Aut(\CL)$ in  Proposition~\ref{equ. left adjoint automorphism}.  Next, 
motivated again by the $\CL(P,H)$ case, we  introduce a notion of `extended bisections'  where one of the axioms of $\CB(\CL)$  is replaced by requiring $\sigma$ to be convolution-invertible. The vertical such extended bisections then reduce in the case of $\CL(P,H)$ to the more general notion of gauge transformation for quantum principal bundles in  \cite{Ma:dia}, where these are only required to be convolution-invertible. The issue here is that in noncommutative geometry, there can be too few algebra maps so that $\Aut_H(P)$ and its vertical part are often too small for a viable gauge theory, necessitating a more general notion. 

What we find at this point is a remarkable connection between bisections and 2-cocycles $Z^2(\CL,B)$ used in the theory of (co)twisting of Hopf algebras and Hopf algebroids. The latter goes back to Drinfeld \cite{Dri}, who noted that the property of being a  quasi-Hopf algebra was stable under conjugation of the coproduct by a 2-cochain. A corollary \cite{Ma:book} of Drinfeld's work was that the property of being a  Hopf algebra is similarly stable if the 2-cochain obeys a further `non-Abelian' 2-cocycle condition. There is also a notion of 1-cochains and their coboundary, leading to a non-Abelian cohomology for such twisting of Hopf algebras up to equivalence. Building on the theory of twisting and, dually, cotwisting, of bialgebroids and Hopf algebroids\cite{Boehm,HM22}, Section~\ref{seccoh} now uses extended bisections to similarly fill in this picture. For a left bialgebroid $\CL$, we interpret the group of extended vertical bisections as 1-cochains $C^1(\CL,B)$ and introduce a suitable coboundary such that the cocycles $Z^1(\CL,B)$ are exactly the monoid $\CB_{ver}(\CL)$ of vertical bisections. Moreover, any element of $C^1(\CL,B)$ transforms an element of $Z^2(\CL,B)$ to another, leading us in Theorem~\ref{thm: boundary map} to introduce a non-Abelian cohomology $\CH^2(\CL,B)$ as the equivalence classes under such transformations.  This result is unexpected because it puts into the same setting very different kinds of objects: bisections with a classical meaning as above (including gauge transformations) and (co)twisting coming out of quantum group theory. The equivalence of (co)module categories under (co)twist is a kind of `gauge transform' in some abstract sense and suggests that gauge transformations in classical geometry are part of a  higher notion of gauge theory that remains to be understood. We show in the case of $\CL(P,H)$ that this non-Abelian cohomology reduces to corresponding $\tilde C^1(H,P)$, $\tilde Z^1(H,P)$ and $\tilde Z^2(H,P)$ spaces of maps from tensor powers of $H$ to $P$. The 2-cocycle case is done under a simplifying assumption that $P$ is `$\ra$-commutative' with respect to some right action $\ra$ of $H$ on $P$ and  $\CH^2(\CL(P,H),B)$ then reduces to $\tilde \CH^2(H,P)$ in  Theorem~\ref{thm:ESZ2}.  Meanwhile, Proposition~\ref{prop. BEH} checks that $\CH^2(\CL(P,H),B)$ similarly reduces when $P$ is a trivial quantum principal bundle or cleft Hopf-Galois extension to a cohomology $\CH^2_{\la,as}(H,B)$ which governs these. The 2-cocycles $Z^2_{\la,as}(H,P)$ here were studied in  \cite{HM22} although the cohomology mainly studied there was slightly different as $\la$ was also allowed to be transformed. 

The remaining main results are in Section~\ref{secact}, where we see how  our general constructions play out for another important class of Hopf algebroids, which we refer to as {\em action Hopf algebroids}  $B\# H^{op}$. These were introduced by J-H. Lu in  \cite{Lu} and we will use them in an equivalent version with a left-bialgebroid built from a braided-commutative algebra $B$ in the category $\DYH$ of Drinfeld-Yetter modules of a Hopf algebra $H$. (This category is  also referred to as that of $H$-crossed modules \cite{Ma:book}, but we avoid such terminology in the present work due to conflict with group crossed modules.) An object here has both a right action $\ra$ and a right coaction, compatible in a certain way that, if $H$ is finite dimensional, amounts to the category of right modules of the Drinfeld quantum double $D(H)$ in  \cite{Dri}. That  $B\# H^{op}$ is a Hopf algebroid in a certain sense due to Lu was shown in \cite{BM} while for our purposes we check in Proposition~\ref{thm. left and anti-left Hopf algebroid on YD module} that it is a left and anti-left Hopf algebroid in the sense we need. This is in line with \cite{BS}, but now given more explicitly.  Propositions~\ref{thmZ} and~\ref{propZ} then show that the group of left bisections $\CB(B\#H^{op})$ can be identified concretely as a group $Z^1_\ra(H,B)$ of non-abelian 1-cocycles with respect to the action, 
\[ \varrho:H\to B,\quad \varrho(gh)=( \varrho(g)\ra h{}_{\scriptscriptstyle{(1)}})\varrho(h{}_{\scriptscriptstyle{(2)}})\]
for all $h,g\in H$, that are also invertible in a certain sense. Here, $h{}_{\scriptscriptstyle{(1)}}\tens h{}_{\scriptscriptstyle{(2)}}$ denotes the coproduct of $H$. Such 1-cocycles are of interest for any action of a Hopf algebra on an algebra but in general do not form a group under convolution when $B$ is not commutative. Our results imply that when $B$ is braided-commutative then they do form a group for a certain product. We also give a characterisation of $\Aut(B\# H^{op})$ to complete the crossed module structure. Finally, it is known that the action Hopf algebroid construction behaves well under a Drinfeld twist of both $H$ and $B$ \cite{BP} and Theorem~\ref{thm:actH2} shows that $\CH^2(B\# H^{op},B)$ reduces to a non-Abelian cohomology $H^2_\ra(H,B)$  for the cotwist version in  \cite{HM22}. 

An elementary but canonical example \cite{Lu} for finite-dimensional $H$  is the Weyl Hopf algebroid $H^*\# H^{op}$ or `Heisenberg double'. As an algebra, this is isomorphic \cite{Ma:book} to the linear maps ${\rm Lin}(H)$, while a similar construction in the classical case where $H^*$ is replaced by an algebraic group gives its Hopf algebroid of differential operators. Proposition~\ref{Zaut}  characterises $Z^1_\ra(H,H^*)$ in this case as the counital invertible elements of $H^*$, while Proposition~\ref{propweylZ2} characterises $Z^2_\ra(H,H^*)$. As a new (but related) source of examples, Section~\ref{secquasi} introduces an action Hopf algebroid associated to any coquasitriangular Hopf algebra $H$, where $B=\underline{H}$ is the right comodule transmutation of $H$ \cite{Ma:book}. The latter is a braided Hopf algebra in the category $\CM^H$ of $H$-comodules, which functorially embeds in $ \DYH$, but as such it is not braided-commutative. It was observed in  \cite{BegMa,Ma:hod}, however, that it {\em is} braided-commutative if we view it in $\DYH$ via a different action and the same coaction. We show that  $\underline H\# H^{op}$ in this case  is isomorphic to the Weyl Hopf algebroid if $H$ is finite-dimensional and factorisable in sense of  \cite{Dri,Ma:book}.   A concrete example of the construction is the Hopf algebroid  $B_q[SL_2]\# \CO_q[SL_2]^{op}$ for the standard quantum group $\CO_q[SL_2]$ and $B_q[SL_2]$  its transmutation as in  \cite{Ma:book}. 

The paper also includes right handed bisections $\CB^R(\CL)$ on left Hopf algebroids, where the theory is similar, and concludes in Section~\ref{secdual} with the dual versions of some of our main constructions, as these are also of interest.  The non-Abelian cohomology also has a $C^2(\CL,B)$ and $Z^3(\CL,B)$ level leading to a natural notion of coquasi-Hopf algebroids, omitted from an earlier version of the present work and to be presented elsewhere.  Another direction for further work could be quantum principal bundles or Hopf-Galois extensions in which the structure Hopf algebra is generalised to a Hopf algebroid. The cocycle actions of Hopf algebroids on other algebras in \cite{BB} could be viewed as cleft examples and the relevant 2-cocycles, while different from the present work, would be a possible point of contact.

\section{\bf Preliminaries}\label{secpre}

We start with some preliminaries on groupoids and Hopf algebroids, with emphasis in the former case on the role of bisections in gauge theory as a main motivation for the paper.

\subsection{Recap of groupoids and classical bisections} \label{sec:preclass}

We recall that a groupoid can be thought of concretely as $\CG=(\CG^0,\CG^1,s,t)$ where $\CG^0$ is the `base' and $\CG^1$ is the total (or `arrow') space, $s,t:\CG^1\to \CG^0$ are source and target maps and $\CG^1$ has a product between composable elements, i.e. where the target of the first element equals the source of the 2nd. We also require an identity section $\CG^0\to \CG^1$, an inversion operation on $\CG^1$, and some natural compatibilities between the various structures. Morphisms between groupoids are a pair of maps between the total and base spaces respecting these structures. More details are in \cite{Mac}. 

A \textit{bisection} on a classical groupoid $(\mathcal{G}^{0},  \mathcal{G}^{1})$ is a map $\sigma: \mathcal{G}^{0}\to \mathcal{G}^{1}$, such that
$s\circ \sigma=\id_{\mathcal{G}^{0}}$, and  $t\circ \sigma : \mathcal{G}^{0} \to \mathcal{G}^{0}$ is a diffeomorphism. The collection of  bisections,
denoted $\B(\mathcal{G})$, forms a group: given two  bisections $\sigma_1 $ and $\sigma_{2}$, their multiplication is defined by
\begin{align}
   ( \sigma_{1}\ast\sigma_{2})(x):=\sigma_{1} \big((t\circ \sigma_{2})(x)\big)\sigma_{2}(x), \qquad \mbox{for} \quad x\in \mathcal{G}^{0} .
\end{align}
The group identity is the identity section and the inverse is given by
\begin{align}
    \sigma^{-1}(x)= \Big(\sigma\big((t \circ\sigma)^{-1}(x)\big)\Big)^{-1}.
\end{align}
Here, $(t \circ\sigma)^{-1}$ as a diffeomorphism of $\mathcal{G}^{0}$. We refer to  \cite[Sec.~1.4]{Mac}  for more details. The subset $\B_{ver}(\CG)$ of `vertical' bisections, defined as those bisections where $t\circ \sigma=\id_{\CG^0}$, forms a normal subgroup of $\B(\CG)$.

One can see that $\B(\CG)$ and the set $\Aut(\CG)$ of groupoid automorphisms form a Whitehead group crossed module. We recall that this consists of  $(M, N, \mu, \la)$ for a pair of groups $M$, $N$, 
a group homomorphism $\mu: M\to N$ and a left group action $\la$ of $N$ on $M$ by group automorphisms such that
\[ \mu(n\la m)=n\mu(m) n^{-1},\quad \mu(m)\la (m')=mm'm^{-1},\]
for all  $n\in N, m, m'\in M$. A similar notion of crossed $G$-sets was introduced by Whitehead in  \cite{Whi} and such objects are of interest as concrete realisations of the abstract notion of a 2-group \cite{Beh}. In our case, the map $\mu: \B(\CG)\to \Aut(\CG)$ sends $\sigma$ to $\Ad_\sigma:g\mapsto \sigma(t(g))\,g\,(\sigma(s(g)))^{-1}$ for every $g\in \CG^{1}$. The action of an automorphism $\Phi$ on a bisection $\sigma$ is given by $\Phi\la \sigma(x)=\Phi\circ \sigma\circ \Phi^{-1}|_{\CG^{0}}(x)$, for $x\in \CG^{0}$. That bisections form a group crossed module was remarked in  \cite{MeZhu}.

In particular, given a classical principal bundle $P$ with structure group $G$ and base manifold $M$, the Ehresmann groupoid $\CG$ has $\CG^1=\frac{P\times P}{G}$, the set of equivalence classes in $P\times P$ under the diagonal $G$-action, and $\CG^0=M$. We denote these classes by $[p,q]$ for any $p,q\in P$. Then $\CB(\CG)$  is known  \cite{Mac} to be isomorphic to the group of  ($G$-equivariant) principal bundle automorphisms,
\begin{equation}
\mathrm{Aut}_{G}(P) :=\{\varphi:P\to P \, | \,\  \varphi (pg) = \varphi(p)g,\quad \forall p\in P,\ g\in G  \} \, ,
\end{equation}
 while $\B_{ver}(\CG)$ is isomorphic to
the subgroup of gauge transformations, i.e., principal bundle
automorphisms which are vertical (in the sense that they project to the identity on the base):
\begin{equation}\label{AVclass}
\mathrm{Aut}_{P/G}(P) :=\{\varphi:P\to P \, | \,\, \varphi (pg) = \varphi(p)g  \, , \, \, \pi( \varphi(p)) = \pi(p), \quad \forall p\in P,\ g\in G  \} .
\end{equation}

The first isomorphism here can be constructed as follows. Given $\varphi\in \mathrm{Aut}_{G}(P)$, we can define a left bisection $\sigma_{\varphi}(m)=[\varphi(p), p]$, where $p\in \pi^{-1}(m)$, and $m\in M$. Conversely, let $\sigma$ be a left bisection, $m$ an arbritary point on $M$ and $p\in\pi^{-1}(m)$. Writing $\sigma(m)=[q, p]$, we set $\varphi_{\sigma}(p)=q$. It is not hard to see that $\varphi_{\sigma}$ is $G$-equivariant and that these two constructions are inverse to each other.  In this context, for a given gauge transformation $\varphi$, the map $\mu$ for the crossed-module structure is $\mu(\sigma_{\varphi})=\Ad_{\sigma_{\varphi}}:[p,p']\mapsto [\varphi(p),\varphi(p')]$ for every $[p,p']\in \frac{P\times P}{G}$, as the resulting element of $\Aut(\CG)$. 

We will also be interested in action groupoids. Given a smooth left action of a Lie group $G$ on a manifold $M$, the associated action Lie groupoid has $\CG^1=G\times M$ and $\CG^0=M$. The source and target maps are 
given by $s(g, m)=m$, $t(g,m )=gm$, for any $m\in M$ and $g\in G$. The identity section is $m\mapsto(1, m)$ and the product is given by $(g, m)(h, n)=(gh, n)$, defined if and only if $m=hn$, where $g,h\in G$ and $m,n\in M$. The inverse is given by $(g, m)^{-1}=(g^{-1}, gm)$. Elements of $\CB(\CG)$ are given by functions $\varrho:M\to G$ such that $\psi:M\to M$ given by $\psi(m):=\varrho(m)m$ is invertible (and the identity for a vertical bisection). See   \cite{Mac} for more details.

\subsection{Recap of Hopf algebroids}

Let $B$ be a unital algebra. A {\em $B$-ring} is a triple $(A,\mu,\eta)$ where $A$ is a $B$-bimodule and 
$\mu:A\ot_ {B} A \to A$, $\eta:B\to A$ are $B$-bimodule maps satisfying the associativity and unit conditions
\begin{equation}
\mu\circ(\mu\ot _{B}  \id_A)=\mu\circ (\id_A \ot _{B} \mu),
\quad
\mu\circ(\eta \ot _{B} \id_A)=\id_A=\mu\circ (\id_A\ot _{B} \eta).
\end{equation}
A morphism of $B$-rings $f:(A,\mu,\eta)\to (A',\mu',\eta')$ is a
$B$-bimodule map $f:A \to A'$ such that
$f\circ \mu=\mu'\circ(f\ot_{B} f)$ and $f\circ \eta=\eta'$.

From  \cite[Lem.~2.2]{Boehm}, there is a bijective correspondence between $B$-rings $(A,\mu,\eta)$ and algebras $A$ equipped with algebra maps $\eta : B \to A$. Starting with a $B$-ring $(A,\mu,\eta)$, one obtains a multiplication map $A \ot A \to A$ by composing the canonical surjection $A \ot A \to A\ot_B A$ with the map $\mu$. Conversely, starting with an algebra and an algebra map $\eta : B \to A$, a $B$-bilinear associative multiplication $\mu:A\ot_ {B} A \to A$ is obtained from the universality of the coequaliser $A \ot A \to A\ot_B A$ which identifies an element $ a b \ot a'$ with $ a \ot b a'$. Unless specified otherwise, algebra maps between unital algebras are required to respect both the products and the units.

Dually,  given a unital algebra $B$, a {\em $B$-coring} is a
triple $(C,\Delta,\varepsilon)$. Here $C$ is a $B$-bimodule with $B$-bimodule maps
$\Delta:C\to C\ot_{B} C$ and $\varepsilon: C\to B$ satisfying the
coassociativity and counit conditions
\begin{align}
(\Delta\ot _{B} \id_C)\circ \Delta = (\id_C \ot _{B} \Delta)\circ \Delta, \quad
(\varepsilon \ot _{B} \id_C)\circ \Delta = \id_C =(\id_C \ot _{B} \varepsilon)\circ \Delta.
\end{align}
A morphism of $B$-corings $f:(C,\Delta,\varepsilon)\to
(C',\Delta',\varepsilon')$ is a $B$-bimodule map $f:C \to C'$, such that
$\Delta'\circ f=(f\ot_{B} f)\circ \Delta$ and
$\varepsilon' \circ f =\varepsilon$.

Now suppose that $s:B\to \cL$ and $t:B^{op}\to \cL$ are algebra maps with images that commute. Then $\eta(b\tens c)=s(b)t(c)$ is an algebra map $\eta: B^e\to \cL$, where $B^e=B\tens B^{op}$. Moreover, this data is equivalent to making $\cL$ a $B^e$-ring with left and right actions of $B^e$ given by $(b\tens c).X=s(b) t(c)X$ and $X.(b\tens c)=Xs(b)t(c)$ respectively, for all $b,c\in B$ and $X\in \cL$. The left $B^e$-action is equivalent to a $B$-bimodule structure
\begin{equation}\label{eq:rbgd.bimod}
b.X.c= s(b) t(c)X
\end{equation}
for all $b,c\in B$ and $X\in \cL$. With respect to this bimodule structure, we have
\begin{align*}
     \cL\ot_{B} \cL=&\cL\ot \cL/\langle t(b)X\ot Y-X\ot s(b)Y\ |\  b\in B, X,Y\in \cL\rangle.
\end{align*}
One can check that the subset
\begin{equation*} \cL\times_{B} \cL :=\{\ \sum_i X_i \ot_{B} Y_i\ |\ \sum_i X_it(b) \ot_{B} Y_i=
\sum_i X_i \ot_{B}  Y_i  s(b),\quad \forall b\in B\ \}\subseteq \cL\tens_B\cL
\end{equation*}
 is a well-defined algebra via factor-wise multiplication, called the {\em Takeuchi product}.

\begin{defi}\label{def:right.bgd} Let $B$ be a unital algebra. A left $B$-bialgebroid is  first of all a unital algebra $\cL$ equipped with algebra maps  $s:B\to \cL$ and $t:B^{op}\to \cL$ such that $s(b)t(b')=t(b')s(b)$ for all $b,b'\in B$. Moreover, we require that $\cL$ is a $B$-coring with coproduct $\Delta$, counit $\varepsilon$ and $B$-bimodule structure  (\ref{eq:rbgd.bimod}) such that:
\begin{itemize}
\item[(i)]  $\Delta$ corestricts to an algebra map  $\cL\to \cL\times_B \cL$;
\item[(ii)]  $\varepsilon$ is a `left character' in the sense
\begin{equation*}\varepsilon(1_{\cL})=1_{B}, \quad \varepsilon(Xs(\varepsilon(Y)))=\varepsilon(XY)=\varepsilon(Xt (\varepsilon(Y)))\end{equation*}
for all $X,Y\in \cL$ and $b\in B$.
\end{itemize}
\end{defi}

We call $s:B\to \cL$ and $t:B^{op}\to \cL$ the source and target maps. They necessarily compose with $\varepsilon$ to the identity on $B, B^{op}$ respectively.

 \begin{defi}{\rm \cite{kornel}}. A morphism between a left $B$-bialgebroid $(\CL,s,t,\Delta,\eps)$ and a left $B'$-bialgebroid $(\CL',s',t',\Delta',\eps')$  is a pair $(\Phi, \varphi)$ of algebra maps  $\Phi: \cL \to \cL'$, $\varphi : B \to B'$ such that
\begin{align}
\Phi\circ s & = s'\circ \varphi , \qquad \Phi\circ t = t'\circ \varphi ,  \label{amoeba(i)} \\
(\Phi\ot_{B'} \Phi)\circ \Delta & = \Delta' \circ \Phi , \qquad
\varepsilon'\circ \Phi= \varphi\circ \varepsilon. \label{amoeba(ii)} 
 \end{align}   
It is an isomorphism when both maps in the pair are invertible. We denote by $\Aut(\CL)$ the isomorphisms of $\CL$ to itself. An automorphism from $\CL$ to itself is said to be {\em vertical} if $\varphi$ is the identity.
\end{defi}
Note that $\varphi$, if it exists, is uniquely determined from $\Phi$ by applying $\varepsilon'$ to the first of (\ref{amoeba(i)}), hence it is not additional data but its existence is a restriction on $\Phi$.  Automorphisms of a bialgebroid $\cL$
form a group $\Aut(\cL)$ by composition of $\Phi$, which implies composition of  $\varphi$ using the first of  (\ref{amoeba(i)}), and vertical ones form a subgroup. Since $\varphi$ is the auxiliary automorphism on the base, vertical automorphisms correspond geometrically to automorphisms of groupoids which preserves the base.

\begin{defi}\label{defHopf}{\rm \cite[Thm.-Def. 3.5]{schau1}.} 
A left bialgebroid $\cL$ is a {\em left Hopf algebroid} if
\[\lambda: \cL\ot_{B^{op}}\cL\to \cL\ot_{B}\cL,\quad
    \lambda(X\ot_{B^{op}} Y)=\one{X}\ot_{B}\two{X}Y\]
is invertible.  Similarly, a left bialgebroid $\cL$ is an {\em anti-left} Hopf algebroid if
\[\mu: \cL\ot^{B^{op}}\cL\to \cL\ot_{B}\cL,\quad
    \mu(X\ot^{B^{op}} Y)=\one{Y}X\ot_{B}\two{Y}\]
is invertible.  The balanced tensor products are given by 
\begin{align*}
\cL\ot_{B^{op}} \cL:=&\cL\ot \cL/\langle Xt(b)\ot Y-X\ot t(b)Y\ |\ b\in B, X,Y\in \cL\rangle,\\
    \cL\ot^{B^{op}} \cL:=&\cL\ot \cL/\langle s(b)X\ot Y-X\ot Ys(b)\ |\ b\in B, X,Y\in \cL\rangle.
\end{align*}
\end{defi}
If $B=k$ then this reduces to the map $\cL\tens\cL\to \cL\tens\cL$ given by $h\tens g\mapsto h\o\tens h\t g$ which for a usual Hopf algebra has inverse $h\tens g\mapsto h\o\tens (Sh\t)g$ if there is an antipode, where we used the Sweedler notation $\Delta h=h\o\tens h\t$ for the coproduct. We will also use such a notation for the coproduct of a bialgebroid. For an (anti-)left Hopf algebroid, we adopt the shorthand
\begin{equation}\label{X+-} X_{+}\ot_{B^{op}}X_{-}:=\lambda^{-1}(X\ot_{B}1),\end{equation}
\begin{equation}\label{X[+][-]} X_{[-]}\ot^{B^{op}}X_{[+]}:=\mu^{-1}(1\ot_{B}X).\end{equation}
We will also need some identities, and we recall from  \cite[Prop.~3.7]{schau1} that for a left Hopf algebroid,
\begin{align}
    \one{X_{+}}\ot_{B}\two{X_{+}}X_{-}&=X\ot_{B}1\label{equ. inverse lamda 1},\\
    \one{X}{}_{+}\ot_{B^{op}}\one{X}{}_{-}\two{X}&=X\ot_{B^{op}}1\label{equ. inverse lamda 2},\\
    (XY)_{+}\ot_{B^{op}}(XY)_{-}&=X_{+}Y_{+}\ot_{B^{op}}Y_{-}X_{-}\label{equ. inverse lamda 3},\\
    1_{+}\ot_{B^{op}}1_{-}&=1\ot_{B^{op}}1\label{equ. inverse lamda 4},\\
    \one{X_{+}}\ot_{B}\two{X_{+}}\ot_{B^{op}}X_{-}&=\one{X}\ot_{B}\two{X}{}_{+}\ot_{B^{op}}\two{X}{}_{-}\label{equ. inverse lamda 5},\\
    X_{+}\ot_{B^{op}}\one{X_{-}}\ot_{B}\two{X_{-}}&=X_{++}\ot_{B^{op}}X_{-}\ot_{B}X_{+-}\label{equ. inverse lamda 6},\\
    X&=X_{+}t(\eps(X_{-}))\label{equ. inverse lamda 7},\\
    X_{+}X_{-}&=s(\eps(X))\label{equ. inverse lamda 8},\\
    t(a)X_{+}\ot_{B^{op}}X_{-}&=X_{+}\ot_{B^{op}}X_{-}t(a)\label{equ. source and target map with lambda inv 5},\\
    (s(a)t(b)Xs(c)t(d))_{+}\ot_{B^{op}}(s(a)t(b)Xs(c)t(d))_{-}&=s(a)X_{+}s(c)\ot_{B^{op}}s(d)X_{-}s(b)\label{equ. source and target map with lambda inv 1},
\end{align}
for all $X, Y\in \cL$ and $a,b,c,d\in B$. Similarly, for an anti-left Hopf algebroid, we have the useful identities
\begin{align}
    \one{X_{[+]}}X_{[-]}\ot_{B}\two{X_{[+]}}&=1\ot_{B}X\label{equ. inverse mu 1},\\
    \two{X}{}_{[-]}\one{X}\ot^{B^{op}}\two{X}{}_{[+]}&=1\ot^{B^{op}}X\label{equ. inverse mu 2},\\
    (XY)_{[-]}\ot^{B^{op}}(XY)_{[+]}&=Y_{[-]}X_{[-]}\ot^{B^{op}}X_{[+]}Y_{[+]}\label{equ. inverse mu 3},\\
    1_{[-]}\ot^{B^{op}}1_{[+]}&=1\ot^{B^{op}}1\label{equ. inverse mu 4},\\
    X_{[-]}\ot^{B^{op}}\one{X_{[+]}}\ot_{B}\two{X_{[+]}}&=\one{X}{}_{[-]}\ot^{B^{op}}\one{X}{}_{[+]}\ot_{B}\two{X}\label{equ. inverse mu 5},\\
    (\one{X_{[-]}}\ot_{B}\two{X_{[-]}})\ot^{B^{op}}X_{[+]}&=(X_{[+][-]}\ot_{B}X_{[-]})\ot^{B^{op}}X_{[+][+]}\label{equ. inverse mu 6},\\
    X&=X_{[+]}s(\varepsilon(X_{[-]}))\label{equ. inverse mu 7},\\
    X_{[+]}X_{[-]}&=t(\varepsilon(X))\label{equ. inverse mu 8},\\
    (s(a)t(b)Xs(c)t(d))_{[-]}\ot^{B^{op}}(s(a)t(b)Xs(c)t(d))_{[+]}&=t(c)X_{[-]}t(a)\ot^{B^{op}} t(b)X_{[+]}t(d)\label{equ. inverse mu 9},\\
     X_{[-]}\ot^{B^{op}} s(b)X_{[+]}&=X_{[-]}s(b)\ot^{B^{op}} X_{[+]}\label{equ. inverse mu 12}.
\end{align}

Given a left $B$-bialgebroid $\cL$,   there is an algebra structure on ${}_{B}\Hom_{B}(\cL\otimes_{B^e}\cL, B)$ with the `convolution' product and unit given by
\begin{align}\label{equ. convolution product and unit}
    (f\star g) (X, Y):=f(\one{X}, \one{Y})g(\two{X}, \two{Y}),\quad \tilde{\varepsilon}(X, Y)= \varepsilon(XY)
\end{align}
for all $X, Y\in \cL$ and $f, g\in {}_{B}\Hom_{B}(\cL\otimes_{B^e}\cL, B)$. Here, and in what follows, is it sometimes convenient to write the value of a map on a tensor product in an element-wise manner with a sum of such terms and  descent to the tensor product understood. The $B$-bimodule structure on $\cL\otimes_{B^e}\cL$ is just the left $B^e$ action (so $b.(X, Y). c=(s(b)t(c)X, Y)$  for all $b, c\in B$) as induced by the algebra map $\eta: B^e\to \cL$.

\begin{defi} \rm\cite{HM22,Boehm, BB}. \label{Lcotwist}
Let $\CL$ be a left $B$-bialgebroid.  An \textup{invertible normalised 2-cocycle} on
$\cL$ is a convolution invertible element $\Gamma\in {}_{B}\Hom_{B}(\cL\otimes_{B^e}\cL, B)$ such that
\begin{align*}(i)\quad &\Gamma(X, s(\Gamma(\one{Y}, \one{Z}))\two{Y}\two{Z})=\Gamma(s(\Gamma(\one{X}, \one{Y}))\two{X}\two{Y}, Z),\quad \Gamma(1_{\cL}, X)=\varepsilon(X)=\Gamma(X, 1_{\cL});\\
(ii)\quad &\Gamma(X, Ys(b))=\Gamma(X, Yt(b))\end{align*}
for all $X, Y, Z\in \cL$ and $b\in B$. The collection of such 2-cocycles will be denoted $Z^{2}(\cL, B)$.
\end{defi}

We call $\tilde{\varepsilon}$ given in (\ref{equ. convolution product and unit}) the trivial 2-cocycle. Note that (ii) implies the minimal condition in  \cite{HM22} that $\Gamma^{-1}$ should be a right-handed 2-cocycle in the sense
\begin{align}\label{equ. right handed 2-cocycle}
    &\Gamma^{-1}(X, t(\Gamma^{-1}(\two{Y}, \two{Z}))\one{Y}\one{Z})=\Gamma^{-1}(t(\Gamma^{-1}(\two{X}, \two{Y}))\one{X}\one{Y}, Z),\\ &\Gamma^{-1}(1_{\cL}, X)=\varepsilon(X)=\Gamma^{-1}(X, 1_{\cL})
\end{align}
and also implies that $\Gamma^{-1}$ obeys (ii).

It is explained in  \cite{Boehm} that given a left $B$-bialgebroid and an invertible normalised 2-cocycle,  there is a twisted product
\begin{align}
    X\cdot_{\Gamma} Y:=s(\Gamma(\one{X}, \one{Y}))t(\Gamma^{-1}(\three{X}, \three{Y}))\two{X}\two{Y},
\end{align}
on the underlying vector space of $\cL$.  Moreover, the twisted product with the original $B^e$-bimodule and $B$-coring structure on $\cL$ form a left bialgebroid, which we denote  $\cL^\Gamma$.
\begin{thm} \rm\cite{HM22}. \label{thm. 2 cocycle twist1}
If $\cL$ is a (anti-)left Hopf algebroid, then $\cL^\Gamma$ is a (anti-)left Hopf algebroid.
\end{thm}

  If $B=k$ the ground field then a left Hopf algebroid reduces to a usual Hopf algebra and $\Gamma$ reduces to a usual Drinfeld cotwist cocycle in the sense of  \cite{Ma:book} with condition (ii) automatic. We next recall notions needed for cocycle smash products, as for example in the survey \cite{mont}. 
\begin{defi}\label{def: assoc type} Let $H$ be a Hopf algebra and $B$ a left $H$-module algebra with action  $\la$. A 2-cocycle on $H$ with values in $B$ is a linear map $\gamma:H\tens H\to B$ such that {\rm \cite{mont}}
\begin{align*}
    &\gamma(h, 1)=\gamma(1, h)=\varepsilon(h)1,\quad (\one{h}\triangleright\gamma(\one{g},\one{f}))\gamma(\two{h}, \two{g}\two{f})=\gamma(\one{h}, \one{g})\gamma(\two{h}\two{g}, f)
    \end{align*}
for all $h,g,f\in H$. We define $Z^2_{\la,as}(H,B)$ as the set of convolution-invertible 2-cocycles of {\em associative type} {\rm\cite{HM22}} in the sense
\[\gamma(\one{h}, \one{g})((\two{h}\two{g})\triangleright b)=((\one{h}\one{g})\triangleright b) \gamma(\two{h}, \two{g})
\]
for all $b\in B$, $h, g\in H$.
\end{defi}

Usually, as in  \cite{mont}, one considers cocycles as a pair $(\la,\gamma)$, where $\la$ is not necessarily an action, but in the associative type case it becomes an action. In the present work, we consider the action fixed and focus on $\gamma$. In either case, there is a cocycle smash product algebra $B\#_\gamma H$  built on the tensor product vector space with product
\[ (a\#_{\gamma}h)(b\#_{\gamma}g)=a(h\o\la b)\gamma(h\t,g\o)\#_{\gamma}h\th g\t. \]
This is a cleft extension\cite{mont} or `trivial' quantum principal bundle. We will also have recourse to  $Z^2_{\ra, as}(H,B)$ defined similarly with
\begin{align*}
    &\gamma(h, 1)=\gamma(1, h)=\varepsilon(h)1,\quad (\gamma(\one{g},\one{f})\ra h\o)\gamma(\two{h}, \two{f}\two{g})=\gamma(\one{h}, \one{g})\gamma(\two{g}\two{h}, f),\\
    &\gamma(\one{h}, \one{g})(b\ra (\two{g}\two{h}))=(b\ra (\one{g}\one{h})) \gamma(\two{h}, \two{g})
\end{align*}
for all $b\in B$, $f, h, g\in H$.

\section{\bf Bisections on Hopf algebroids}\label{secbisec}

\subsection{Left bisections on left Hopf algebroids}

As a noncommutative geometry version or `quantisation' of the classical notion of a bisection of a groupoid in Section~\ref{sec:preclass}, we have the following definition of bisection on any left bialgebroid.

\begin{defi}\label{def. left bisection of Hopf algebroids}
A left bisection on a left bialgebroid $\cL$ over $B$ is a unital linear map $\sigma: \cL\to B$ such that
\begin{itemize}
    \item [(1)] $\sigma(t(b)X)=\sigma(X)b$;
    \item[(2)] $\phi:=\sigma\circ s\in \Aut(B)$;
    \item[(3)] $\sigma(XY)=\sigma(Xt(\sigma(Y)))=\sigma(Xs(\phi^{-1}(\sigma(Y))))$
\end{itemize}
for all $b\in B$ and $X,Y\in \CL$.
\end{defi}

Given two left bisections $\sigma$ and $\eta$, there is a product given by
\begin{align}
    (\sigma\ast\eta)(X):=\sigma(s(\eta(\one{X}))\two{X}).
\end{align}
This is well defined over the balanced tensor product since $\sigma$ is right $B$-linear and makes the collection of bisections into a monoid, which we denote $\CB(\CL)$. We call a left bisection $\sigma$  {\em vertical} if $\sigma\circ s=\id_{B}$. We also observe that for any left bisection $\sigma$, 
\begin{align}\label{eqn:leftmodsigm}
    \sigma(s(b)X)=\sigma(s(b)t\circ\sigma(X))=\sigma(s(b))\sigma(X)=\phi(b)\sigma(X)
\end{align}
for all $b\in B$, $X\in \CL$.

\begin{thm}\label{thm. group of left bisections}
Let $\cL$ be a left bialgebroid  over $B$. Then $\CB(\cL)$ is a monoid with product $*$ and identity  given by the counit of $\cL$. It is a group if $\CL$ is a left Hopf algebroid, with the inverse of a  left bisection $\sigma$  given by
\begin{align}
    \sigma^{-1}(X)=(\sigma\circ s)^{-1}\varepsilon(X_{+}t(\sigma(X_{-})))
\end{align}
for all $X\in \cL$.  The vertical left bisections $\CB_{ver}(\CL)$ form a submonoid or subgroup respectively, with product reducing to the opposite convolution product. 
\end{thm}
\begin{proof}
We first show that the stated product of two left bisections $\sigma$ and $\eta$ is another left bisection. Here,  $(\eta\ast\sigma)(t(b))=b$ and
\begin{align*}
   ( \eta\ast\sigma)(t(b)X)&=\eta(s(\sigma(\one{X}))t(b)\two{X})=(\eta\ast\sigma)(X)b.
\end{align*}
 Similarly, $(\eta\ast\sigma)(s(b))=\eta\circ s\circ\sigma\circ s (b)$ so $(\eta\ast\sigma)\circ s\in \Aut(B)$ and
\begin{align*}
  ( \eta\ast\sigma)(s(b)X)=&\eta(s(\sigma(s(b)\one{X}))\two{X})
   =\eta(s(\sigma(s(b))\sigma(\one{X}))\two{X})\\
   =&\eta\circ s\circ\sigma\circ s (b)(\eta\ast\sigma)(X)
   =(\eta\ast\sigma)\circ s (b)(\eta\ast\sigma)(X).
\end{align*}
Moreover,
\begin{align*}
    (\eta\ast\sigma)(XY)=&\eta(s(\sigma(\one{X}\one{Y}))\two{X}\two{Y})\\
    =&\eta(s(\sigma(\one{X}t(\sigma(\one{Y}))))\two{X}s\circ(\eta\circ s)^{-1}\circ\eta(\two{Y}))\\
    =&\eta(s(\sigma(\one{X}))\two{X}s\big(\sigma(\one{Y})(\eta\circ s)^{-1}\circ\eta(\two{Y})\big))\\
    =&\eta(s(\sigma(\one{X}))\two{X}s\circ(\eta\circ s)^{-1}\big(\eta\circ s\circ\sigma(\one{Y})\eta(\two{Y})\big))\\
    =&\eta(s(\sigma(\one{X}))\two{X}s\circ(\eta\circ s)^{-1}\big((\eta\ast\sigma)(Y)\big))\\
    =&\eta(s(\sigma(\one{X}))\two{X}t\big((\eta\ast\sigma)(Y)\big))\\
    =&(\eta\ast\sigma)(Xt\big((\eta\ast\sigma)(Y)\big)).
\end{align*}
Similarly, from the 5th step above,
\begin{align*}
    (\eta\ast\sigma)(XY)=&\eta(s(\sigma(\one{X}))\two{X}s\circ(\eta\circ s)^{-1}\big((\eta\ast\sigma)(Y)\big))\\
    =&\eta(s\Big(\sigma(\one{X}t\circ(\eta\circ s)^{-1}\big((\eta\ast\sigma)(Y)\big))\Big)\two{X})\\
    =&\eta(s\Big(\sigma(\one{X}s\circ(\sigma\circ s)^{-1}\circ(\eta\circ s)^{-1}\big((\eta\ast\sigma)(Y)\big))\Big)\two{X})\\
    &\eta(s\Big(\sigma(\one{X}s\circ((\eta\ast\sigma)\circ s)^{-1}\big((\eta\ast\sigma)(Y)\big))\Big)\two{X})\\
    =&(\eta\ast\sigma)(Xs\circ((\eta\ast\sigma)\circ s)^{-1} \big((\eta\ast\sigma)(Y)\big)).
\end{align*}
Next, for the group identity,
\begin{align*}
    (\sigma\ast\varepsilon)(X)&=\sigma(s(\varepsilon(\one{X}))\two{X})=\sigma(X),\\
    (\varepsilon\ast\sigma)(X)&=\varepsilon(s(\sigma(\one{X}))\two{X})=\sigma(\one{X})\varepsilon(\two{X})=\sigma(t(\varepsilon(\two{X}))\one{X})=\sigma(X).
\end{align*}
Once we know that the result is another bisection, the product can be iterated and is easily seen to be associative using left $B$-linearity and coassociativity of the coproduct of $\cL$. Moreover, if $\sigma$ is vertical then the product is $(\sigma\ast\tau)(X)=\tau(X\o)\sigma(X\t)$ so that, in particular, the product on $\CB_{ver}(\CL)$ reduces to the opposite of the convolution product.

Next, if $\CL$ is a left Hopf algebroid then   $\sigma^{-1}$ as stated is a left bisection since
\begin{align*}
    \sigma^{-1}(t(b)X)=&(\sigma\circ s)^{-1}\varepsilon(X_{+}t(\sigma(X_{-}s(b))))
    =(\sigma\circ s)^{-1}\varepsilon(X_{+}t(\sigma(X_{-}t\circ\sigma\circ s(b))))\\
    =&(\sigma\circ s)^{-1}\varepsilon(t\circ\sigma\circ s(b)X_{+}t(\sigma(X_{-}))
    =(\sigma\circ s)^{-1}\varepsilon(X_{+}t(\sigma(X_{-}))b=\sigma^{-1}(X)b,
\end{align*}
where the 3rd step uses (\ref{equ. source and target map with lambda inv 5}). We also have that $\sigma^{-1}(s(b)X)=(\sigma\circ s)^{-1}\circ\varepsilon(s(b)X_{+}t(\sigma(X_{-})))=(\sigma\circ s)^{-1}(b)\sigma(X)$. Moreover,
\begin{align*}
    \sigma^{-1}(Xt(\sigma^{-1}(Y)))=&(\sigma\circ s)^{-1}\varepsilon(X_{+}t(\sigma(s\big(\sigma^{-1}(Y)\big)X_{-})))\\
    =&(\sigma\circ s)^{-1}\varepsilon(X_{+}t((\sigma\circ s)\big(\sigma^{-1}(Y)\big)\sigma(X_{-})))\\
    =&(\sigma\circ s)^{-1}\varepsilon(X_{+}t\Big((\sigma\circ s)\circ(\sigma\circ s)^{-1}\circ\varepsilon\big(Y_{+}t(\sigma(Y_{-}))\big)\sigma(X_{-})\Big))\\
    =&(\sigma\circ s)^{-1}\varepsilon(X_{+}t\Big(\varepsilon\big(Y_{+}t(\sigma(Y_{-}))\big)\sigma(X_{-})\Big))\\
    =&(\sigma\circ s)^{-1}\varepsilon(X_{+}t\Big(\varepsilon\big(t(\sigma(X_{-}))Y_{+}t(\sigma(Y_{-}))\big)\Big))\\
    =&(\sigma\circ s)^{-1}\varepsilon(X_{+}t\Big(\varepsilon\big(Y_{+}t(\sigma(Y_{-}t(\sigma(X_{-}))))\big)\Big))\\
    =&(\sigma\circ s)^{-1}\varepsilon(X_{+}Y_{+}t(\sigma(Y_{-}t(\sigma(X_{-})))))\\
    =&(\sigma\circ s)^{-1}\varepsilon(X_{+}Y_{+}t(\sigma(Y_{-}X_{-})))\\
    =&(\sigma\circ s)^{-1}\varepsilon((XY)_{+}t(\sigma((XY)_{-})))
    =\sigma^{-1}(XY),
\end{align*}
where the 6th step uses (\ref{equ. source and target map with lambda inv 5}) and the 9th step uses (\ref{equ. inverse lamda 3}).
Similarly, we have,
\begin{align*}
  \sigma^{-1}(Xs\circ(\sigma\circ s)\circ\sigma^{-1}(Y))=&(\sigma\circ s)^{-1}\circ \varepsilon\Big(X_{+}s\circ\sigma\circ s\circ\sigma^{-1}(Y)t\circ\sigma(X_{-})\Big)\\
  =&(\sigma\circ s)^{-1}\circ \varepsilon\Big(X_{+}s\circ\sigma\circ s\circ(\sigma\circ s)^{-1}\circ\varepsilon\big(Y_{+}t(\sigma(Y_{-}))\big)s\circ\sigma(X_{-})\Big)\\
  =&(\sigma\circ s)^{-1}\circ \varepsilon\Big(X_{+}s\circ\varepsilon\big(Y_{+}t(\sigma(Y_{-}))\big)s\circ\sigma(X_{-})\Big)\\
  =&(\sigma\circ s)^{-1}\circ \varepsilon\Big(X_{+}s\circ\varepsilon\big(t\circ\sigma(X_{-})Y_{+}t(\sigma(Y_{-}))\big)\Big)\\
  =&(\sigma\circ s)^{-1}\circ \varepsilon\Big(X_{+}s\circ\varepsilon\big(Y_{+}t(\sigma(Y_{-}t\circ\sigma(X_{-})))\big)\Big)\\
  =&(\sigma\circ s)^{-1}\varepsilon(X_{+}Y_{+}t(\sigma(Y_{-}X_{-})))\\
    =&(\sigma\circ s)^{-1}\varepsilon((XY)_{+}t(\sigma((XY)_{-})))
    =\sigma^{-1}(XY),
\end{align*}
where the 5th step uses (\ref{equ. source and target map with lambda inv 5}) and the 7th step uses (\ref{equ. inverse lamda 3}).
Finally, 
\begin{align*}
    (\sigma^{-1}\ast\sigma)(X)=&(\sigma\circ s)^{-1}\circ\varepsilon(s(\sigma(\one{X}))\two{X}{}_{+}t(\sigma(\two{X}{}_{-})))\\
    =&(\sigma\circ s)^{-1}\circ\varepsilon(s(\sigma(\one{X_{+}}))\two{X_{+}}t(\sigma(X_{-})))\\
    =&(\sigma\circ s)^{-1}\circ\varepsilon(s(\sigma(\one{X_{+}}))\two{X_{+}}s(\sigma(X_{-})))\\
    =&(\sigma\circ s)^{-1}\circ\varepsilon(s\big(\sigma(\one{X_{+}}t(\sigma(X_{-})))\big)\two{X_{+}})\\
    =&(\sigma\circ s)^{-1}\sigma\big(\one{X_{+}}t(\sigma(X_{-}))\big)(\sigma\circ s)^{-1}\circ\varepsilon\big(\two{X_{+}}\big)\\
    =&(\sigma\circ s)^{-1}\sigma\big(t\circ\varepsilon\big(\two{X_{+}})\one{X_{+}}t(\sigma(X_{-}))\big)\\
    =&(\sigma\circ s)^{-1}\sigma\big(X_{+}t(\sigma(X_{-}))\big)\\
    =&(\sigma\circ s)^{-1}\sigma\big(X_{+}X_{-}\big)\\
    =&(\sigma\circ s)^{-1}\sigma\big(s\circ\varepsilon(X)\big)=\varepsilon(X),
\end{align*}
where the 2nd step uses (\ref{equ. inverse lamda 5}) and the 9th step uses (\ref{equ. inverse lamda 8}). On the other side,
\begin{align*}
    (\sigma\ast\sigma^{-1})(X)=&\sigma(s\circ\sigma^{-1}(\one{X})\two{X})\\
    =&\sigma(s\circ(\sigma\circ s)^{-1}\varepsilon\big(\one{X}{}_{+}t(\sigma(\one{X}{}_{-}))\big)\two{X})\\
    =&\varepsilon\big(\one{X}{}_{+}t(\sigma(\one{X}{}_{-}))\big)\sigma(\two{X})\\
    =&\varepsilon\big(t\circ\sigma(\two{X})\one{X}{}_{+}t(\sigma(\one{X}{}_{-}))\big)\\
    =&\varepsilon\big(\one{X}{}_{+}t(\sigma(\one{X}{}_{-}t\circ\sigma(\two{X})))\big)\\
    =&\varepsilon\big(\one{X}{}_{+}t(\sigma(\one{X}{}_{-}\two{X}))\big)=\varepsilon(X),
\end{align*}
where the 5th step uses (\ref{equ. source and target map with lambda inv 5}) and the last step uses (\ref{equ. inverse lamda 2}). \end{proof}

We will also be interested in a less restrictive version of vertical left bisections, which we call `extended'.

\begin{defi}\label{defextver} Given a left bialgebroid $\CL$, the set $\CB_{ext.ver}(\CL)$ of extended vertical left bisections is defined as for $\CB_{ver}(\CL)$ but with the first equality of axiom (3) in Definition~\ref{def. left bisection of Hopf algebroids} replaced by convolution-invertibility. The remainder of axiom (3) in the vertical case reduces to $\sigma(X t(b))=\sigma(X s(b))$ for all $X\in \CL, b\in B$. 
\end{defi}

\begin{prop}\label{prop.left.extended} $\CB_{ext. ver}(\CL)$ for a left bialgebroid $\CL$ is a group with the opposite convolution product. If $\CL$ is a left Hopf algebroid then $\CB_{ver}(\CL) \subseteq\CB_{ext. ver}(\CL)$ is a subgroup.
\end{prop}
\proof (1) For any $\sigma,\tau\in \CB_{ver}\subseteq\CB_{ext. ver}(\CL)$, that the first axiom of Definition~\ref{def. left bisection of Hopf algebroids} is closed under $\ast$ proceeds as for the theorem, the second axiom is automatic as we limit ourselves to the vertical case, convolution invertibility is automatically closed under $\ast$, and it remains to check that 
\begin{align*}
(\tau\ast \sigma)(Xt(b))=&\sigma(\one{X})\tau(\two{X}t(b))=   \sigma(\one{X})\tau(\two{X}s(b)) =\sigma(\one{X}t(b))\tau(\two{X})\\
=&\sigma(\one{X}s(b))\tau(\two{X})=(\tau\ast \sigma)(Xs(b)),
\end{align*}
where the 2nd step uses the corresponding identity for $\tau$ and the 4th for $\sigma$. The unit of this group is just the counit of $\cL$. The convolution inverse $\sigma^{-1}$ also satisfies remainder of axiom (3), 
\begin{align*}
    \sigma^{-1}(Xt(b))=&\sigma^{-1}(\one{X}t(b))\sigma(\two{X})\sigma^{-1}(\three{X})\\
    =&\sigma^{-1}(\one{X})\sigma(\two{X}s(b))\sigma^{-1}(\three{X})\\
    =&\sigma^{-1}(\one{X})\sigma(\two{X}t(b))\sigma^{-1}(\three{X})\\
    =&\sigma^{-1}(\one{X})\sigma(\two{X})\sigma^{-1}(\three{X}s(b))=\sigma^{-1}(Xs(b))
\end{align*}
for all $X\in \CL, b\in B$. 

(2) That $\CB_{ver}(\CL)$ is a subgroup holds by Theorem~\ref{thm. group of left bisections} as this shows
that its elements are convolution-invertible.
\endproof

\subsection{Crossed module  associated to the group of left bisections}

We have explained the group crossed module (or equivalently, 2-group) associated to bisections in the case of a classical groupoid in Section~\ref{sec:preclass}. Here, we show that we continue to have a group crossed module in a parallel way even for a Hopf algebroid. This justifies our axioms of a bisection at the Hopf algebroid level and provides a bigger class of examples of group crossed modules as now arising even when the base and total space coordinate algebras are  noncommutative.

\begin{prop}\label{equ. left adjoint automorphism} For every  left bisection $\sigma$ of a left Hopf algebroid $\cL$ over $B$, there is an associated left Hopf algebroid automorphism $(\Ad^L_\sigma:\CL\to \CL, \ad^L_\sigma:B\to B)$ defined by
\begin{align*}
    \Ad^{L}_{\sigma}(X):=s(\sigma(\one{X}))\two{X}{}_{+}t(\sigma(\two{X}{}_{-})),\quad \ad^{L}_{\sigma}:=\sigma\circ s.
\end{align*}
Moreover, an automorphism $(\Phi, \phi)$ of a Hopf algebroid $\CL$ acts on a left bisection $\sigma$ by
\begin{align}\label{tautaction}
    (\Phi,\phi)\triangleright\sigma:=\phi\circ\sigma\circ\Phi^{-1}.
\end{align}
The map
$\Ad^{L}: \CB(\cL)\to\Aut(\cL)$ and the  action $\triangleright$ of $\Aut(\cL)$ on $\CB(\cL)$ make $\big( \CB(\cL), \Aut(\cL) \big)$ into a group crossed module.
\end{prop}
\proof
It is straight forward to check that $\Ad^L_\sigma$ is well defined. Moreover, $\Ad^{L}_{\sigma}\circ s=\ad^{L}_{\sigma}\circ s$ and $\Ad^{L}_{\sigma}\circ t=\ad^{L}_{\sigma}\circ t$, so this is a $B$-bimodule map. One can also check that $\varepsilon\circ \Ad^{L}_{\sigma}=\ad^{L}_{\sigma}\circ \varepsilon$. Next, we show that $\Ad^{L}_{\sigma}$ preserves the product and coproduct,
\begin{align*}
    \Ad^{L}_{\sigma}(XY)=&s(\sigma(\one{X}\one{Y}))\two{X}{}_{+}\two{Y}{}_{+}t(\sigma(\two{Y}{}_{-}\two{X}{}_{-}))\\
    =&s(\sigma(\one{X}t\circ\sigma(\one{Y})))\two{X}{}_{+}\two{Y}{}_{+}t(\sigma(\two{Y}{}_{-}t\circ\sigma(\two{X}{}_{-})))\\
    =&s(\sigma(\one{X}))\two{X}{}_{+}s\circ\sigma(\one{Y})t\circ\sigma(\two{X}{}_{-})\two{Y}{}_{+}t(\sigma(\two{Y}{}_{-}))\\
    =&\Ad^{L}_{\sigma}(X)\Ad^{L}_{\sigma}(Y),
\end{align*}
where the 3rd step uses (\ref{equ. source and target map with lambda inv 5}), and
\begin{align*}
    \Ad^{L}_{\sigma}\ot_{B}\Ad^{L}_{\sigma}(\one{X}\ot_{B}\two{X})=&s(\sigma(\one{X}))\two{X}{}_{+}t(\sigma(\two{X}{}_{-}))\ot_{B}s(\sigma(\three{X}))\four{X}{}_{+}t(\sigma(\four{X}{}_{-}))\\
    =&s(\sigma(\one{X}))t(\sigma(\three{X}))\two{X}{}_{+}t(\sigma(\two{X}{}_{-}))\ot_{B}\four{X}{}_{+}t(\sigma(\four{X}{}_{-}))\\
    =&s(\sigma(\one{X}))\two{X}{}_{+}t(\sigma(\two{X}{}_{-}t(\sigma(\three{X}))))\ot_{B}\four{X}{}_{+}t(\sigma(\four{X}{}_{-}))\\
    =&s(\sigma(\one{X}))\two{X}{}_{+}t(\sigma(\two{X}{}_{-}\three{X}))\ot_{B}\four{X}{}_{+}t(\sigma(\four{X}{}_{-}))\\
    =&s(\sigma(\one{X}))\two{X}\ot_{B}\three{X}{}_{+}t(\sigma(\three{X}{}_{-}))\\
    =&s(\sigma(\one{X}))\one{\two{X}{}_{+}}\ot_{B}\two{\two{X}{}_{+}}t(\sigma(\two{X}{}_{-}))\\
    =&\Delta\circ \Ad^{L}_{\sigma}(X),
\end{align*}
where the 3rd step uses (\ref{equ. source and target map with lambda inv 5}), the 5th step uses (\ref{equ. inverse lamda 2}) and the 6th step uses (\ref{equ. inverse lamda 5}). Also, given two left bisections $\sigma$ and $\eta$, we have
\begin{align*}
    \Ad^{L}_{\sigma}\circ \Ad^{L}_{\eta}(X)=&s(\sigma(s(\eta(\one{X}))\one{\two{X}{}_{+}}))\two{\two{X}{}_{+}}{}_{+}t(\sigma(s(\eta(\two{X}{}_{-}))\two{\two{X}{}_{+}}{}_{-}))\\
    =&s(\sigma(s(\eta(\one{X}))\two{X}))\three{X}{}_{+}t(\sigma(s(\eta(\one{\three{X}{}_{-}}))\two{\three{X}{}_{-}})) =\Ad^{L}_{\sigma\ast\eta},
\end{align*}
where we use (\ref{equ. inverse lamda 5}) and (\ref{equ. inverse lamda 6}). Finally, one can check that $\la$ is a left action of $\Aut(\CL)$ on $\CB(\cL)$ by groups automorphisms, which we omit as it follows similar lines to a special case in  \cite{HL20}. We thus have a similar conclusion in our more general setup. \endproof

This generalises a result \cite[Thm.~7.1]{HL20} for the specific case of the Ehresmann-Schauenburg Hopf algebroid. 
There is also a right-handed version of bisections and of our results thus far. 

\begin{defi}\label{def. right bisection of Hopf algebroids}
A right bisection on a left bialgebroid  $\CL$ is a unital linear map $\sigma: \cL\to B$, such that
\begin{itemize}
    \item [(1)] $\sigma(s(b)X)=b\sigma(X)$;
    \item[(2)] $\psi:=\sigma\circ t\in \Aut(B)$;
    \item[(3)] $\sigma(XY)=\sigma(Xs(\sigma(Y)))=\sigma(Xt(\psi^{-1}(\sigma(Y))))$
\end{itemize}
for all $b\in B$ and $X,Y\in \CL$. \end{defi}

This time, if $\cL$ is an {\em anti-left} Hopf algebroid then the set of right bisections forms a group $\CB^R(\CL)$, with identity given by the counit, the product given by 
\begin{align}
    (\sigma\bullet\eta)(X):=\sigma(t\circ\eta(\two{X})\one{X}),
\end{align}
and the inverse of $\sigma$ given by
\begin{align}\label{equ. the inverse of right bisection}
    \sigma^{-1}(X)=(\sigma\circ t)^{-1}\varepsilon(X_{[+]}s(\sigma(X_{[-]})))
\end{align}
for all $X\in \cL$.  The set $\CB^R_{ver}(\CL)$ of vertical right bisections where $\sigma\circ t=\id_{B}$   is the same set as $\CB_{ver}(\CL)$ with opposite monoid structure on comparing the definitions and noting (\ref{eqn:leftmodsigm}). 
If $\CL$ is a left or anti-left Hopf algebroid then one or other of these is a group and hence both are, with opposite group structures.   Similarly, extended vertical right bisections of a left bialgebroid are taken with the convolution product but are the same set as the left version, i.e. with the opposite group structure to the one in Proposition~\ref{prop.left.extended}. Note that being opposite groups still means that they are isomorphic as groups if we compose with the group inversion. 

 Finally, given a right bisection $\sigma$, we define an automorphism $(\Ad_\sigma^{R},\ad_\sigma^R)$ of an anti-left Hopf algebroid $\cL$ by
 \begin{align}\label{equ. right adjoint automorphism}
    \Ad^{R}_{\sigma}(X):=t(\sigma(\two{X}))\one{X}{}_{[+]}s(\sigma(\one{X}{}_{[-]})),\quad \ad^{R}_{\sigma}:=\sigma\circ t, 
\end{align}
and view this construction as a group homomorphism $\CB^R(\cL)\to \Aut(\CL)$.  As before, an automorphism $(\Phi, \phi)$ of $\cL$ acts tautologically on a right bisection $\sigma$ by
\begin{align}\label{equ. right bisection automorphism}
    (\Phi,\phi)\triangleright\sigma:=\phi\circ\sigma\circ\Phi^{-1}, 
\end{align}
which one can check is a left action of $\Aut(\CL)$ on $\CB^R(\cL)$ by group automorphisms. We then have that $\big( \CB^R(\cL), \Aut(\cL), \Ad^{R},\la \big)$ is also a group crossed module. The proofs are similar and we omit them.

\subsection{Non-Abelian second cohomology $\CH^2(\CL,B)$}\label{seccoh}

Let $\CL$ be a left $B$-bialgebroid. In this section, the same group $\CB_{ext. ver}(\CL)$ of extended vertical left bisections in Definition~\ref{defextver} will arise naturally from a very different point of view as the 1-cochains in a certain non-Abelian cohomology. Moreover,  in the left Hopf algebroid case, the original (non-extended) vertical ones $\CB_{ver}(\CL)$ will appear as 1-cocycles. Meanwhile, the 2-cocycles are  precisely the $Z^2(\CL,B)$ that control Drinfeld-type cotwists of Hopf algebroids in Definition~\ref{Lcotwist}. Thus, we bring together two notions that one might not have expected even to be related, namely bisections coming from geometry and cotwisting theory coming  from quantum groups.

\begin{defi}\label{remver}  Let $\CL$ be a left $B$-bialgebroid. We let $C^{1}(\cL, B)$ denote the set of convolution invertible elements $U\in {}_{B}\Hom_{B}(\cL, B)$ such that
\begin{align}\label{eq1}
  U(Xs(b))=U(Xt(b))
\end{align}
for all $X\in\cL$ and $b\in B$. This forms a group under convolution product coinciding with $\CB_{ext. ver}(\CL)^{op}$.
\end{defi}

The sets here coincide from their definition but the product in  Proposition~\ref{prop.left.extended} has the opposite to the convolution product (i.e., we can identify it more canonically with the group of extended vertical right bisections).

\begin{thm}\label{thm: boundary map}
Let $\CL$ be a left bialgebroid. 

(1) For any $U\in C^{1}(\cL, B)$, we obtain a 2-cocycle $\partial U\in Z^2(\CL,B)$ by
\begin{align}\label{boundary map}
    \partial U(X, Y):=U^{-1}(\one{X}t(U^{-1}(\one{Y})))U(\two{X}\two{Y})
\end{align}
for all $X,Y\in \CL$, where $U^{-1}$ is the inverse of $U$ under convolution product. Moreover, 
the subset $Z^1(\CL,B)$  of $U\in C^1(\CL,B)$ for which $\partial U$ is trivial forms a submonoid under convolution,  coinciding with $\CB_{ver}(\CL)^{op}$. 

(2)  If $\Gamma\in Z^2(\CL,B)$ and $U\in C^{1}(\cL, B)$ then
\begin{align}\label{equ of bialgebroid cocycle}
    \Gamma^U(X, Y)=U^{-1}(\one{X}t(U^{-1}(\one{Y})))\Gamma(\two{X}, \two{Y})U(\three{X}\three{Y})
\end{align}
for all $X,Y\in \CL$ is another 2-cocycle in $Z^2(\CL,B)$. The set of orbits in $Z^2(\CL,B)$ under such an action is denoted by $\H^{2}(\cL, B)$. 
\end{thm}
\begin{proof}
(1) For the coboundary case, we first check that $\partial U$ is $B$-bilinear,
\begin{align*}
    \partial U(s(b)t(b')X, Y)=&U^{-1}(s(b)\one{X}t(U^{-1}(\one{Y})))U(t(b')\two{X}\two{Y})\\
    =&b U^{-1}(\one{X}t(U^{-1}(\one{Y})))U(\two{X}\two{Y})b'.
\end{align*}
It is also well defined over the balanced tensor product $\otimes_{B^{e}}$,
\begin{align*}
    \partial U(Xs(b)t(b'), Y)=&U^{-1}(\one{X}s(b)t(U^{-1}(\one{Y})))U(\two{X}t(b')\two{Y})\\
    =&U^{-1}(\one{X}s(b)s(U^{-1}(\one{Y})))U(\two{X}t(b')\two{Y})\\
    =&U^{-1}(\one{X}s(bU^{-1}(\one{Y})))U(\two{X}t(b')\two{Y})\\
    =&U^{-1}(\one{X}s(U^{-1}(s(b)\one{Y})))U(\two{X}t(b')\two{Y})\\
    =&\partial U(X, s(b)t(b')Y),
\end{align*}
where the 2nd step uses (\ref{eq1}). Next, we show that $\partial U$ is invertible with  inverse given by
\begin{align}\label{inverse of boundard map}
    (\partial U)^{-1}(X, Y)=U^{-1}(\one{X}\one{Y})U(\two{X}s(U(\two{Y}))).
\end{align}
On the one hand,
\begin{align*}
    (\partial U)\ast  (\partial U)^{-1} (X, Y)=&U^{-1}(\one{X}t(U^{-1}(\one{Y})))U(\two{X}\two{Y})U^{-1}(\three{X}\three{Y})U(\four{X}s(U(\four{Y})))\\
    =&U^{-1}(\one{X}t(U^{-1}(\one{Y})))\varepsilon(\two{X}\two{Y})U(\three{X}s(U(\three{Y})))\\
    =&U^{-1}(\one{X}t(U^{-1}(\one{Y})))\varepsilon(\two{X}t(\varepsilon(\two{Y})))U(\three{X}s(U(\three{Y})))\\
    =&U^{-1}(\one{X}t(U^{-1}(\one{Y})))\varepsilon(\two{X})U(\three{X}s(\varepsilon(\two{Y}))s(U(\three{Y})))\\
    =&U^{-1}(\one{X}t(U^{-1}(\one{Y})))\varepsilon(\two{X})U(\three{X}s(U(\two{Y})))\\
    =&U^{-1}(\one{X}t(U^{-1}(\one{Y})))U(s(\varepsilon(\two{X}))\three{X}s(U(\two{Y})))\\
    =&U^{-1}(\one{X}t(U^{-1}(\one{Y})))U(\two{X}s(U(\two{Y})))\\
    =&U^{-1}(\one{X})U(\two{X}s(U^{-1}(\one{Y}))s(U(\two{Y})))\\
    =&U^{-1}(\one{X})U(\two{X}s(\varepsilon(Y)))\\
    =&U^{-1}(\one{X})U(\two{X}t(\varepsilon(Y)))=\varepsilon(Xt(\varepsilon(Y)))=\varepsilon(XY),
\end{align*}
where the 10th step uses (\ref{eq1}). On the other hand,
\begin{align*}
    (\partial U)^{-1}\ast  (\partial U)(X, Y)=&U^{-1}(\one{X}\one{Y})U(\two{X}s(U(\two{Y})))U^{-1}(\three{X}t(U^{-1}(\three{Y})))U(\four{X}\four{Y})\\
    =&U^{-1}(\one{X}t(U(\two{Y}))\one{Y})U(\two{X})U^{-1}(\three{X})U(\four{X}s(U^{-1}(\three{Y}))\four{Y})\\
    =&U^{-1}(\one{X}t(U(\two{Y}))\one{Y})U(\two{X}s(U^{-1}(\three{Y}))\four{Y})\\
    =&U^{-1}(\one{X}\one{Y})U(\two{X}s(U(\two{Y}))s(U^{-1}(\three{Y}))\four{Y})\\
    =&U^{-1}(\one{X}\one{Y})U(\two{X}\two{Y}) =\varepsilon(XY)
\end{align*}
as required. Moreover, $(\partial U)^{-1}$ is clearly $B$-bilinear and is well defined over the balanced tensor product,
\begin{align*}
    (\partial U)^{-1}(Xs(b)t(b'), Y)=&U^{-1}(\one{X}s(b)\one{Y})U(\two{X}t(b')s(U(\two{Y})))\\
    =&U^{-1}(\one{X}s(b)\one{Y})U(\two{X}t(b')t(U(\two{Y})))\\
    =&U^{-1}(\one{X}s(b)\one{Y})U(\two{X}t(U(\two{Y})b'))\\
    =&U^{-1}(\one{X}s(b)\one{Y})U(\two{X}t(U(t(b')\two{Y}))).
\end{align*}

Finally, we check that $\partial U$ satisfies the 2-cocycle condition. On the one hand,
\begin{align*}
    \partial U(X, s(\partial &U(\one{Y}, \one{Z}))\two{Y}\two{Z})\\
   =&\partial U(X, s\big(U^{-1}(\one{Y}t(U^{-1}(\one{Z})))U(\two{Y}\two{Z})\big)\three{Y}\three{Z})\\
   =&U^{-1}(\one{X}t\bigg(U^{-1}\Big(s\big(U^{-1}(\one{Y}t(U^{-1}(\one{Z})))U(\two{Y}\two{Z})\big)\three{Y}\three{Z}\Big)\bigg))U(\two{X}\four{Y}\four{Z})\\
   =&U^{-1}(\one{X}t\Big(\big(U^{-1}(\one{Y}t(U^{-1}(\one{Z})))U(\two{Y}\two{Z})\big)U^{-1}\big(\three{Y}\three{Z}\big)\Big))U(\two{X}\four{Y}\four{Z})\\
   =&U^{-1}(\one{X}t\Big(U^{-1}(\one{Y}t(U^{-1}(\one{Z})))\varepsilon(\two{Y}\two{Z})\Big))U(\two{X}\three{Y}\three{Z})\\
   =&U^{-1}(\one{X}t(\varepsilon(\two{Y}\two{Z}))t\Big(U^{-1}(\one{Y}t(U^{-1}(\one{Z})))\Big))U(\two{X}\three{Y}\three{Z})\\
   =&U^{-1}(\one{X}t\Big(U^{-1}(\one{Y}t(U^{-1}(\one{Z})))\Big))U(\two{X}s(\varepsilon(\two{Y}\two{Z}))\three{Y}\three{Z})\\
   =&U^{-1}(\one{X}t(U^{-1}(\one{Y}t(U^{-1}(\one{Z})))))U(\two{X}\two{Y}\two{Z}).
\end{align*}
On the other hand,
\begin{align*}
    \partial U(s(\partial &U(\one{X}, \one{Y}))\two{X}\two{Y}, Z)\\
    =&\partial U(\one{X}, \one{Y})) \partial U(\two{X}\two{Y}, Z)\\
    =&U^{-1}(\one{X}t(U^{-1}(\one{Y})))U(\two{X}\two{Y})U^{-1}(\three{X}\three{Y}t(U^{-1}(\one{z})))U(\four{X}\four{Y}\two{Z})\\
    =&U^{-1}(\one{X}t(U^{-1}(\one{Y})))U(\two{X}\two{Y})U^{-1}(\three{X}\three{Y})U(\four{X}\four{Y}s(U^{-1}(\one{z}))\two{Z})\\
    =&U^{-1}(\one{X}t(U^{-1}(\one{Y})))\varepsilon(\two{X}\two{Y})U(\three{X}\three{Y}s(U^{-1}(\one{Z}))\two{Z})\\
    =&U^{-1}(\one{X}t(U^{-1}(\one{Y})))U(s(\varepsilon(\two{X}\two{Y}))\three{X}\three{Y}s(U^{-1}(\one{Z}))\two{Z})\\
    =&U^{-1}(\one{X}t(U^{-1}(\one{Y})))U(\two{X}\two{Y}s(U^{-1}(\one{Z}))\two{Z}),
\end{align*}
which simplifies to the same expression. One can also check that $\partial U(X, Ys(b))=\partial U(X, Yt(b))$ and $\partial U^{-1}(X, Ys(b))=\partial U^{-1}(X, Yt(b))$. 

Moreover, if $U$ is a vertical bisection then $\partial U(X, Y)=\varepsilon(XY)$ as $U^{-1}$ satisfies (3) of Definition \ref{def. left bisection of Hopf algebroids}. Conversely,  if $\partial U(X,Y)=\varepsilon(XY)$ then we have
\begin{align*}
    U(X\,s(U(Y)))=& U(X\o\,s(U(Y\o)))U^{-1}(X\t\,t(U^{-1}(Y\t)))U(X\th\, Y\th)\\
    =& U(X\o\,t(U(Y\o)))U^{-1}(X\t\,s(U^{-1}(Y\t)))U(X\th\, Y\th)\\
    =& U(X\o)U^{-1}(X\t\,s(U(Y\o))\,s(U^{-1}(Y\t)))U(X\th\, Y\th)\\
    =& \varepsilon(X\o\,Y\o)U(X\t\, Y\t)\\
    =&U(XY).
\end{align*}
Hence these are the same set, and we know that $\CB_{ver}(\CL)$ is a monoid with the opposite convolution product.

(2)  For the general case, the proof is similar. We first check that (\ref{equ of bialgebroid cocycle}) is well defined. Indeed, by the same method as in part (1), the right hand side is $B$-bilinear and well defined over the balanced tensor product. We can also see that it is invertible with inverse given by
\begin{align}
    (\Gamma^U{})^{-1}(X, Y)=U^{-1}(\one{X}\one{Y})\Gamma^{-1}(\two{X}, \two{Y})U(\three{X}s(U(\three{Y}))).
\end{align}
Thus,
\begin{align*}
    (\Gamma^U{})^{-1}\ast\Gamma^U{}(X, Y)=&U^{-1}(\one{X}\one{Y})\Gamma^{-1}(\two{X}, \two{Y})U(\three{X}s(U(\three{Y}))))\\
    &U^{-1}(\four{X}t(U^{-1}(\four{Y})))\Gamma(\five{X}, \five{Y})U(\six{X}\six{Y})\\
    =&U^{-1}(\one{X}\one{Y})\Gamma^{-1}(\two{X}t(U(\three{Y})), \two{Y})U(\three{X})\\
    &U^{-1}(\four{X})\Gamma(\five{X}s(U^{-1}(\four{Y})), \five{Y})U(\six{X}\six{Y})\\
    =&U^{-1}(\one{X}\one{Y})\Gamma^{-1}(\two{X}t(U(\three{Y})), \two{Y})\Gamma(\three{X}s(U^{-1}(\four{Y})), \five{Y})U(\four{X}\six{Y})\\
    =&U^{-1}(\one{X}\one{Y})\Gamma^{-1}(\two{X}, \two{Y})\Gamma(\three{X}s(U(\three{Y}))s(U^{-1}(\four{Y})), \five{Y})U(\four{X}\six{Y})\\
    =&U^{-1}(\one{X}\one{Y})\Gamma^{-1}(\two{X}, \two{Y})\Gamma(\three{X}, \three{Y})U(\four{X}\four{Y})=\varepsilon(XY).
\end{align*}
One can similarly check that $(\Gamma^U{})^{-1}\ast\Gamma^U{}(X, Y)=\varepsilon(X Y)$.  Checking the 2-cocycle condition is slightly different. On the one hand,
\begin{align*}
    \Gamma^U{}(X, &s(\Gamma^U{}(\one{Y}, \one{Z}))\two{Y}\two{Z})\\
    =&U^{-1}(\one{X}t\bigg(U^{-1}\Big(s\big(U^{-1}(\one{Y}t(U^{-1}(\one{Z})))\Gamma(\two{Y}, \two{Z})U(\three{Y}\three{Z})\big)\four{Y}\four{Z}\Big)\bigg))\\
    &\Gamma(\two{X}, \five{Y}\five{Z})U(\three{X}\six{Y}\six{Z})\\
    =&U^{-1}(\one{X}t\bigg(\big(U^{-1}(\one{Y}t(U^{-1}(\one{Z})))\Gamma(\two{Y}, \two{Z})U(\three{Y}\three{Z})\big)U^{-1}\Big(\four{Y}\four{Z}\Big)\bigg))\\
    &\Gamma(\two{X}, \five{Y}\five{Z})U(\three{X}\six{Y}\six{Z})\\
    =&U^{-1}(\one{X}t\bigg(\big(U^{-1}(\one{Y}t(U^{-1}(\one{Z})))\Gamma(\two{Y}, \two{Z})\big)\bigg))\Gamma(\two{X}, \three{Y}\three{Z})U(\three{X}\four{Y}\four{Z})\\
    =&U^{-1}(\one{X}t(\Gamma(\two{Y}, \two{Z}))t\big((U^{-1}(\one{Y}t(U^{-1}(\one{Z})))\big))\Gamma(\two{X}, \three{Y}\three{Z})U(\three{X}\four{Y}\four{Z})\\
    =&U^{-1}(\one{X}t\big((U^{-1}(\one{Y}t(U^{-1}(\one{Z})))\big))\Gamma(\two{X}, s(\Gamma(\two{Y}, \two{Z}))\three{Y}\three{Z})U(\three{X}\four{Y}\four{Z}).
\end{align*}
On the other hand,
\begin{align*}
    \Gamma^U{}(s(\Gamma^U{}(&\one{X}, \one{Y}))\two{X}\two{Y}, Z)\\
    =&U^{-1}(\one{X}t(U^{-1}(\one{Y})))\Gamma(\two{X}, \two{Y}) U(\three{X}\three{Y})U^{-1}(\four{X}\four{Y}t(U^{-1}(\one{Z})))\\
    &\Gamma(\five{X}\five{Y}, \two{Z})U(\six{X}\six{Y}\three{Z})\\
    =&U^{-1}(\one{X}t(U^{-1}(\one{Y}t(U^{-1}(\one{Z})))))\Gamma(\two{X}, \two{Y}) U(\three{X}\three{Y})U^{-1}(\four{X}\four{Y})\\
    &\Gamma(\five{X}\five{Y}, \two{Z})U(\six{X}\six{Y}\three{Z})\\
    =&U^{-1}(\one{X}t\big(U^{-1}(\one{Y}t(U^{-1}(\one{Z})))\big))\Gamma(s(\Gamma(\two{X}, \two{Y}))\three{X}\three{Y}, \two{Z})U(\four{X}\four{Y}\three{Z}),
\end{align*}
where the second step uses $U(Xs(b))=U(Xt(b))$ and $\Gamma(X, Ys(b))=\Gamma(X, Yt(b))$ repeatedly. This equates to the expression for $\Gamma^U{}(X, s(\Gamma^U{}(\one{Y}, \one{Z}))\two{Y}\two{Z})$ using that $\Gamma$ is a 2-cocycle. One can similarly check the other part of the cocycle condition. 
 Finally, we check that we have a group action of $C^1(\CL, B)$. Thus, for any $U,V\in C^1(\CL, B)$, 
 \begin{align*}
    &\Gamma^{U\ast V}(X,Y)\\
    &= V^{-1}(\one{X})U^{-1}(\two{X}\,t(V^{-1}(Y\o)\, U^{-1}(Y\t)))\Gamma(X\th\,Y\th)U(X\fo\,Y\fo)V(X\fiv\,Y\fiv)\\
    &=V^{-1}(\one{X})U^{-1}(\two{X}\, t(U^{-1}(Y\t))\,t(V^{-1}(Y\o)))\Gamma(X\th\,Y\th)U(X\fo\,Y\fo)V(X\fiv\,Y\fiv)\\
    &=V^{-1}(\one{X})U^{-1}(\two{X}\, t(U^{-1}(Y\t))\,s(V^{-1}(Y\o)))\Gamma(X\th\,Y\th)U(X\fo\,Y\fo)V(X\fiv\,Y\fiv)\\
    &=V^{-1}(\one{X}\,t(V^{-1}(Y\o)))U^{-1}(\two{X}\, t(U^{-1}(Y\t)))\Gamma(X\th\,Y\th)U(X\fo\,Y\fo)V(X\fiv\,Y\fiv)=(\Gamma^U){}^{V}(X,Y)
 \end{align*} for all $X,Y$ in $\CL$.
 \end{proof}
 
 The identification of $Z^1(\CL,B)$ in part (1) also gives this as identified with $\CB^R_{ver}(\CL)$ as a monoid. Hence it is group if $\CL$ is a left or anti-left Hopf algebroid by the discussion after Definition~\ref{def. right bisection of Hopf algebroids}. 
 
\begin{prop}\label{propisomGamma} Let $\CL$ be a  left bialgebroid and $U\in C^1(\CL,B)$. Then
\begin{align*} \Ad_U:\CL\to \CL^{\partial U},\quad
        \Ad_{U}(X)=s(U(\one{X}))t(U^{-1}(\three{X}))\two{X}
    \end{align*}
is an isomorphism of left bialgebroids and  is an automorphism of $\CL$ if $U\in Z^1(\CL,B)$.
 
 More generally, if $\Gamma\in Z^2(\CL,B)$ is equivalent to $\Gamma^U$  then the same map $\Ad_U$ gives an isomorphism $\Ad_{U}:\cL^{\Gamma}\to\cL^{\Gamma^U}$ of left bialgebroids. If $\CL$ is a  left  (resp. anti-left) Hopf algebroid then so are $\CL^\Gamma$ and $\CL^{\Gamma^U}$,  and $\Ad_U$ is an isomorphism as such. \end{prop}
\begin{proof} The first part is a special case of the general statement on setting $\Gamma$ trivial. For the latter, we have 
    \begin{align*}
        \Ad_{U}(X)&\cdot_{\Gamma^U{}}\Ad_{U}(Y)\\
        =&s\Big(U^{-1}\big(s(U(X\o))X\t s(U^{-1}(s(U(Y\o))Y\t ))\big)\Gamma(X\th  , Y\th  )U(X\fo  Y\fo  )\Big)\\
        &\quad t\Big(U^{-1}(X\si  Y\si  )\Gamma^{-1}(X\sev , Y\sev )U\big(t(U^{-1}(X\nin ))X\eig s(U(t(U^{-1}(Y\nin ))Y\eig ))\big)\Big)X\fiv Y\fiv \\
        =&s(\Gamma(X\o, Y\o)U(X\t Y\t ))t(U(X\fo  Y\fo  )\Gamma(X\fiv ,Y\fiv ))X\th  Y\th  =\Ad_{U}(X\cdot_{\Gamma}Y)
    \end{align*}
for all  $X, Y\in\cL$.    It is not hard to see that $\Ad_{U}$ is a $B$-bilinear map and also a coring map. The last part of the statement follows by   \cite[Thm. 3.7, Rem. 3.8]{HM22} that cotwisting preserves a left or anti-left Hopf algebroid structure.  
\end{proof}

\begin{cor} If $\CL$ is a left Hopf algebroid and $U\in Z^1(\CL,B)$  then $\Ad_{U}=\Ad^L_U$ is the crossed module map in Proposition~\ref{equ. left adjoint automorphism} when $U$ is viewed as a vertical left bisection. Similarly, if $\CL$ is an anti-left Hopf algebroid then $\Ad_U=\Ad_{ U^{-1}}^{R}$ when $U$ is viewed as a vertical right bisection. 
\end{cor}
\begin{proof} We have in the first case
    \begin{align*}
        \Ad_{U}(X)=&s(U(X\o))t(U^{-1}(X\th  ))X\t \\
        =&s(U(X\o))t(\varepsilon(X\th{}_{+}s(U(X\th{}_{-}))))X\t \\=&s(U(X\o))t(\varepsilon(X\t{}_{+}\t\,s(U(X\t{}_{-}))))X\t{}_{+}\o\\
        =&s(U(X\o))t(\varepsilon(X\t{}_{+}\t))X\t{}_{+}\o\,t(U(X\t{}_{-}))\\
        =&s(U(X\o))X\t{}_{+}\,t(U(X\t{}_{-}))=\Ad_{U}^{L}(X).
        \end{align*}
 In the second case, recalling (\ref{equ. right bisection automorphism}), we have
      \begin{align*}
          \Ad_{U^{-1}}(X)=&t(U(X\th  ))s(U^{-1}(X\o))X\t \\
          =&t(U(X\th  ))s(\varepsilon(X\o{}_{[+]}t(U(X\o{}_{[-]}))))X\t \\
        =&t(U(X\t ))s(\varepsilon(X\o{}_{[+]}\o\,t(U(X\o{}_{[-]}))))X\o{}_{[+]}\t\\
        =&t(U(X\t ))s(\varepsilon(X\o{}_{[+]}\o\,))X\o{}_{[+]}\t\,s(U(X\o{}_{[-]}))\\
        =&t(U(X\t ))X\o{}_{[+]}\,s(U(X\o{}_{[-]}))=\Ad_{U}^{R}(X).
      \end{align*}
Since $\partial U$ is trivial, $\Ad_{U}$ in either case is necessarily an automorphism  by Proposition~\ref{propisomGamma}.
\end{proof}

\section{Bisections and cocycles on ES-Hopf algebroids}\label{secesb}

Here we see how the results of Sections~\ref{secbisec} play out for the case of the Ehresmann-Schauenburg (ES)  Hopf algebroid $\CL(P,H)$
associated to a quantum principle bundle or Hopf-Galois extension. Vertical bisections on $\CL(P,H)$ were previously discussed in some form in \cite{HL20}. We check that similarly the group of all left bisections reduces as expected to a group $\Aut_H(P)$ of bundle automorphisms. Remarkably, we find that the right bisections  reduce to the same group.  We  also show that  $\CH^2(\CL,B)$ in this context reduces to a previous notion of non-Abelian cohomology when the quantum principal bundle is cleft, i.e. trivial in a extension theory sense.

\subsection{Bisections on $\CL(P,H)$}\label{sec:ESbisec}

We first recall the notion of a Hopf-Galois extension or quantum principal bundle as in \cite{BrzMa,brz-tr,Ma:dia,BegMa}, for example, and define its automorphism group.

\begin{defi} \label{def:hg}
Let $H$ be a Hopf algebra, $P$ an $H$-comodule algebra with coaction $\Delta_R$ and $B:= P^{co H}=\big\{b\in P ~|~ \Delta_R (b) = b \ot 1_H \big\} \subseteq
P$ the coinvariant subalgebra.
The extension $B\subseteq P$ is called  \textup{Hopf--Galois} if the
\textit{canonical map}
\begin{align*}
\chi := (m \ot \id) \circ (\id \ot _B \Delta_R ) :
P \ot _B P \longrightarrow P \ot H ,
\quad q \ot_B p  &\mapsto q \zero{p} \ot \one{p}
\end{align*}
is an isomorphism.
\end{defi}

Since the canonical map $\chi$ is left $P$-linear, its inverse is
determined by the restriction $\tau:=\chi^{-1}_{|_{1_P \ot H}}$, referred to as the \textit{translation map} and denoted
\begin{eqnarray*}
\tau:  H\to P\ot _B P, 
\quad h \mapsto \tau(h) = \tuno{h} \ot_B \tdue{h} \, .
\end{eqnarray*}
A Hopf-Galois extension is said to be faithfully flat if $P$ is a  faithfully flat left $B$-module. Then it is known, e.g. \cite[Prop.~3.6]{brz-tr} \cite[Lem.~34.4]{BW} that it obeys
\begin{align}\label{equ. translation map 1}
  \tuno{h} \ot_B \zero{\tdue{h}} \ot \one{\tdue{h}} &= \,\tuno{\one{h}} \ot_B \tdue{\one{h}} \ot
\two{h},\\
\label{equ. translation map 2}
~~ \tuno{\two{h}}  \ot_B \tdue{\two{h}} \ot S\one{h}&= \zero{\tuno{h}} \ot_B {\tdue{h}}  \ot \one{\tuno{h}},\\
\label{equ. translation map 3}
\tuno{h}\zero{\tdue{h}}\ot \one{\tdue{h}} &= 1_{P} \ot h,\\
\label{equ. translation map 4}
    \zero{p}\tuno{\one{p}}\ot_{B}\tdue{\one{p}} &= 1_{P} \ot_{B}p,\\
\label{equ. translation map 5}    \tuno{\one{h}}\ot_{B}\tdue{\one{h}}\tuno{\two{h}}\ot_{B}\tdue{\two{h}} &=\tuno{h}\ot_{B}1\ot_{B}\tdue{h}
\end{align}
for all $h\in H$ and $p\in P$.

\begin{defi}\label{def:autP} Let $B=P^{co H}\subseteq P$ be a faithfully flat Hopf--Galois  extension. The {\em bundle automorphism group} $\Aut_H(P)$ is the collection
of invertible right $H$-comodule unital algebra maps of $P$, with composition product. 
\end{defi}
Note that $F\in \Aut_H(P)$ restricts to $F|_B\in \Aut(B)$, and conversely $F\in \Aut_H(P)$ if and only if it is an $H$-colinear unital algebra map with invertible $F|_B$, because
 \[
F^{-1}(p) =(F|_{B})^{-1} (\zero{p} F (\tuno{\one{p}}) )\, \tdue{\one{p}}
\]
for all $p\in P$ provides the inverse. This can be checked using (\ref{equ. translation map 1})-(\ref{equ. translation map 5}), see \cite{HL20}. 
The case where $F|_B=\id$ is called a {\em vertical bundle automorphism} or `gauge transformation'. This provides noncommutative  or `quantum' versions of the classical notions recalled in Section~\ref{sec:preclass}. In this context, it is known from examples that it can be too strong to demand that $F$ is an algebra map and hence one has a notion of extended gauge transformations, where this is replaced by $F$ unital and invertible \cite{Ma:dia,Ma:book}.

\begin{defi}\label{def:ec}
Let $B=P^{co H}\subseteq P$ be a faithfully flat Hopf-Galois extension. The \textup{Ehresmann-Schauenburg Hopf algebroid} is
\begin{equation*}
\cL(P, H)=(P\ot P)^{co H} := \{p\ot  q\in P\ot  P \, \quad|\quad \zero{p}\ot  \zero{q}\ot  \one{p}\one{q}=p\ot  q\ot  1_H \} \label{ec2}
\end{equation*}
 with $B$-bimodule structure inherited from $P$ and $B$-coring coproduct and counit
 \[
\Delta(p\otimes q)=\zero{p}\otimes \one{p}\tuno{}\ot_{B}p\o\tdue{}\otimes q,\quad  \varepsilon(p\otimes q)=pq.
\]
Moreover, $\cL(P, H)$ is a $B^e$-ring with the product and unit
\[ (p\otimes q)\cdot_{\mathcal{L}}(r\otimes u)=pr\otimes uq,\quad \eta(b\tens c)=b\tens c
\]
for all $p\otimes q$, $r\otimes u \in \cL(P, H)$ and $b\tens c\in B^e$. Here $s(b)=b\otimes 1$ and $t(b)=1\otimes b$. \end{defi}
The coinvariance condition in the definition of $\CL$ is equivalent to
\[ \zero{p}\otimes \one{p}\tuno{}\ot_{B}p\o\tdue{} q=p\otimes q\otimes_B 1\]
used in the original formulation \cite{schau1}\cite[\S 34.13]{BW}. That we have a left Hopf algebroid was shown in \cite{HM22}, with
\begin{align}
     \lambda(1\ot_{B}(p\tens q))= (p\ot q)_{+}\ot_{B^{op}}(p\ot q)_{-}=p\ot\tdue{\one{q}}\ot_{B^{op}}\zero{q}\ot\tuno{\one{q}}
\end{align}
for all $p\ot q\in\cL(P, H)$. 

\begin{prop}\label{propgauge} $\CB(\CL(P,H))\simeq \Aut_H(P)\simeq \CB^R(\CL(P,H))$ as groups, with the latter requiring $H$ to have an invertible antipode.
\end{prop}
\begin{proof} (1) The first isomorphism is a slight generalisation of \cite{HL20}. As there, given a left bisection $\sigma\in \CB(\cL(P, H))$,  we can construct a gauge transformation by
    \begin{align}\label{equ. left bisection to gauge transformation}
        F_{\sigma}(p):=\sigma(\zero{p}\ot\tuno{\one{p}})\tdue{\one{p}}
    \end{align}
for all $p\in P$. Conversely, for any gauge transformation $F$, we can construct a left bisection by
\begin{align}
    \sigma_{F}(p\ot q):=F(p)q
\end{align}
for  all $p\tens q\in\cL(P,H)$. One can also check that the group structures are compatible, i.e. $\sigma_{F}\ast\sigma_{G}=\sigma_{F\circ G}$ for any $F,G\in\Aut_{H}(P)$. In particular, the inverse of a left bisection 
\begin{align}
        \sigma^{-1}(p\ot q)=(\sigma\circ s)^{-1}(p\sigma(\zero{q}\ot\tuno{\one{q}})\tdue{\one{q}})
    \end{align}
 using Theorem~\ref{thm. group of left bisections} corresponds to the inverse of $F_\sigma$ using the observation below Definition~\ref{def:autP}.

(2) Given a right bisection $\tau\in \CB^R(\cL(P, H))$, and provided $H$ has an invertible antipode, we  construct a gauge transformation by
    \begin{align}\label{equ. right bisection to gauge transformation}
        F_{\tau}(p):=\tuno{S^{-1}(\one{p})}\tau(\tdue{S^{-1}(\one{p})}\ot\zero{p})
    \end{align}
for all $p\in P$. Conversely, for any gauge transformation $F$, we can construct a right bisection
\begin{align}\label{equ. gauge transformation to right bisection}
    \tau_{F}(p\ot q):=pF(q)
\end{align}
for all $p\ot q\in \cL(P, H)$. Moreover, for any $F\in \Aut_{H}(P)$,
\begin{align*}
    F_{\tau_{F}}(p)=&\tuno{S^{-1}(\one{p})}\tdue{S^{-1}(\one{p})}F(\zero{p})=F(p).
\end{align*}
Conversely, for any right bisection $\tau$,
\begin{align*}
    \tau_{F_{\tau}}(p\ot q)=&p\tuno{S^{-1}(\one{q})}\tau(\tdue{S^{-1}(\one{q})}\ot\zero{q})=\zero{p}\tuno{\one{p}}\tau(\tdue{\one{p}}\ot q)=\tau(p\ot q),
\end{align*}
where the 2nd step uses that $\zero{p}\ot\one{p}\ot q=p\ot S^{-1}(\one{q})\ot\zero{q}$. Moreover,
\begin{align*}
    (\tau_{F}\bullet\tau_{G})(p\ot q)=&\tau_{F}(t\circ\tau_{G}(\tdue{\one{p}}\ot q)(\zero{p}\ot \tuno{\one{p}}))\\
    =&\tau_{F}(\zero{p}\ot \tuno{\one{p}}\tdue{\one{p}}G(q))\\
    =&\tau_{F}(p\ot G(q))=pF\circ G(q)=\tau_{F\circ G}(p\ot q)
\end{align*}
for $F, G\in \Aut_{H}(P)$. The bijection implies that $\CB^R(\CL(P,H))$ with $\bullet$ is also a group,  isomorphic to $\Aut_H(P)$. Explicitly, given a right bisection $\tau$, its inverse can be given by
    \begin{align*}
        \tau^{-1}(p\ot q)=(\tau\circ t)^{-1}(p\tau(\zero{q}\ot\tuno{\one{q}})\tdue{\one{q}}).
    \end{align*}
\end{proof}

That $\CB^R(\CL(P,H))$ is a group agrees with a recent result \cite{HS25} that $\CL(P,H)$ is an anti-left Hopf algebroid when $H$ has invertible antipode. The group crossed modules associated to left and right bisections likewise reduce to a group crossed module $\mu:\Aut_H(P)\to \Aut(\CL(P,H))$. Explicitly, this sends $F\in \Aut_H(P)$ to $(\Phi_F,\phi_F)$,  where
\[ \Phi_F(p\tens q)= F(p) \tens F(q),\quad \phi_F=F|_B. \]
Moreover, any automorphism $(\Phi,\phi)$ sends $F$ to $((\Phi, \phi)\la F)(p)=\phi\circ\sigma_{F}\circ\Phi^{-1}(\zero{p}\ot\tuno{\one{p}})\tdue{\one{p}}$.

We next study the subgroup $\CB_{ver}(\CL(P,H)))^{op}=\CB^R_{ver}(\CL(P,H)))$ of vertical bisections or `gauge transformations' of the quantum principal bundle. Note that the isomorphism in Proposition~\ref{propgauge}  when restricted to the vertical bisections becomes inversion. Since $\CL(P,H)$ is a left Hopf algebroid, we know from Theorem~\ref{thm: boundary map} and the comment after it that the group of vertical right bisections coincides with $Z^1(\CL(P,H),B)$.

\begin{prop}\label{lem. vertical bisection of es bialgebroid} Let $B=P^{co H}\subseteq P$ be a faithfully flat Hopf-Galois extension. Then $Z^{1}(\cL(P, H), B)\simeq\tZ^1(H,P)$ as groups, where the latter is defined as the set of unital $H$-colinear maps $f:H\to P$ obeying
\[  p_{\scriptscriptstyle{(0)}}q_{\scriptscriptstyle{(0)}}f(p\o q\o )=p_{\scriptscriptstyle{(0)}}f(p\o )q_{\scriptscriptstyle{(0)}}f(q\o )\]
for all $p,q\in P$, where $H$ is a  right $H$-comodule by its right adjoint coaction. This forms a group under  convolution product, with inverse
\[  f^{-1}(h)=\tuno{h\t }f(Sh\o )\tdue{h\t }.\]
\end{prop}
\begin{proof}
    Given $\sigma\in Z^{1}(\cL(P, H), B)$ in the form of a right bisection, we define $f\in \tZ^1(H,P)$ by
    \begin{align}
        f_{\sigma}(h):=\tuno{h\o }\sigma(\tdue{h\o }\ot \tuno{h\t })\tdue{h\t }, 
    \end{align}
 which is well defined as $\sigma$ is $B$-bilinear. By  applying (\ref{equ. translation map 1}) and (\ref{equ. translation map 2}), one has that this is a comodule map,
\begin{align*}
    f_{\sigma}(h)_{\scriptscriptstyle{(0)}}\ot f_{\sigma}(h)\o =f_{\sigma}(h\t )\ot (Sh\o )h\th  .
\end{align*}
Moreover, by using $\sigma(XY)=\sigma(X s\circ\sigma(Y))$, we can get $p_{\scriptscriptstyle{(0)}}q_{\scriptscriptstyle{(0)}}f_{\sigma}(p\o q\o )=p_{\scriptscriptstyle{(0)}}f_{\sigma}(p\o )q_{\scriptscriptstyle{(0)}}f_{\sigma}(q\o )$. Indeed,
\begin{align*}
    p_{\scriptscriptstyle{(0)}}q_{\scriptscriptstyle{(0)}}f_{\sigma}(p\o q\o )=&\sigma(p_{\scriptscriptstyle{(0)}}q_{\scriptscriptstyle{(0)}}\ot \tuno{q\o }\tuno{p\o })\tdue{p\o }\tdue{q\o }\\
    =&\sigma(p_{\scriptscriptstyle{(0)}}\ot\sigma(q_{\scriptscriptstyle{(0)}}\ot \tuno{q\o })\tuno{p\o })\tdue{p\o }\tdue{q\o }\\
    =&\sigma(p_{\scriptscriptstyle{(0)}}\ot\tuno{p\o })\tdue{p\o }\sigma(q_{\scriptscriptstyle{(0)}}\ot \tuno{q\o })\tdue{q\o }\\
    =&p_{\scriptscriptstyle{(0)}}f_{\sigma}(p\o )q_{\scriptscriptstyle{(0)}}f_{\sigma}(b\o ).
\end{align*}
Conversely, if $f\in \tZ^{1}(H, P)$, we can define $\sigma_{f}\in Z^{1}(\cL(P, H), B)$ by
\begin{align}
    \sigma_{f}(p\ot q):=p_{\scriptscriptstyle{(0)}}f(p\o )q.
\end{align}
Using that $f$ is a comodule map, one can check that the image of $\sigma_{f}$ belongs to $B$ and  that it is $B$-bilinear.  Next, for $X=p\ot p'$, $Y=q\ot q'$, we have
\begin{align*}
    \sigma_{f}(XY)=p_{\scriptscriptstyle{(0)}}q_{\scriptscriptstyle{(0)}}f(p\o q\o )q'p'=p_{\scriptscriptstyle{(0)}}f(p\o )q_{\scriptscriptstyle{(0)}}f(q\o )q'p'=\sigma_{f}(X t\circ\sigma_{f}(Y)).
\end{align*}
One can also check that $\sigma_{f}(Xt\circ\sigma_{f}(Y))=\sigma_{f}(Xs\circ\sigma_{f}(Y))$ by using  that $p_{\scriptscriptstyle{(0)}}bf(p\o )=p_{\scriptscriptstyle{(0)}}f(p\o )b$ for all $b\in B$.
We also have
\begin{align*}
    \sigma_{f_{\sigma}}(p\ot q)=p_{\scriptscriptstyle{(0)}}\tuno{p\o }\sigma(\tdue{p\o }\ot \tuno{p\t })\tdue{p\t }q=\sigma(p\ot q)
\end{align*}
and $f_{\sigma_{f}}(h)=f(h)$. Moreover, one can check that $\sigma_{f\ast g}(X)=\sigma_{f}\bullet\sigma_{g}(X)$ for the product of vertical right bisections. Finally, it is straightforward to check that $\sigma_f^{-1}$ corresponds to the formula stated for  the inverse in $\tZ^1(H,P)$. \end{proof}

Combining  Propositions~\ref{propgauge} and~\ref{lem. vertical bisection of es bialgebroid}, we also have a group isomorphism between vertical bundle automorphisms or `gauge transformations' and $\tZ^{1}(H, P)$.  Given a vertical bundle automorphism $F$, the corresponding $f\in \tZ^{1}(H, P)$ is given by
\[ f(h):=\tuno{h}F(\tdue{h})\]
for all $h\in H$. Conversely, if $f\in \tZ^{1}(H, P)$ then $F(a):=a_{\scriptscriptstyle{(0)}}f(a\o )$ is a vertical gauge transformation. This is in keeping with the classical geometry where equivariant functions on the total space correspond to gauge transformations. More generally, we have:

\begin{prop}\label{prop. tc} Let $B=P^{co H}\subseteq P$ be a faithfully flat Hopf-Galois extension. Then $C^1(\cL(P,H)),B)\simeq\tC^1(H,P)$ as groups, where the latter is  defined as the set of unital convolution invertible
$H$-colinear maps $f: H\to P$, such that
\begin{align*}    p_{\scriptscriptstyle{(0)}}bf(p\o )=p_{\scriptscriptstyle{(0)}}f(p\o )b
\end{align*}
 for all $p\in P$ and $b\in B$, and forms a group with the convolution product.
\end{prop}
We omit the proof as it follows the same lines as the proof of Proposition~\ref{lem. vertical bisection of es bialgebroid}. This recovers the correspondence in  \cite[Prop. 2.2]{Ma:dia} between such $f$ and extended gauge transformations.

\subsection{2-cocycles on $\CL(P,H)$ in the $\ra$-commutative case}\label{sec:EScocy}

Here, we provide limited results where we suppose that the algebra $P$ of a Hopf-Galois extension is also a right $H$-module algebra with action $\ra$ such that
\[ pq=\zero{q}(p\ra\one{q})\]
for all $p, q\in P$, in which case we say that $P$ is $\ra$-commutative. This is `braided commutative' in the special case where the action and coaction fit together to form a Drinfeld-Yetter module as in Section~\ref{secact}, but this further structure will not be needed here. In particular, our assumption implies that $B=P^{co H}$ is contained in the centre of $P$, which implies for example that $\tC^1(H,P)$ in Proposition~\ref{prop. tc} is just all convolution-invertible $H$-colinear maps as the displayed condition becomes empty.

\begin{lem}\label{proptZ2} Let $B=P^{co H}\subseteq P$ be a faithfully flat Hopf-Galois extension and suppose that $P$ is also a right $H$-module algebra such that $P$ is $\ra$-commutative. Then $Z^{2}(\cL(P, H), B)$ is in one to one correspondence with the set  $\tZ^2(H,P)$ defined as elements $\gamma\in Z^2_{\ra,as}(H,P)$ that are $H$-colinear in the sense
\[ \gamma(h, g)_{\scriptscriptstyle{(0)}}\ot \gamma(h, g)\o =\gamma(h\t , g\t )\ot S(g\o h\o )g\th  h\th  .\]
\end{lem}
\begin{proof}
    Let $\gamma\in \tZ^2(H,P)$ and define
    \begin{align}
        \Gamma_{\gamma}(p\ot p', q\ot q')=p_{\scriptscriptstyle{(0)}}q_{\scriptscriptstyle{(0)}}\gamma(q\o , p\o )q'p'
    \end{align}
    for all $p\ot p', q\ot q'\in \cL(P, H)$. One can check that $\Gamma_{\gamma}$ is a $B$-bimodule map and that its image belongs to $B$. Moreover, it is well defined over $\ot_{B^{e}}$. To check the 2-cocycle condition, let $X=p\ot p'$, $Y=q\ot q'$ and $Z=r\ot r'$. Then on the one hand,
\begin{align*}
    \Gamma_{\gamma}&(X, s(\Gamma_{\gamma}(\one{Y}\one{Z}))\two{Y}\two{Z})\\
    &=\Gamma_{\gamma}(p\ot p', \zero{q}\zero{r}\gamma(\one{r},\one{q})\tuno{\two{r}}\tuno{\two{q}}\tdue{\two{q}}\tdue{\two{r}}\ot r'q')\\
    &=\Gamma_{\gamma}(p\ot p', \zero{q}\zero{r}\gamma(\one{q},\one{r})\ot r'q')\\    &=\zero{p}\zero{q}\zero{r}\gamma(\one{r}, \one{q})\gamma(q\t r\t , p)r'q'p'.
\end{align*}
On the other hand,
    \begin{align*}
        \Gamma_{\gamma}&(s(\Gamma_{\gamma}(\one{X}\one{Y}))\two{X}\two{Y}, Z)\\
        &=\zero{p}\zero{q}\gamma(q\o , p\o )r_{\scriptscriptstyle{(0)}}\gamma(r\o ,p\t q\t )r'q'p'\\
        &=\zero{p}\zero{q}\zero{r}\gamma(q\o , p\o )\ra r\o \gamma(r\t ,p\t q\t )r'q'p',
    \end{align*}
where the last step uses that $P$ is $\ra$-commutative. The two expressions are equal by $\gamma$ a cocycle. 
Conversely, if $\Gamma\in C^{1}(\cL(P, H), B)$, define $\gamma_{\Gamma}\in \tZ^{2}(H, P)$ by
\begin{align}
    \gamma_{\Gamma}(h, g):=\tuno{h\o }\tuno{g\o }\Gamma(\tdue{g\o }\ot\tuno{g\t }, \tdue{h\o }\ot\tuno{h\t })\tdue{g\t }\tdue{h\t },
\end{align}
which is well defined over all the balanced tensor products since $\Gamma$ is $B$-bilinear and defined over $\ot_{B^{e}}$. It is  $H$-colinear by using the same method as in Lemma \ref{lem. vertical bisection of es bialgebroid}. Using $\gamma_{\Gamma\star \Lambda}=\gamma_{\Gamma}\ast\gamma_{\Lambda}$, one can also see that $\gamma_{\Gamma}$ is  invertible with inverse given by $\gamma_{\Gamma^{-1}}$. To check the 2-cocycle condition, on the one hand,
\begin{align*}
    \gamma_{\Gamma}&(h\o , g\o )\gamma_{\Gamma}(g\t h\t , f)\\
    &=\tuno{h\o }\tuno{g\o }\Gamma(\tdue{g\o }\ot\tuno{g\t }, \tdue{h\o }\ot\tuno{h\t })\tuno{f\o }\\
    &\quad \Gamma(\tdue{f\o }\ot\tuno{f\t }, \tdue{g\t }\tdue{h\t }\ot\tuno{h\th  }\tuno{g\th  })\tdue{f\t }\tdue{g\th  }\tdue{h\th  }\\
    &=\tuno{h\o }\tuno{g\o }\tuno{f\o }\Gamma(\tdue{g\o }\ot\tuno{g\t }, \tdue{h\o }\ot\tuno{h\t })\\
    &\quad \Gamma(\tdue{f\o }\ot\tuno{f\t }, \tdue{g\t }\tdue{h\t }\ot\tuno{h\th  }\tuno{g\th  })\tdue{f\t }\tdue{g\th  }\tdue{h\th  },
\end{align*}
and on the other hand,
\begin{align*}
    \gamma_{\Gamma}&(g\o , f\o )\ra h\o  \gamma_{\Gamma}(h\t , f\t g\t )\\
    &=\big(\tuno{g\o }\tuno{f\o }\Gamma(\tdue{f\o }\ot\tuno{f\t }, \tdue{g\o }\ot\tuno{g\t })\tdue{f\t }\tdue{g\t }\big)\ra h\o \\
    &\quad\tuno{h\t }\tuno{g\th  }\tuno{f\th  }\Gamma(\tdue{f\th  }\tdue{g\th  }\ot\tuno{g\fo  }\tuno{f\fo  }, \tdue{h\t }\ot\tuno{h\th  } )\tdue{f\fo  }\tdue{g\fo  }\tdue{h\th  }\\
    &=\zero{\tuno{h\t }}\big(\tuno{g\o }\tuno{f\o }\Gamma(\tdue{f\o }\ot\tuno{f\t }, \tdue{g\o }\ot\tuno{g\t })\tdue{f\t }\tdue{g\t }\big)\ra (h\o \tuno{h\t }{}\o )\\
    &\quad\tuno{g\th  }\tuno{f\th  }\Gamma(\tdue{f\th  }\tdue{g\th  }\ot\tuno{g\fo  }\tuno{f\fo  }, \tdue{h\t }\ot\tuno{h\th  } )\tdue{f\fo  }\tdue{g\fo  }\tdue{h\th  }\\
    &=\tuno{h\th  }\big(\tuno{g\o }\tuno{f\o }\Gamma(\tdue{f\o }\ot\tuno{f\t }, \tdue{g\o }\ot\tuno{g\t })\tdue{f\t }\tdue{g\t }\big)\ra (h\o S h\t )\\
    &\quad\tuno{g\th  }\tuno{f\th  }\Gamma(\tdue{f\th  }\tdue{g\th  }\ot\tuno{g\fo  }\tuno{f\fo  }, \tdue{h\t }\ot\tuno{h\th  } )\tdue{f\fo  }\tdue{g\fo  }\tdue{h\th  }\\
    &=\tuno{h\o }\tuno{g\o }\tuno{f\o }\Gamma(\tdue{f\o }\ot\tuno{f\t }, \tdue{g\o }\ot\tuno{g\t })\\
    &\quad\Gamma(\tdue{f\t }\tdue{g\t }\ot\tuno{g\th  }\tuno{f\th  }, \tdue{h\o }\ot\tuno{h\t } )\tdue{f\th  }\tdue{g\th  }\tdue{h\t }, 
\end{align*}
where the second step uses that $P$ is $\ra$-commutative. The two are equal by $\Gamma$ a 2-cocycle. By using that $\gamma_{\Gamma}$ is colinear, one can also see that
\begin{align*}
    \gamma_{\Gamma}(h\o , g\o )a\ra(g\t h\t )=(b\ra (\one{g}\one{h})) \gamma_{\Gamma}(\two{h}, \two{g}).
\end{align*}
Finally, it is not hard to check that $\Gamma_{\gamma_{\Gamma}}=\Gamma$ and $\gamma_{\Gamma_{\gamma}}=\gamma$. \end{proof}

\begin{rem}\label{rem:Pgamma}
If $P$ is an $H$-comodule algebra and $\ra$-commutative, and $\gamma\in \tZ^2(H,P)$, then we can define a new unital algebra $P^\gamma$ with the same unit and a modified product on the underlying vector space of $P$,
\begin{align}
    p\cdot_{\gamma}q=p_{\scriptscriptstyle{(0)}}q_{\scriptscriptstyle{(0)}}\gamma(q\o , p\o )
\end{align}
for all $p,q\in P$.  The new product is associative using the 2-cocycle condition for $\gamma$. 
    \end{rem}

\begin{thm}\label{thm:ESZ2}
Let $B=P^{co H}\subseteq P$ be a faithfully flat Hopf-Galois extension and $P$ be $\ra$-commutative.    For any $f\in \tC^1(H,P)$, we obtain a 2-cocycle $\partial f\in \tZ^2(H,P)$ by
    \begin{align}
        \partial f(h, g)=  (f^{-1}(g\o )\ra h\o) f^{-1}(h\t ) f(g\t h\th)
    \end{align}
for all $h,g\in H$. The subset  of elements $f\in \tC^1(H,P)$ for which $\partial f(h,g)=\eps(h)\eps(g)$ for all $h,g\in H$, is a subgroup isomorphic to $\tZ^1(H,P)$. Moreover, if $\gamma\in \tZ^2(H,P)$ and $f\in \tC^1(H,P)$ then
\begin{align}
    \gamma^{f}(h,g):=( f^{-1}(g\o )\ra h\o) f^{-1}(h\t )\gamma(h\t , g\t ) f(g\th  h\th  )
\end{align}
also belongs to $\tZ^2(H,P)$ and gives an action of the group $\tC^1(H,P)$ on this set. The set $\tilde\CH^{2}(H, P)$ of orbits under this action can be identified with  $\H^{2}(\cL(P, H), B)$.
\end{thm}
\begin{proof}
It is easy to see that if $f\in \tC(H,P)$ obeys $f(hg)=f(g\o)(f(h)\ra g\t)$ for all $g,h\in H$ then $f\in \tZ^1(H,P)$. Conversely, if $f\in \tZ^1(H,P)$, we know that there is $\sigma\in Z^{1}(\cL(P, H), B)$, such that $f=f_{\sigma}$ by Proposition \ref{lem. vertical bisection of es bialgebroid}. Then,
\begin{align*}
 f(hg)=&   f_{\sigma}(hg)=\tuno{g\o }\tuno{h\o }\sigma(\tdue{h\o }\tdue{g\o }\ot \tuno{g\t }\tuno{h\t })\tdue{h\t }\tdue{g\t }\\
    =&\tuno{g\o }\tuno{h\o }\sigma(\tdue{h\o }\sigma(\tdue{g\o }\ot \tuno{g\t })\ot\tuno{h\t })\tdue{h\t }\tdue{g\t }\\
    =&\tuno{g\o }\sigma(\tdue{g\o }\ot \tuno{g\t })\tuno{h\o }\sigma(\tdue{h\o }\ot\tuno{h\t })\tdue{h\t }\tdue{g\t }\\
    =&\tuno{g\o }\sigma(\tdue{g\o }\ot \tuno{g\t })f_{\sigma}(h)\tdue{g\t }\\
    =&\tuno{g\o }\sigma(\tdue{g\o }\ot \tuno{g\t })\tdue{g\t }\zero{} (f_{\sigma}(h)\ra\tdue{g\t }\o)\\
    =&\tuno{g\o }\sigma(\tdue{g\o }\ot \tuno{g\t })\tdue{g\t } (f_{\sigma}(h)\ra g\th)\\
    =&f_{\sigma}(g\o) (f_{\sigma}(h)\ra g\t) =f(g\o) (f(h)\ra g\t).
\end{align*}
Hence, $f\in \tZ^1(H,P)$ if and only if $f(hg)=f(g\o)(f(h)\ra g\t)$. Moreover, it is easy to see that if $f\in \tC(H, P)$ obeys $f(hg)=f(g\o)(f(h)\ra g\t)$ for all $h,g\in H$ then $\partial f(h,g)=\varepsilon(h)\varepsilon(g)$ for all $h,g\in H$. Conversely, if the latter holds then 
\begin{align*}
    f(g\o)(f(h)\ra g\t)=&f(g\o)(f(h\o)\ra g\t) \partial f(g\th, h\t)\\
    =&f(g\o)(f(h\o)\ra g\t)(f^{-1}(h\t)\ra g\th)f^{-1}(g\fo)f( h\th g\fiv) =f(hg).
\end{align*}
Next, by $\ra$-commutativity, we have
\begin{align*}
    (f(g)\ra h\o) f(h\t)=&f(h\t)\zero{}\, ((f(g)\ra h\o)\ra f(h\t)\o)
    =f(h\th)\,(f(g)\ra (h\o (Sh\t)
    h\fo))\\
    =&f(h\o)\,(f(g)\ra h\t), 
\end{align*}
from which one can check directly that $\gamma^{f}\in \tZ^{2}(H, P)$. Moreover, for any $f,\tilde f\in \tC^1(H,P)$, 
 \begin{align*}
     \gamma^{f\ast f'}(h,g)=&(f'^{-1}(g\o)\ra h\o)(f^{-1}(g\t)\ra h\t)\,f'^{-1}(h\th)\,f^{-1}(h\fo)\gamma(h\fiv,g\th)f(g\fo\,h\si)f'(g\fiv h\sev)\\
     =&(f'^{-1}(g\o)\ra h\o)\,f'^{-1}(h\t)\,(f^{-1}(g\t)\ra h\th)f^{-1}(h\fo)\gamma(h\fiv,g\th)f(g\fo\,h\si)f'(g\fiv h\sev)\\
     =&(\gamma^f){}^{f'}(h,g)
 \end{align*}
for all $h,g\in H$, hence we have a group action. For the last part, by Lemma \ref{proptZ2}, it is sufficient to show that $\gamma^{f}$ corresponds to the 2-cocycle $\Gamma_{\gamma}^{U_{f}}\in Z^{2}(\cL, B)$, where $U_{f}$ is given by Proposition~\ref{prop. tc} and $\Gamma_{\gamma}$ is given by Lemma \ref{proptZ2}. More precisely, by applying Theorem~\ref{thm: boundary map}, we can see for any $p\ot p'$ and $q\ot q'\in\cL(P, H)$ that
\begin{align*}
    \Gamma_{\gamma}^{U_{f}}(X, Y)=&U_{f}{}^{-1}(\one{X}t(U^{-1}(\one{Y})))\Gamma_{\gamma}(\two{X}, \two{Y})U_{f}(\three{X}\three{Y})\\
    =&p_{\scriptscriptstyle{(0)}}q_{\scriptscriptstyle{(0)}}f^{-1}(p\o )\ra q\o f^{-1}(q\t )\gamma(q\th  , p\t )f(p\th  q\fo  )q'p'= \Gamma_{\gamma^{f}}
\end{align*}
using Lemma \ref{proptZ2} and $\gamma^f$ as stated. \end{proof}

The theorem gives us a concrete description of $\H^2(\cL(P, H),B)$ for the ES-Hopf algebroid.
Note that we have not discussed $Z^2(\CL(P,H),B)$ and $\tilde Z^2(H,P)$ as groups; their structure is more complicated as one has to use a modified $P^\gamma$ as in Remark~\ref{rem:Pgamma} on application of the first cocycle (this is similar to the composition of Drinfeld cotwists of comodule algebras).

\subsection{2-cocycles on $\CL(P,H)$ in the  cleft extension case} \label{sec:EStriv}

Here we look at $\CL(P,H)$ in the trivial bundle or cleft extension case. By \cite[Lem.~4.5]{HM22}, this is isomorphic to  a {\em cocycle Hopf algebroid}  $B^e\#_\gamma H$ for a 2-cocycle $\gamma$. To keep things simple, we proceed with  trivial $\gamma$, in which case we have a Hopf algebroid $B^e\#H$ associated to a left $H$-module algebra $B$ and built on $B^e\tens H$
with product
\[(b\tens b'\# h)(c\tens c'\# g)=b(\one{h}\triangleright c)\ot c' (S\two{g}\triangleright  b')\# \two{h}\one{g}\]
and source and target maps
\[s(b)=b\ot1\# 1, \quad t(b)=1\ot b\# 1.\]
The coproduct and counit are
\[ \Delta (b\tens b'\# h)=(b\tens 1\# \one{h})\tens_B (1\tens b'\#\two{h}),\quad \eps(b\ot b'\# h)=b(\one{h}\triangleright b'),\]
where the $B$-bimodule structure is
\[ c. (b\ot b'\# h). c'=cb\ot b' (S\two{h}\triangleright c')\# \one{h}\]
for all $b,b',c,c'\in B$. We briefly examine how $C^1(\CL,B), Z^1(\CL,B), Z^2(\CL,B)$ and $\CH^2(\CL,B)$ reduce when applied to $\CL=B^e\# H$. To this end, we define
\[
Z^{1}_{\la,as}(H, B):=\{u: H\to B \quad\textrm{unital linear}\ |\  u(hg)=(h\o \la u(g))u(h\t ),  (h\o \la b)u(h\t )=u(h\o )(h\t \la b)\}
\]
\[
C^{1}_{\la,as}(H, B):=\{u: H\to B\quad\textrm{conv. inv. unital linear}\ | \  (h\o \la b)u(h\t )=u(h\o )(h\t \la b)\}
\]
with coboundary map
\[ \partial: C^{1}_{\la, as}(H, B)\to Z^{2}_{\la,as}(H, B),\quad \partial u(h, g)=u^{-1}(h\o )(h\t \la u^{-1}(g\o ))u(h\th  g\t ).\]
One can check that $Z^1_{\la,as}(H,B)$ is the subgroup of elements $u\in C^1_{\la, as}(H,B)$ for which $\partial u(h,g)=\eps(h)\eps(g)$ for all $h,g\in H$. Moreover, the image of $\partial$ obeys the condition in Definition~\ref{def: assoc type}. Also, if $\gamma\in Z^{2}_{\la,as}(H, B)$  then so is
\[\gamma^u(h, g)=u^{-1}(h\o )(h\t \la u^{-1}(g\o ))\gamma(h\th  , g\t )u(h\fo  g\th  ),\]
and it is straightforward to check that this gives a group action of $C^1_{\la,as}(H,B)$ on $Z^2_{\la,as}(H,B)$. We denote the set of orbits in $Z^{2}_{\la,as}(H, B)$ under this transformation by $\H^{2}_{\la,as}(H, B)$.

\begin{prop}\label{prop. BEH} $Z^1(B^{e}\#H, B)\simeq Z^1_{\la,as}(H, B)$, $ C^{1}(B^{e}\#H, B)\simeq C^{1}_{\la,as}(H, B)$ as groups. Moreover, $Z^{2}(B^{e}\#H, B)\simeq Z^{2}_{\la,as}(H, B)$ and  $\H^{2}(B^{e}\#H, B)\simeq \H^{2}_{\la,as}(H, B)$.
\end{prop}
\begin{proof}
  $Z^1_{\la,as}(H, B)$ is a group with inverse $u^{-1}(h)=h\o \la u(Sh\t )$. Both $C^1_{\la, as}(H, B)$ and $Z^1_{\la,as}(H, B)$ are groups with convolution product as product. The isomorphism $Z^{1}_{\la,as}(H, B)\simeq Z^{1}(B^{e}\#H, B)$ and $C^{1}_{\la, as}(H, B)\simeq C^{1}(B^{e}\#H, B)$ is given in the two directions by
    \[U_{u}(a\ot b\#h)=a(h\o \la b)u(h\t ),\quad 
u_{U}(h)=U(1\ot 1\# h).
\]
The one to one correspondence
$Z^{2}_{\la,as}(H, B)\simeq Z^{2}(B^{e}\#H, B)$ is given by \cite{HM22}
\[
 \Gamma_{\gamma}(b\ot b'\# h, c\ot c'\# g)=b(\one{h}\triangleright c)\gamma(\two{h}, \one{g})((\three{h}\two{g})\triangleright c')(\four{h}\triangleright b'),
\]
\[
\gamma_{\Gamma}(h, g)=\Gamma(1\ot 1\# h, 1\ot 1\# g).\
\]
For the last part of the statement, we observe that for any $U\in C^{1}(B^e\# H, B)$,
\begin{align*}
    \gamma_{\Gamma^{U}}(h,g)=&\Gamma^{U}(1\ot 1\# h, 1\ot 1\# g)\\
    =&U^{-1}(1\ot 1\# h\o t(U^{-1}(1\ot 1\# g\o)))\Gamma(1\ot 1\# h\th, 1\ot 1\# g\th)U((1\ot 1\# h\fo)(1\ot 1\# g\fo))\\
    =&U^{-1}(1\ot u^{-1}_{U}(g\o)\#h\o)\gamma_{\Gamma}(h\t,g\t)u_{U}(h\th\,g\th)\\
    =&(h\o\la u^{-1}_{U}(g\o))u^{-1}_{U}(h\t)\gamma_{\Gamma}(h\t,g\t)u_{U}(h\th\,g\th)=(\gamma_{\Gamma})^{u_{U}}(h,g)
\end{align*}
for all $h,g\in H$. \end{proof}

We will find similar results in Section~\ref{subsec, 2-cocycle} in our study of action Hopf algebroids, to which we now turn.

\section{Bisections and cocycles on action Hopf algebroids}\label{secact}

Classical action groupoids are an important class covered in the preliminaries. We first recall the  `quantum' version of this construction, which can be viewed as a Drinfeld-Yetter module version of the original construction of Lu in \cite{Lu} or a right hand version of the one given in \cite[Thm. 4.1]{BM}. Let $H$ be a Hopf algebra with invertible antipode and $B$ an object in the Drinfeld-Yetter category $\DYH$. This means that $B$ is  both a right $H$-module by an action $\ra$ and right $H$-comodule by a coaction $v\mapsto \zero{v}\tens v\o$, with the two compatible according to 
\begin{align}\label{equ. Drinfeld-Yetter equation}
    \zero{(b\triangleleft \two{h})}\ot\one{h}\one{(b\triangleleft \two{h})}=\zero{b}\ra\one{h}\ot\one{b}\two{h}
\end{align}
for all $b\in B$ and $h\in H$. This category is the dual or `centre' of the monoidal category $\CM_H$ of right $H$-modules and for finite-dimensional $H$ is just modules of the Drinfeld double $D(H)$, see \cite{Ma:rep,Ma:book}. (Here, $D(H)$ is generated by $H,H^{*op}$ and a module of the latter is being viewed as an $H$-comodule, which works better when $H$ is infinite dimensional.)  Moreover, we ask that $B$ is an algebra in this braided category (the structure maps are morphisms, i.e. $H$-linear and $H$-colinear) and that as such it is  braided commutative, which amounts to 
\[ ab=\zero{b}(a\ra\one{b})\]
for all $a, b\in B$. In this case, $B\# H^{op}$ (which is built on $B\ot H$ as a vector space) is a left bialgebroid as follows. We can first define a $B$-coring structure on $B\# H^{op}$, where the source and target maps are given by
\begin{align}\label{equ. source and target map of action Hopf algebroid}
    s(a)=a\ot 1,\quad t(a)=\zero{a}\ot\one{a}.
\end{align}
Then the coproduct, counit and product are
\begin{align}\label{equ. coproduct and counit of action Hopf algebroid}
 \Delta(a\# h)=a\#\one{h}\ot_{B} 1_{B}\#\two{h}\quad \varepsilon(a\#h)=a\varepsilon_{H}(h),\quad 
    (a\#h)(b\#g)=a(b\ra\one{h})\#g\two{h}
\end{align}
for all $a,b\in B$ and $h,g\in H$. The proofs are  similar to those in other versions and hence we omit them except to verify that the image of the coproduct belongs to Takeuchi product, as a useful check of conventions. Thus,
\begin{align*}
    a\#\one{h}\ot_{B} 1\#\two{h}s(b)=&a\#\one{h}\ot_{B} (b\ra\two{h})\#\three{h}\\
    =&t(b\ra\two{h})a\#\one{h}\ot_{B} 1\#\three{h}\\
    =&(\zero{(b\ra\two{h})}\#\one{(b\ra\two{h})})(a\#\one{h})\ot_{B} 1\#\three{h}\\
    =&\zero{(b\ra\two{h})}(a\ra\one{(b\ra\two{h})})\#\one{h}\two{(b\ra\two{h})}\ot_{B} 1\#\three{h}\\    =&a\zero{(b\ra\two{h})}\#\one{h}\one{(b\ra\two{h})}\ot_{B} 1\#\three{h}\\ 
       =&a(\zero{b}\ra\one{h})\#\one{b}\two{h}\ot_{B}1\#\three{h}= (a\#\one{h})t(b)\ot_{B} 1\#\two{h},
\end{align*}
where the penultimate step uses (\ref{equ. Drinfeld-Yetter equation}). It is known from \cite{BM} that action Hopf algebroids have an antipode in the sense of \cite{Lu}. Given this, \cite[Prop.~4.2]{BS} implies the following, which we give more explicitly. 

\begin{prop}\label{thm. left and anti-left Hopf algebroid on YD module}
    If $B$ is an algebra in $\DYH$ and braided commutative then the bialgebroid $B\# H^{op}$ is a left and anti-left Hopf algebroid.
\end{prop}
\begin{proof}
    For the left Hopf algebroid structure $\lambda^{-1}$, we  define
    \begin{align}\label{equ. left Hopf algebroid structure on smash product}
        (a\#h)_{+}\ot_{B^{op}}(a\#h)_{-}:=a\#\one{h}\ot_{B^{op}}1\#S^{-1}(\two{h})
    \end{align}
and check that
    \begin{align*}
        \lambda\circ\lambda^{-1}(a\#h\ot_{B}b\#g)
        =&a\#\one{h}\ot_{B}(1\#\two{h})(1\#S^{-1}(\three{h}))(b\#g)=a\#h\ot_{B}b\#g,\\
 \lambda^{-1}\circ\lambda(a\#h\ot_{B^{op}}b\#g)=&a\#\one{h}\ot_{B^{op}}(1\#S^{-1}(\two{h}))(1\#\three{h})(b\#g)=a\#h\ot_{B^{op}}b\#g.
\end{align*}
Similarly, for the anti-left Hopf algebroid structure, we have
\begin{align}\label{equ. anti-left Hopf algebroid structure on smash product}
        (a\#h)_{[-]}\ot^{B^{op}}(a\#h)_{[+]}=1\#\one{a}S\one{h}\ot^{B^{op}}\zero{a}\#\two{h}.
    \end{align}
\end{proof}

\subsection{Bisections on action Hopf algebroids}

We first introduce a group $Z^1_\ra(H,B)$ parallel to our treatment for quantum principal bundles, and then show that the group of bisections can be identified with this. 

\begin{prop}\label{thmZ} Let $B$ be a braided-commutative algebra in $\DYH$ and $Z^1_\ra(H,B)$ the set of unital linear maps $\varrho:H\to B$ such that the right non-Abelian cocycle condition
\[ \varrho(gh)=(\varrho(g)\ra \one{h})\varrho(\two{h})\]
holds for all $h,g\in H$ and such that the linear map $\psi: B\to B$ given by
\[
\psi(b):=\zero{b}\varrho(\one{b})
\]for all $b\in B$, is invertible. Then $Z^1_\ra(H,B)$ is a group with product and inverse
\[ (\varrho\bullet\varrho')(h)=\varrho(\one{h}) \psi\circ\varrho'(\two{h}),\quad 
\varrho^{-1}(h)=\psi^{-1}(\varrho(S\one{h})\ra\two{h})
\]
for all $h\in H$, and $\psi$ is an algebra automorphism of $B$ providing a group homomorphism $Z^1_\ra(H,B)\to \Aut(B)$.
\end{prop}
\begin{proof}
First,  $\psi$ is an algebra automorphism of $B$ since
\begin{align*}    \psi(ab)=&\zero{a}\zero{b}\varrho(\one{a}\one{b})=\zero{a}\zero{b}\varrho(\one{a})\triangleleft\one{b}\varrho(\two{b})=\zero{a}\varrho(\one{a})\zero{b}\varrho(\one{b})
    =\psi(a)\psi(b).
\end{align*}
By applying the Drinfeld-Yetter condition, we also observe that 
\begin{align}   \psi(a)\ra\one{h}\varrho(\two{h})=&\zero{a}\ra\one{h}\varrho(\one{a})\ra\two{h}\varrho(\three{h})=\zero{a}\ra\one{h}\varrho(\one{a}\two{h})\nonumber\\
=&\zero{(a\ra\two{h})}\varrho(\one{h})\ra\one{(a\ra\two{h})}\varrho(\two{(a\ra\two{h})})\nonumber\\
=&\varrho(\one{h})\zero{(a\ra\two{h})}\varrho(\one{(a\ra\two{h})})=\varrho(\one{h})\psi(a\ra\two{h}). \label{equ. varrho prop}
\end{align}

Next, we check that the stated product has the required properties. Clearly, it is associative with unit  provided by the counit of $H$ times $1_{B}$, and it is not hard to see that $\zero{b}(\varrho\bullet\varrho')(\one{b})=\psi\circ\psi'(b)$ since $\psi$ is an algebra map (where $\psi'$ is the map associated to $\varrho'$). Moreover,
\begin{align*}
   ((\varrho\bullet\varrho')(g))\ra\one{h}(\varrho\bullet\varrho')(\two{h})=&\varrho(\one{g})\ra\one{h}\psi(\varrho'(\two{g}))\ra\two{h}\varrho(\three{h})\psi\circ\varrho'(\four{h})\\
=&\varrho(\one{g})\ra\one{h}\varrho(\two{h})\psi(\varrho'(\two{g})\ra\three{h}\varrho'(\four{h}))\\
=&(\varrho\bullet\varrho')(gh)
\end{align*}
as required. Next, there is a convolution inverse $\varrho^{*}$ of $\varrho$ given by $\varrho^{*}(h)=\varrho(S\one{h})\ra\two{h}$, since 
\begin{align*}   
(\varrho(S\one{h})\ra\two{h})\varrho(\three{h})=&\varrho((S\one{h})\two{h})=\varepsilon(h)1_{B},\\ 
 \varrho(\one{h})\varrho(S\two{h})\ra\three{h}=&\big((\varrho(\one{h})\ra S\three{h})\varrho(S\two{h})\big)\ra\four{h}=\varrho(\one{h}S\two{h})\ra \three{h}=\varepsilon(h)1_{B}.
\end{align*}
It is then easy to see that $\varrho^{-1}=\psi^{-1}\circ\rho^*$ is the inverse of $\varrho$. We check that it has the required properties to land in $Z^1_\ra(H,B)$,
\begin{align*}
\zero{a}\varrho^{-1}(\one{a})=&\psi^{-1}(\psi(\zero{a})\varrho(S\one{a})\ra\two{a})=\psi^{-1}(\zero{a}\varrho(\one{a})\varrho(S\two{a})\ra\three{a})=\psi^{-1}(a),\\
 \varrho^{-1}(g)\ra\one{h}\varrho^{-1}(\two{h})=&\psi^{-1}\big(\psi(\varrho^{-1}(g)\ra\one{h})\varrho^{*}(\two{h})\big)=\psi^{-1}((\varrho(S\one{h})\varrho^{*}(g))\ra\two{h})=\varrho^{-1}(gh),
\end{align*}
where the second step in the second calculation uses (\ref{equ. varrho prop}).  \end{proof}

\begin{prop}\label{propZ} Let $B$ be a  braided-commutative algebra in $\DYH$. Then $\CB^R(B\# H^{op})\simeq Z^1_\ra(H,B)\simeq (\CB(B\# H^{op}))^{op}$ as groups.\end{prop}
\begin{proof}
    Let $\varsigma$ be a right bisection and define $\varrho:H\to B$ by
    $\varrho(h):=\varsigma(1\#h)$ for all $h\in H$. Clearly, $\varrho$ is unital since $\varsigma$ is, and  $\varsigma(b\#h)=\varsigma(s(b)(1\#h))=b\varsigma(1\#h)=b\varrho(h)$ for all  $b\#h\in B\# H^{op}$. Moreover,  $\psi(b)=\zero{b}\varrho(\one{b})=\varsigma(\zero{b}\#\one{b})=\varsigma\circ t (b)$ recovers the associated  automorphism of $B$. Finally, for  $X=1\#h$ and $Y=1\#g$, we have  that $\varsigma(XY)=\varrho(gh)$  equals $\varsigma(Xs\circ\varsigma(Y))=\varrho(g)\triangleleft\one{h}\varrho(\two{h})$, hence $\varrho\in Z^1_\ra(H,B)$  as in Proposition~\ref{thmZ}.

Conversely, given $\varrho:H\to B$ in $Z^1_\ra(H,B)$, define $\varsigma(b\#h):=b\varrho(h)$.
One can check that $\varsigma$ is left $B$-linear. From the construction, we know that $\varsigma\circ t$ is an automorphism since $\psi$ is. Moreover,  if $X=a\#h$ and $Y=b\#g$ then 
\begin{align*}
    \varsigma(XY)=a(b\triangleleft\one{h})\varrho(g\two{h})=\varsigma(Xs\circ\varsigma(Y))\end{align*}
from $\varrho(gh)=\varrho(g)\triangleleft\one{h}\varrho(\two{h})$. Next, adopting a shorthand  $c:=\psi^{-1}\varsigma(Y)=\psi^{-1}(b\varrho(g))$, 
\begin{align*}
    \varsigma(Xt\circ\psi^{-1}\varsigma(Y))=&a(\zero{c}\triangleleft\one{h})\varrho(\one{c}\two{h})=a\zero{(c\triangleleft\two{h})}\varrho(\one{h}\one{(c\triangleleft\two{h})})\\
    =&a\zero{(c\triangleleft\two{h})}\varrho(\one{h})\triangleleft\one{(c\triangleleft\two{h})}\varrho(\two{(c\triangleleft\two{h})})\\
    =&a\varrho(\one{h})\zero{(c\triangleleft\two{h})}\varrho(\two{(c\triangleleft\one{h})})\\
    =&a\varrho(\one{h})\psi(c\triangleleft\two{h})\\
    =&a\psi(c)\triangleleft\one{h}\varrho(\two{h})\\
    =&a(b\varrho(g))\triangleleft\one{h}\varrho(\two{h})=  \varsigma(Xs\circ\varsigma(Y)), 
\end{align*}
where the 6th step uses (\ref{equ. varrho prop}). This completes the one to one correspondence. Now let $\varsigma$ and $\varsigma'$ be two right bisections, corresponding to $\varrho$ and $\varrho'$ respectively by the construction above, and let $X=b\#h$. Then
\begin{align*}
    (\varsigma'\bullet\varsigma)(X)=&\varsigma'(t\circ\varsigma(\two{X})\one{X})\\
   =&\varsigma'(t\circ\varsigma(1\#\two{h})b\#\one{h})\\
    =&b\zero{\varrho(\two{h})}\varrho'(\one{h}\one{\varrho(\two{h})})\\
    =&b\zero{\varrho(\two{h})}\varrho'(\one{h})\triangleleft\one{\varrho(\two{h})}\varrho'(\two{\varrho(\two{h})})\\
    =&b\varrho'(\one{h})\zero{\varrho(\two{h})}\varrho'(\one{\varrho(\two{h})})\\
    =&b\varrho'(\one{h})\psi'(\varrho(\two{h}))= b(\varrho'\bullet\varrho)(h).
\end{align*}
Hence, the correspondence with $\CB^R(B\# H^{op})$ is as groups. 

Finally, given $\varrho\in Z^1_{\ra}(H, B)$, we construct a left bisection $\sigma(b\#h):=\psi^{-1}(b\varrho(h))$. This is well defined by using the conditions on $\varrho$ as above. On the other hand, given a left bisection $\sigma$, we define $\varrho:=\phi\circ\rho$, where $\phi(b):=\sigma(b\# 1)$, and $\rho(h):=\sigma(1\# h)$. The calculations are similar to the right case and give a group isomorphisms of $Z_{\ra}^{1}(H, B)$ with $\CB(B\# H^{op})^{op}$. \end{proof}

\subsection{Automorphisms of action Hopf algebroids}

Recall that an automorphism of a left bialgebroid over $B$ is a pair of algebra automorphisms $(\Phi, \phi)$ of $\cL$ and $B$ forming a commuting square with the source and target maps, and such that $\Phi$ is a coring map.

\begin{lem}\label{lem. automorphisms of YD module}
   Automorphisms of $B\#H^{op}$ are in one to one correspondence with pairs $(\phi, \nu)$, where $\phi: B\to B$ is an algebra automorphism and $\nu: H\to B\#H^{op}$ sending $h\mapsto \eins{\nu(h)}\#\zwei{\nu(h)}$ is an antialgebra map such that for all $b\in B, h\in H$,
       \begin{itemize}
    \item[(1)] $\nu(\one{h})\ot_{B}\nu(\two{h})=\eins{\nu(h)}\#\one{\zwei{\nu(h)}}\ot_{B}1\#\two{\zwei{\nu(h)}}$,  $\eins{\nu(h)}\varepsilon(\zwei{\nu(h)})=\varepsilon(h)$;
               \item[(2)]$(s\circ\phi)(\zero{b})\nu(\one{b})=\zero{\phi(b)}\#\one{\phi(b)}$;
        \item[(3)] $\nu(h)(s\circ\phi)(b)=(s\circ\phi)(b\triangleleft\one{h})\nu(\two{h})$, $\nu(h)(t\circ\phi)(b)=(t\circ\phi)(b\triangleleft\two{h})\nu(\one{h})$;
             \item [(4)]  $\Phi_{(\phi, \nu)}: B\#H^{op}\to B\#H^{op}$ defined by $\Phi_{(\phi, \nu)}(b\#h):=\phi(b)\nu(h)$  is invertible.
    \end{itemize}
\end{lem}
\begin{proof}
    Given an automorphism $(\Phi, \phi)$ of $B\#H^{op}$, one necessarily has
     $\Phi(b\#h)=\phi(b)\nu(h)$ for some map $\nu$. The map $\phi$ is already an algebra automorphism and $\nu$ is an antialgebra map since $\Phi$ is an algebra map and $(1\#h)(1\#g)=1\#gh$. It is a coring map as in (1) since $(\Phi, \phi)$ is, and $\Phi_{(\phi, \nu)}$ is invertible since $\Phi_{(\phi, \nu)}=\Phi$. Condition (2) is equivalent to $\Phi\circ t=t\circ \phi$. Next, by observing that $(1\#h)(b\#1)=(b\triangleleft\one{h}\#1)(1\#\two{h})$, we obtain the first equation of (3) since $\Phi$ is an algebra map. By the Drinfeld-Yetter condition (\ref{equ. Drinfeld-Yetter equation}), we have that $(1\#h)(\zero{b}\#\one{b})=(\zero{(b\triangleleft\two{h})}\#\one{(b\triangleleft\two{h})})(1\#\one{h})$ or equivalently, $(1\#h)t(b)=t(b\triangleleft\two{h})(1\#\one{h})$. By applying $\Phi$ on both sides, we obtain the second equation of (3).
Conversely, given $(\phi, \nu)$ with the conditions stated, it is not hard to see that $\Phi(b\#h):=\phi(b)\nu(h)$ is an algebra automorphism of $B\#H^{op}$ as well as the other conditions to form a left bialgebroid automorphism along with $\phi$. \end{proof}

One can also compute the actions $\Ad^L$ and $\Ad^R$ for the group crossed module structure for action Hopf algebroids.

\begin{lem}\label{lemactcross} Let $\varsigma$ with associated map $\psi$ be a right bisection of $B\# H^{op}$ corresponding to $\varrho$,  and  let $\varrho^{*}$ be the convolution inverse of $\varrho$. Then for $X=b\#h$, 
\begin{align}
    \Ad^{R}_{\varsigma}(X)=s\circ\psi(b)s\circ\varrho^{*}(\one{h})t\circ\varrho(\three{h})(1\#\two{h}),\quad \ad^R_\varsigma(b)=\psi(b).
\end{align}
If $\sigma$ is a left bisection corresponding to $\varrho$, so $\sigma(b\#h)=\psi^{-1}(b\varrho(h))$, then letting $\rho(h):=\psi^{-1}(\varrho(h))$ and $\phi:=\psi^{-1}$, we have
\begin{align}
 \Ad^{L}_{\sigma}(X)=s\circ\phi(b)s\circ\rho(\one{h})t\circ\rho^{*}(\three{h})(1\#\two{h}),\quad \ad^L_\sigma(b)=\phi(b).
\end{align}
\end{lem}
\begin{proof} By applying (\ref{equ. anti-left Hopf algebroid structure on smash product}) in (\ref{equ. right adjoint automorphism}) to $X=b\#h$,
\begin{align*}
    \Ad^{R}_{\varsigma}(X)=&t(\varsigma(\two{X}))\one{X}{}_{[+]}s(\sigma(\one{X}{}_{[-]}))\\
    =&(\zero{\varrho(\three{h})}\#\one{\varrho(\three{h})})(\zero{b}\#\two{h})(\varrho(\one{b}S\one{h})\#1)\\
    =&t(\varrho(\three{h}))(\zero{b}\varrho(\one{b})\varrho(S\one{h})\ra\two{h}\#\three{h})\\    =&s\circ\psi(b)s\circ\varrho^{*}(\one{h})t\circ\varrho(\three{h})(1\#\two{h}),\end{align*}
while $\ad^R_\sigma=b{}_{\scriptscriptstyle{(0)}}\varrho(b{}_{\scriptscriptstyle{(1)}})$.  Next, for $\sigma$ a left bisection corresponding to $(\phi,\rho)$ and applying (\ref{equ. left Hopf algebroid structure on smash product}) and Proposition~\ref{equ. left adjoint automorphism}, we have
\begin{align*}
    \Ad^{L}_{\sigma}(X)=&s(\sigma(\one{X}))\two{X}{}_{+}t(\sigma(\two{X}{}_{-}))\\
    =&(\phi(a)\rho(\one{h})\#1)(1\#\two{h})(\zero{\rho(S^{-1}\three{h})}\#\one{\rho(S^{-1}\three{h})})\\
    =&(\phi(a)\rho(\one{h})\#1)(\zero{\rho(S^{-1}\three{h})}\triangleleft\one{\two{h}}\#\one{\rho(S^{-1}\three{h})}\two{\two{h}})\\
    =&(\phi(a)\rho(\one{h})\#1)(\zero{(\rho(S^{-1}\four{h})\triangleleft\three{h})}\#\two{h}\one{(\rho(S^{-1}\four{h})\triangleleft\three{h})})\\    =&s\circ\phi(a)s\circ\rho(\one{h})t\circ\rho^{*}(\three{h})(1\#\two{h}),
\end{align*}
where we use (\ref{equ. Drinfeld-Yetter equation}) in the 4th step. For the 5th step, we use that the convolution inverse of $\rho$ is given by $\rho^{*}(h)=\rho(S^{-1}\two{h})\ra\one{h}$. We also immediately have $\ad^L_\sigma(a)=\phi(a)$. 
\end{proof}

In terms of the above isomorphisms, we obtain the following group crossed module. The map $\mu: Z^1_\ra(H,B) \to \Aut(B\# H^{op})$ sends   $\varrho$ to the automorphism given in the form of Lemma~\ref{lem. automorphisms of YD module}  by $(\phi_\varrho,\nu_\varrho)$   with
\[ \phi_\rho(a)=a{}_{\scriptscriptstyle{(0)}}\varrho(a{}_{\scriptscriptstyle{(1)}}),\quad \nu_\varrho(h)=(\varrho(Sh\o )\ra h\t )\zero{\varrho(h\fo)}\#h\th  \varrho(h\fo  )\o .\]
Given automorphism $(\Phi, \phi)$ on $B\# H^{op}$ in the form $(\phi,\nu)$,  its inverse $(\Phi^{-1}, \phi^{-1})$  corresponds to $(\phi^{-1}, \nu^{-1})$ by Lemma \ref{lem. automorphisms of YD module} for a map $\nu^{-1}:H\to B\# H^{op}$ characterised by
\begin{equation}\label{nuinv} \phi^{-1}(\eins{\nu(h)})\eins{\nu^{-1}(\zwei{\nu(h)})}\#\zwei{\nu^{-1}(\zwei{\nu(h)})}=1\#h=\phi(\eins{\nu^{-1}(h)})\eins{\nu(\zwei{\nu^{-1}(h)})}\#\zwei{\nu(\zwei{\nu^{-1}(h)})}\end{equation}
and obeying the same properties as listed for $\nu$. We then obtain the group crossed module action as
\begin{equation}\label{crossed module action of action groupoids}
 ((\phi,\nu)\la \varrho)(h)=\phi(\eins{\nu^{-1}(h)}\varrho(\zwei{\nu^{-1}(h)})). \end{equation}

\subsection{2-cocycles on action Hopf algebroids}\label{subsec, 2-cocycle}

We now turn to the non-Abelian cohomology for action Hopf algebroids. Similarly to Proposition~\ref{thmZ}, we give  concrete descriptions of $C^{1}(B\# H^{op}, B)$ and $Z^{2}(B\# H^{op}, B)$, making use of $Z^2_{\ra,as}(H,B)$ in Definition~\ref{def: assoc type}.  

\begin{prop}\label{lemC} Let $B$ be a braided-commutative algebra in $\DYH$. Then 

(1) $C^{1}(B\# H^{op}, B)\simeq C^1_{\ra, ver}(H,B)$ as groups, where the latter is the set of unital convolution invertible linear maps $\varrho:H\to B$ such that
\[ \zero{b}\varrho(hb\o )=\varrho(h)b,\qquad \varrho(h\o ) b\ra h\t =b\ra h\o \varrho(h\t )\]
holds for all $h\in H$ and $b\in B$, taken with  convolution product.  

(2) $Z^{2}(B\#H^{op}, B)$ is in one to one correspondence with the set $Z^2_\ra(H,B)$ consisting of $\gamma\in Z^2_{\ra, as}(H,B)$ such that 
\[ \gamma(h\o , g)(b\ra h\t )=\zero{b}\ra h\o \gamma(h\t , gb\o ),\qquad \zero{b}\gamma(hb\o , g)=\gamma(h, g)b\]
for all $h,g\in H$ and $b\in B$. 
\end{prop}

\begin{proof}
(1)  The proof is similar to those of Propositions~\ref{thmZ} and ~\ref{propZ}. One can check that the convolution product and inverse preserve the condition to be in $C^1_{\ra, ver}(H,B)$,  so this forms a group. The isomorphism between $C^1_{\ra, ver}(H,B)$ and $C^{1}(B\# H^{op}, B)$ has the same form as for Proposition~\ref{propZ}, where the first condition on $\varrho$ implies that $\varsigma_{\varrho}$ is right-$B$-linear and the second condition implies that  $\varsigma_{\varrho}(Xt(b))=\varsigma_{\varrho}(Xs(b))$.

(2)  If $\gamma\in C^1_{\ra, ver}(H,B)$ then we define $\Gamma_{\gamma}(a\# h, b\# g):=a(b\ra h\o )\gamma(h\t , g )$. This is $B$-bilinear by the second stated condition, and well defined over the balanced tensor product by the first stated condition and the Drinfeld-Yetter condition. Moreover, it is invertible since $\gamma$ is, with $\Gamma_{\gamma}^{-1}(a\#h, b\#g)=a(b\ra h\o )\gamma^{-1}(h\t , g\o )$. Next, let $X=a\#h, Y=b\#g, Z=c\#k$. Using that $\gamma$ is of  associative type, we have  
\begin{align*}\Gamma_{\gamma}(s\circ\Gamma_{\gamma}(X\o , Y\o )X\t Y\t , Z)&=a(b\ra h\o )(c\ra g\o h\t )\gamma(h\th  , g\t )\gamma(g\th  h\fo  , k),\\
\Gamma_{\gamma}(X, s\circ\Gamma_{\gamma}(Y\o , Z\o )Y\t Z\t )&=a(b\ra h\o )(c\ra g\o h\t )\gamma(g\t , k\o )\ra h\th  \gamma(h\fo  , k\t g\th  ),\end{align*}
so that $\Gamma_{\gamma}$ satisfies the 2-cocycle condition since $\gamma$ does. Finally, one  can check by direct computation that $\Gamma_{\gamma}(X, Ys(c))=\Gamma_{\gamma}(X, Y t(c))$. Conversely, given a 2-cocycle $\Gamma\in Z^{2}(\cL, B)$, we define $\gamma_{\Gamma}(h, g):=\Gamma(1\#h, 1\#g)$. Then $(\gamma,\ra)$ is  clearly a 2-cocycle on restricting the cocycle condition for $\Gamma$ to  $X=1\#h, Y=1\#g, Z=1\#k$, and of associative type on restricting to $X=1\#h, Y=1\#g, Z=b\#1$. The further stated conditions on $\gamma$ hold since $\Gamma$ is right $B$-linear and defined over the balanced tensor product. Moreover, these two constructions are inverse. Clearly, $\gamma_{\Gamma_{\gamma}}(h,g)=\gamma(h,g)$, while 
    \begin{align*}
        \Gamma_{\gamma_{\Gamma}}(a\#h, b\#g)=&a(b\ra h\o )\gamma_{\Gamma}(h\t , g)=a(b\ra h\o )\Gamma(1\#h\t , 1\#g)\\ =&\Gamma(a(b\ra h\o )\#h\t , 1\#g)  =\Gamma(a\#h, b\#g).
    \end{align*}
 \end{proof}

\begin{thm}\label{thm:actH2}
Let $B$ be a braided-commutative  algebra in $\DYH$. 

   (i)  For any $\varrho\in C^1_{\ra, ver}(H,B)$, we obtain a 2-cocycle $\partial\varrho\in Z^2_\ra(H,B)$ by
    \begin{align}
        \partial\varrho(h, g)=\varrho^{-1}(h\o )\varrho^{-1}(g\o )\ra h\t \varrho(g\t h\th  )
    \end{align}
for all $h,g\in H$. The subset $Z^1_{\ra, ver}(H,B)$ of elements $\varrho\in  C^1_{\ra, ver}(H,B)$ for which $\partial\varrho(h,g)=\eps(h)\eps(g)$ for all $h,g\in H$, is a subgroup isomorphic to $\CB_{ver}(B\# B^{op},B)^{op}$.

(ii)  The group $C^1_{\ra, ver}(H,B)$ acts on $Z^2_\ra(H,B)$ by 
\begin{align}\label{gammarho}
    \gamma^\varrho(h,g):=\varrho^{-1}(h\o )\varrho^{-1}(g\o )\ra h\t \gamma(h\t , g\t )\varrho(g\th  h\th  )
\end{align}
for $\gamma\in Z^2_\ra(H,B)$ and $\varrho\in C^1_{\ra, ver}(H,B)$. The set $\H_{\ra}^{2}(H, B)$ of orbits under this action can be identified with $\CH^{2}(B\# H^{op},B)$.
\end{thm}
\begin{proof}
We check that  $\partial\varrho$ obeys the conditions in Proposition~\ref{lemC}(2). For all $b\in B$ and $h, g\in H$, 
\begin{align*}
    b\zero{}\ra h\o \partial\varrho(h\t , gb\o )=&b\zero{}\ra h\o  \varrho^{-1}(h\t )\varrho^{-1}(g\o b\o )\ra h\th  \varrho(g\t b\t h\fo  )\\
    =&\varrho^{-1}(h\o )(b\zero{}\varrho^{-1}(g\o b\o ))\ra h\t \varrho(g\t b\t h\th  )\\
    =&\varrho^{-1}(h\o )(\varrho^{-1}(g\o )\ra h\t )(b\zero{}\ra h\th  )\varrho(g\t b\o h\fo  )\\
    =&\varrho^{-1}(h\o )(\varrho^{-1}(g\o )\ra h\t )(b\ra h\fo  )\zero{}\varrho(g\t h\th  (b\ra h\fo  )\o )\\
    =&\varrho^{-1}(h\o )\varrho^{-1}(g\o )\ra h\t \varrho(g\t h\th  )(b\ra h\fo  )=\partial\varrho(h\o , g)(b\ra h\t ),
\end{align*}
where the 2nd step uses the second condition of $C^1_{\ra, ver}(H,B)$ and the 3rd and 5th steps use the first condition of $C^1_{\ra, ver}(H,B)$. Similarly,
\begin{align*}
    b\zero{}\partial\varrho(hb\o , g)=&b\zero{}\varrho^{-1}(h\o b\o )(\varrho^{-1}(g\o )\ra (h\t b\t ))\varrho(g\t h\th  b\th  )\\
    =&\varrho^{-1}(h\o )b\zero{}(\varrho^{-1}(g\o )\ra (h\t b\o ))\varrho(g\t h\th  b\t )\\
    =&\varrho^{-1}(h\o )(\varrho^{-1}(g\o )\ra (h\t ))b\zero{}\varrho(g\t h\th  b\o )\\
    =&\varrho^{-1}(h\o )(\varrho^{-1}(g\o )\ra (h\t ))\varrho(g\t h\th  )b=\partial\varrho(h, g)b,
\end{align*}
where the 2nd and 4th steps use the first condition of $C^1_{\ra, ver}(H,B)$. It is straightforward to check that  $\partial\varrho$ is a 2-cocycle of associative type by using the second condition of $C^1_{\ra, ver}(H,B)$. That $\gamma^{\varrho}$ belongs to $Z^2_\ra(H,B)$ follows from the properties of $\gamma$. Next, for any $\varrho, \varrho'\in C^1_{\ra, ver}(H,B)$, we have
\begin{align*}
    \gamma^{\varrho\ast\varrho'}(h,g)=&\varrho'^{-1}(h\o )\,\varrho^{-1}(h\t )\,(\varrho'^{-1}(g\o)\ra h\th)\,(\varrho^{-1}(g\t)\ra h\fo)\, \gamma(h\fiv , g\th )\varrho(g\fo  h\si)\varrho'(g\fiv  h\sev) \\
   =&\varrho'^{-1}(h\o )\,(\varrho'^{-1}(g\o)\ra h\t)\,\varrho^{-1}(h\th )(\varrho^{-1}(g\t)\ra h\fo)\, \gamma(h\fiv , g\th )\varrho(g\fo  h\si)\varrho'(g\fiv  h\sev) \\
   =&(\gamma^{\varrho}){}^{\varrho'}(h,g)
\end{align*}
for all $h,g\in H$, so we have a group action. The second step uses  $\varrho^{-1}(h\o)\,(\varrho'^{-1}(g)\ra h\t)=(\varrho'^{-1}(g)\ra h\o)\,\varrho^{-1}(h\t)$, which can be derived from $\varsigma_{\varrho^{-1}}(1\#h \,s(\varsigma_{\varrho'^{-1}}(1\#g)))=\varsigma_{\varrho^{-1}}(1\#h\, t(\varsigma_{\varrho'^{-1}}(1\#g)))$ due to the correspondence between $C^{1}(B\# H^{op}, B)$ and $ C^1_{\ra, ver}(H,B)$ in Proposition~\ref{lemC}. Here, $\varsigma_{\varrho^{-1}}$  in $C^{1}(B\# H^{op}, B)$ obeys $\varsigma_{\varrho^{-1}}(1\#h)=\varrho^{-1}(h)$.

For the last part, for $X=1\#h, Y=1\#g$, $U\in C^{1}(B\#H^{op}, B)$ and $\Gamma\in Z^{2}(B\#H^{op}, B)$, 
\begin{align*}
   \gamma_{\Gamma^U}(h,g) =&\Gamma^U(X, Y)=U^{-1}(X\o t\circ U^{-1}(Y\o ))\Gamma(X\t , Y\t )U(X\th  Y\th  )\\
    =&\varrho_{U}^{-1}(h\o )(\varrho_{U}(g\o )\ra h\t )\gamma_{\Gamma}(h\th  , g\t )\varrho_{U}(g\th  h\fo  )=\gamma_{\Gamma}^{\varrho_{U}}(h,g),
\end{align*}
where $\varrho_{U}(h)=U(1\#h)$ and $\gamma_{\Gamma}(h, g)=\Gamma(X, Y)$. Hence, the non-Abelian cohomologies correspond as stated. \end{proof}

This gives us a concrete description of $\H^2(\CL,B)$ in the action Hopf algebroid case, which is our main result of the section.

\subsection{Weyl Hopf algebroids}

Here we see how our above constructions look in the special case of the canonical Weyl Hopf algebroid $H^*\# H^{op}$ associated to any finite-dimensional Hopf algebra $H$, which was a main example in \cite{Lu}. For an appropriate dual, this can also be extended to infinite-dimensional Hopf algebras $H$. 

\begin{rem} Before doing the case of interest, we note that there is another canonical example for any Hopf algebra $H$. Here $H\in \DYH$ in a well-known manner \cite{Ma:book} by
\begin{align}\label{canHcrossed}
   h\triangleleft g:=(S\one{g})h\two{g},\quad \Delta_R(h):=\one{h}\ot\two{h}
\end{align}
for all $h,g\in H$. That $H$ is braided commutative is immediate and as a result, $H\#H^{op}$ is a left and anti-left Hopf algebroid. It is easy to see that $H\#H^{op}\simeq \CL(H,k)=H\tens H^{op}$ where we regard  $H$ as a trivial quantum principal bundle with trivial coaction of the ground field $k$, so the base is again $H$. The standard algebra isomorphism here is $\Phi: \CL(H, k)\to H\#H^{op}$ given by
$\Phi(h\ot g)=h\one{g}\#\two{g}$ with inverse  $\Phi^{-1}(h\# g)=hS\one{g}\ot \two{g}$, cf \cite{Ma:book}, and one can check that it gives a Hopf algebroid isomorphism when taking the structures on the two sides. Hence this case is already covered by our results of Section~\ref{secesb}. In particular, $Z^1_\ra(H,H)\simeq \Aut(H)$, the group of algebra automorphisms of $H$, forming a group crossed module with $\mu=\id$ and the adjoint action of the group. \end{rem}

Now suppose (for simplicity) that $H$ is a finite-dimensional Hopf algebra. Then $H^*\in \DYH$ by
\begin{equation}\label{canH*crossed} a\ra h= \<a{}_{\scriptscriptstyle{(1)}},h\>a{}_{\scriptscriptstyle{(2)}},\quad   (\id \tens b)(\Delta_R a)= b{}_{\scriptscriptstyle{(1)}} a S b{}_{\scriptscriptstyle{(2)}}\end{equation}
for all $h\in H$ and $a,b\in H^*$. This is well-known in a dual form where $H$ acts on itself by multiplication and coacts by an adjoint coaction. For completeness, we check the compatibility
\begin{align*} (a\ra h{}_{\scriptscriptstyle{(2)}}){}_{\scriptscriptstyle{(0)}}\tens\<b, h{}_{\scriptscriptstyle{(1)}} (a\ra h{}_{\scriptscriptstyle{(2)}}){}_{\scriptscriptstyle{(1)}}\>&=b{}_{\scriptscriptstyle{(2)}}\<a{}_{\scriptscriptstyle{(1)}},h{}_{\scriptscriptstyle{(2)}}\> a{}_{\scriptscriptstyle{(2)}} S b{}_{\scriptscriptstyle{(3)}}\<b,h{}_{\scriptscriptstyle{(1)}}\>=b{}_{\scriptscriptstyle{(2)}} a{}_{\scriptscriptstyle{(2)}} S b{}_{\scriptscriptstyle{(3)}} \<b{}_{\scriptscriptstyle{(1)}} a{}_{\scriptscriptstyle{(1)}},h\>\\
&=(b{}_{\scriptscriptstyle{(1)}} a S b{}_{\scriptscriptstyle{(2)}}){}_{\scriptscriptstyle{(2)}} \<(b{}_{\scriptscriptstyle{(1)}} a S b{}_{\scriptscriptstyle{(2)}}){}_{\scriptscriptstyle{(1)}} b{}_{\scriptscriptstyle{(3)}},h\>\\ &=(b{}_{\scriptscriptstyle{(1)}} a S b{}_{\scriptscriptstyle{(2)}}){}_{\scriptscriptstyle{(2)}} \<b{}_{\scriptscriptstyle{(3)}},h{}_{\scriptscriptstyle{(2)}}\>\<(b{}_{\scriptscriptstyle{(1)}} a S b{}_{\scriptscriptstyle{(3)}}){}_{\scriptscriptstyle{(1)}} , h{}_{\scriptscriptstyle{(1)}}\>=a{}_{\scriptscriptstyle{(0)}}\ra h{}_{\scriptscriptstyle{(1)}}\tens \<b, a{}_{\scriptscriptstyle{(1)}} h{}_{\scriptscriptstyle{(2)}}\>.
\end{align*}
That $H^*$ is braided-commutative is
\begin{align*}\<b{}_{\scriptscriptstyle{(0)}} (a\ra b{}_{\scriptscriptstyle{(1)}}), h\>&=\<b{}_{\scriptscriptstyle{(0)}},h{}_{\scriptscriptstyle{(1)}}\>\<a\ra b{}_{\scriptscriptstyle{(1)}},h{}_{\scriptscriptstyle{(2)}}\>=\<b{}_{\scriptscriptstyle{(0)}},h{}_{\scriptscriptstyle{(1)}}\>\<a{}_{\scriptscriptstyle{(1)}},b{}_{\scriptscriptstyle{(1)}}\>\<a{}_{\scriptscriptstyle{(2)}},h{}_{\scriptscriptstyle{(2)}}\>\\
&=\<a{}_{\scriptscriptstyle{(1)}} b S a{}_{\scriptscriptstyle{(2)}}, h{}_{\scriptscriptstyle{(1)}}\>\<a{}_{\scriptscriptstyle{(3)}},h{}_{\scriptscriptstyle{(2)}}\>=\<ab,h\>.\end{align*}
We therefore have the {\em Weyl Hopf algebroid} $H^*\# H^{op}$ cf \cite{Lu}. It is also known \cite{Ma:book} that this is isomorphic as an algebra to a matrix algebra $M_n(k)$, where $n=\dim(H)$.

\begin{prop}\label{Zaut} $Z^1_\ra(H,H^*)$ can be identified with the group $G(H^*)$ of counital invertible elements of $H^*$.  For any $\alpha\in G(H^*)$, the corresponding  1-cocycle $\varrho_{\alpha}$ is given by 
\[\varrho_{\alpha}(h)=\<\alpha\o,h\> \alpha\t \alpha^{-1},\qquad \psi_{\alpha}(a)=\alpha a\alpha^{-1}\]
for any $h\in H$ and $a\in H^*$, where Proposition~\ref{thmZ} 
 identifies $Z^1_\ra(H,H^*)$ as pairs of maps $\varrho:H\to H^*$ and $\psi:H^*\to H^*$. \end{prop}
\begin{proof} The right 1-cocycle condition on $\rho$ for our given action is
\begin{align*} \varrho(gh)&=(\varrho(g)\ra h{}_{\scriptscriptstyle{(1)}}) \varrho(h{}_{\scriptscriptstyle{(2)}})=\<\varrho(g)\o,h{}_{\scriptscriptstyle{(1)}}\>  \varrho(g)\t \varrho(h{}_{\scriptscriptstyle{(2)}})\end{align*}
for all $h,g\in H$. Given a unital such 1-cocycle, we let $\alpha(h)=\eps(\varrho(h))$. This obeys $\eps(\alpha)=\alpha(1)=1$ and
\[\<\alpha\o,h\> \<\alpha\t \alpha^{-1},g\>= \alpha(hg\o)\alpha^{-1}(g\t)=\eps(\varrho(hg\o))\alpha^{-1}(g\t)=\<\varrho(h),g{}_{\scriptscriptstyle{(1)}}\>  \eps( \varrho(g{}_{\scriptscriptstyle{(2)}})) \alpha^{-1}(g\th)=\varrho(g)(h),\] so that the formula stated indeed recovers $\varrho$. Conversely, one can check that for any $\alpha$ as stated, we have a unital 1-cocycle. We also compute
\begin{align*} \<\psi(a),h\>&=\<a_0,h\o\>\<\varrho(a\o),h\t\>=\<a_0,h\o\>\<\alpha\o,a\o\>\<\alpha\t \alpha^{-1},h\t\>\\
&=\<\alpha\o a (S\alpha\t)\alpha\th\alpha^{-1},h\>=\<\alpha a \alpha^{-1},h\>\end{align*}
and then
\begin{align*} (\varrho\bullet\varrho')(h)&=\rho(h\o)\psi\circ\rho'(h\t)=\<\alpha\o,h\o\>\alpha\t\alpha^{-1}\<\beta\o,h\t\>\psi(\beta\t\beta^{-1})\\ &=
\<\alpha\o\beta\o,h\>\alpha\t\alpha^{-1}\alpha \beta\t\beta^{-1}\alpha^{-1}=\<(\alpha\beta)\o,h\>(\alpha\beta)\t(\alpha\beta)^{-1},\end{align*}
so that the group product corresponds to multiplication in $H^*$. It follows, but one can also check directly that
\begin{align*}\varrho^{-1}(h)&= \alpha^{-1} (\varrho(Sh\o)\ra h\t)\alpha =\alpha^{-1}\<\alpha\o,Sh\o\>\<(\alpha\t\alpha^{-1})\o,h\t\>(\alpha\t\alpha^{-1})\t\alpha\\
&=\alpha^{-1}\<\alpha\o,Sh\o\>\<\alpha\t\alpha^{-1}\o,h\t\>\alpha\th\alpha^{-1}\t\alpha =\<\alpha^{-1}\o,h\> \alpha^{-1}\t\alpha\end{align*}
has the same form with $\alpha^{-1}$ in place of $\alpha$. Clearly also $\eps(\psi(a))=\eps(\alpha\alpha^{-1})\eps(a)=\eps(a)$.
\end{proof}

For automorphisms, we similarly find:

\begin{prop}\label{Weylaut} The group
    $\Aut(H^{*}\# H^{op})$ can be identified with the group of counital algebra automorphisms of $H^*$.
\end{prop}
\begin{proof} First, it is convenient to use $\Hom(H, H^{*}\# H^{op})\simeq \Hom(H\ot H, H)$ identified in the obvious way. If $\tilde\nu(h)=\tilde\nu(h)\eins{}\ot\tilde\nu(h)\zwei{}$ denotes the original element of $\Hom(H, H^{*}\# H^{op})$, we can work equivalently with $\nu\in \Hom(H\ot H, H)$ given by $\nu(h \tens g)=\<\tilde\nu(h)\eins{},g\>\tilde\nu(h)\zwei{}$ for all $h,g\in H$. Conversely, we recover $\tilde\nu(h)\eins{}\ot \tilde\nu(h)\zwei{}=f^i\ot \nu(h\ot e_{i})$ using a basis and dual basis, where $f^i\ot e_{i}\in H^*\ot H$ obeys $\<f^i,h\>e_{i}=h$ for any $h\in H$.  We then use the specific action and coaction applicable in the case of the Weyl Hopf algebroid in the conditions for an automorphism in  Lemma \ref{lem. automorphisms of YD module}. A straightforward calculation, which we omit,  gives us that an automorphism of $H^{*}\# H^{op}$ then corresponds to $\phi: H^*\to H^*$ an algebra automorphism and $\nu:H\tens H\to H$ obeying 
\begin{align}\label{Weylaut1}
\nu(hg\ot f)&=\nu(h\ot\nu(g\ot f\o )\o f\t )\nu(g\ot f\o )\t ,\quad \nu(1\tens h)=\varepsilon(h)1_{H},\\\label{Weylaut2}
\Delta\circ \nu(h\ot g)&=\nu(h\o\tens g\o)g\t S g\fo  \ot\nu(h\t \tens g\th  ),\quad \varepsilon(\nu(h\ot g))=\varepsilon(hg),\\
\phi^{*}(g)\t \tens \nu&(\phi^{*}(g)\o S\phi^{*}(g)\th  \ot g\fo  )= \phi^{*}(g\t )\tens g\o Sg\th,\\ \label{Weylaut4}
\phi^{*}(\nu(h\ot g\o)\o& g\t )\tens \nu(h\ot g\o)\t = h\o\phi^{*}(g\o)\tens\nu(h\t \ot g\t ),\\ \label{Weylaut5}
 \phi^{*}(\nu(h\ot g\o)\t &g\th  )\tens \nu(h\ot g\o)\o g\t = h\t \phi^{*}(g\th  )\tens\nu(h\o\ot g\o)g\t \end{align}
for all $h,g,f\in H$.   
The first two conditions listed are equivalent to  $\nu$ in its original form an antialgebra map and coalgebra map.  The remaining three conditions listed are equivalent to the other conditions in Lemma~\ref{lem. automorphisms of YD module}, on noting that $\phi^*:H\to H$ is a coalgebra map defined by $\<\alpha,\phi^{*}(h)\>:=\<\phi(\alpha),h\>$ for any $\alpha\in H^*$ and $h\in H$. We note the left adjoint coaction appearing in some of the above  conditions.

Moreover,  let $(\phi^{-1},\tilde{\nu}^{-1})$ denote the inverse automorphism of $(\phi, \tilde{\nu})$. Then $\phi^{-1}$ is the inverse as an algebra automorphism and $\nu^{-1}$ (corresponding to $\tilde{\nu}^{-1}$) is `inverse' to $\nu$ in the sense 
\begin{equation}\label{Weylaut6}\nu(\nu^{-1}(h\ot \phi^{*}(g\o))\ot g\t )=\varepsilon(g)h=\nu^{-1}(\nu(h\ot \phi^{*-1}(g\o))\ot g\t ).\end{equation}
This follows on one side by  evaluating  the first factor (\ref{nuinv}) on $g\in H$, 
\begin{align*}
    \<\phi^{-1}(\eins{\tilde\nu(h)})&\eins{\tilde\nu^{-1}(\zwei{\tilde\nu(h)})}, g\>\zwei{\tilde\nu^{-1}(\zwei{\tilde\nu(h)})}\\
    &=\<\eins{\tilde\nu(h)}, \phi^{*-1}(g\o)\>\<\eins{\tilde\nu^{-1}(\zwei{\tilde\nu(h)})}, g\t \>\zwei{\tilde\nu^{-1}(\zwei{\tilde\nu(h)})}\\
    &=\<\eins{\tilde\nu^{-1}(\nu(h, \phi^{*-1}(g\o)))}, g\t \>\zwei{\tilde\nu^{-1}(\nu(h, \phi^{*-1}(g\o)))}\\
    &=\nu^{-1}(\nu(h\ot \phi^{*-1}(g\o))\ot g\t ),
\end{align*}
required to equal $\varepsilon(g)h$. Similarly for the other side. Now let $(\phi, \nu)\in \Aut(H^{*}\# H^{op})$ be as above and $\tilde{\phi}(h):=\nu(h\ot 1)$ for all $h\in H$. Then (\ref{Weylaut1}) and (\ref{Weylaut2})  imply that $\tilde{\phi}$ is a unital coalgebra map, and is an automorphism on using (\ref{Weylaut6}). Also, (\ref{Weylaut4}) or (\ref{Weylaut5}) imply that $\phi^{*}=\tilde{\phi}^{-1}$, so $\nu$ is determined from $\phi$. 
 
 Conversely, if $\phi$ is a counital algebra automorphism of $H^*$ then $\tilde{\phi}=\phi^*{}^{-1}$ is a unital coalgebra automorphism map. We then use  (\ref{Weylaut1})   with $f=1$ to recover $\nu$ from $\nu(h\tens 1)=\tilde\phi(h)$ as 
   \begin{align}
       \nu(h\ot g)=\tilde{\phi}(h\tilde{\phi}^{-1}(g\o))Sg\t .
   \end{align}
One can  check directly that $\nu$ satisfies the remaining conditions of (\ref{Weylaut1})-(\ref{Weylaut6}) using the properties of  $\tilde\phi$. Moreover,  $\nu$ is invertible in the sense of (\ref{Weylaut6}) on using that $\tilde{\phi}$ is invertible, 
\begin{align*}
    \nu(\nu^{-1}(h\ot \phi^{*}(g\o))\ot g\t )=& \nu(\nu^{-1}(h\ot \tilde{\phi}^{-1}(g\o))\ot g\t)=\nu(\tilde{\phi}^{-1}(hg\o)\tilde{\phi}^{-1}(Sg\t)\ot g\th)\\
    =&\tilde{\phi}(\tilde{\phi}^{-1}(hg\o)\tilde{\phi}^{-1}(Sg\t)\tilde{\phi}^{-1}(g\th))Sg\fo=\tilde{\phi}(\tilde{\phi}^{-1}(hg\o))Sg\t=\varepsilon(g)h,
\end{align*}
where the 4th step uses that $\tilde{\phi}$ is a coalgebra map. 
It is also clear that composition in $\Aut(H^*\# H^{op})$ corresponds to  composition of the underlying $\phi$.
\end{proof}

\begin{cor} \label{corZ1W} Consider elements of $\Aut(H^{*}\# H^{op})$ and $Z^1_\ra(H,H^*)$  given as above in terms of  corresponding counital algebra automorphisms $\psi$,  and counital invertible elements $\alpha\in H^*$ respectively. In these terms, the associated group crossed module map $\mu: Z^1_\ra(H,H^{*}) \to \Aut(H^{*}\# H^{op})$ sends $\alpha\in H^*$ to  $\psi$ given by $\psi(a)=\alpha a\alpha^{-1}$ for all $a\in H^*$, while the action of   $\phi$ on $\alpha$ is $\phi\la\alpha=\phi(\alpha)$.
\end{cor}
\begin{proof} We let $\varrho\in Z^1_\ra(H,H^*)$ correspond to $\alpha$. The image under $\mu$ from Lemma~\ref{lemactcross} is the automorphism $\Psi_\varrho(a\# h)=s\circ\psi(a)s\circ\varrho^{*}(\one{h})t\circ\varrho(\three{h})(1\#\two{h})$, which corresponds to $\psi$ and
\begin{align*}\nu(h\tens g)&=\<\varrho^{*}(h\o),g\o\>\<\varrho(h\th  ), g\th  \>h\t g\t Sg\fo \\
&=\alpha(g\o)\alpha^{-1}(h\o g\t )\alpha(h\th  g\fo  )\alpha^{-1}(g\fiv )h\t g\th  Sg\si, \end{align*}
where we used an expression for  $\varrho^{-1}$ to obtain $\<\varrho^{*}(h),g\>=\<\alpha, g\o\>\<\alpha^{-1},h g\t \>$ for the second step.  Setting $g=1$, the associated unital coalgebra automorphism is $\nu(h\tens 1)=\alpha^{-1}(h\o)h\t \alpha(h\th  )$, which is indeed $\psi^*{}^{-1}$ as required in Proposition~\ref{Weylaut}. Similarly, from (\ref{crossed module action of action groupoids}), we compute the group crossed module action as
\begin{align*}
    ((\phi, \nu)\la \varrho)(h)(g)=&\<\eins{\tilde\nu^{-1}(h)\varrho(\zwei{\tilde\nu^{-1}(h)}), \phi^{*}(g)\>}=\varrho(\nu^{-1}(h\tens \phi^{*}(g\o)), \phi^{*}(g\t ))\\=&\alpha(\nu^{-1}(h\tens \phi^{*}(g\o))\phi^{*}(g\t ))\alpha^{-1}(\phi^{*}(g\th  )), 
\end{align*}
where we use that $\phi^{*}$ is a coalgebra map. To find the corresponding transformed $\alpha$, we set $g=1$ and $\nu^{-1}(h\tens 1)=\tilde\phi^{-1}(h)=\phi^*(h)$ from Proposition~\ref{Weylaut}.  \end{proof}

Next, it is convenient to write  $\varrho\in C^1_{\ra, ver}(H,H^*)$ as $\rho(h, g):=\varrho(h)(g)$ so that it becomes identified with the convolution invertible maps $\rho: H\ot H\to k$ such that
\[
   \rho(1, g)=\varepsilon(g),\quad \rho(h\o, g\o)h\t g\t =h\o g\o\rho(h\t , g\t ),\quad
    \rho(hg\o Sg\th, g\fo  ) g\t=\rho(h, g\o)g\t\]
for all $h,g\in H$. Moreover, there is a subgroup inclusion $i: Z^1_{\ra, ver}(H,H^*)\hookrightarrow C^1_{\ra, ver}(H,H^*)$ given by
\begin{equation}\label{ialpha}
    i(\alpha)=\alpha\o\ot \alpha\t \alpha^{-1}.\end{equation}

\begin{prop}\label{propweylZ2}  $Z^2_{\ra}(H,H^*)$ can be identified with convolution-invertible maps  $\delta:H\tens H\to k$ that respect the unit in the sense $\delta(h,1)=\varepsilon(h)1=\delta(1,h)$ and obey
\begin{align*}
  & (1)\quad  \delta(h\o, g\o)g\t h\t =g\o h\o\delta(h\t , g\t ); \\
 & (2)\quad  \delta(f\o, h\o)\delta(h\t f\t , g\o)\delta^{-1}(f\th  , g\t h\th  )h\fo  f\fo  \\
 &\qquad=\delta(f\si  , h\si  )\delta(h\sev f\sev , g\o h\o f\o (Sf\fiv )(Sh\fiv  ))\delta^{-1}(f\eig , g\t h\t f\t (Sf\fo  )(Sh\fo  )h\eig )h\th  f\th ;\\
 & (3)\quad \delta(f\o , h\o )\delta(h\t f\t , g\o )\delta^{-1}(f\th  , g\t h\th  )f\fo \\
 &\qquad=\delta(f\eig , h\o f\o Sf\sev )\delta(h\t f\t (Sf\si)  f\nin , g\o )\delta^{-1}(f\ten , g\t h\th  f\th  (Sf\fiv  ))f\fo\end{align*}
 for all $f,g,h\in H$. The corresponding  2-cocycle is
\[ \gamma_{\delta}(g, h)=\<\delta^{2}\delta^{1'}{}\o \delta^{-2}{}\t ,g\>\<\delta^{2'}\delta^{-2}{}\o ,h\>\delta^{1}\delta^{1'}{}\t \delta^{-1},\]
 where $\delta=\delta^1\tens\delta^2=\delta^{1'}\tens\delta^{2'}$ with inverse $\delta^{-1}\ot\delta^{-2}$  are viewed in $H^*\tens H^*$.
\end{prop}
\begin{proof}
    Given $\gamma\in Z^2_\ra(H,H^*)$, we define $\delta_{\gamma}(h, g):=\varepsilon(\gamma(h, g))$. One can check that $\delta_{\gamma}$ is unital and convolution invertible since $\gamma$ is. Using that $\gamma$ is of associative type, one has
      \[\<\gamma(h\o , g\o ), f\o \>g\t h\t f\t =g\o h\o f\o \<\gamma(h\t , g\t ), f\t \>,\]
   which implies that $\delta_{\gamma}(h\o , g\o )g\t h\t =g\o h\o \delta_{\gamma}(h\t , g\t )$ on setting $f=1$. Conditions (2) and (3) for $\delta_{\gamma}$ can then be derived from the properties of $\gamma$ in Proposition~\ref{lemC}.

    Conversely, given $\delta$ as stated, $\gamma_{\delta}$ is clearly counital and convolution invertible, with inverse given by
   \[ \gamma^{-1}_{\delta}(g, h)=\<\delta^{2}{}\t \gamma^{-1'}{}\o \gamma^{-2},g\>\<\delta^{2}{}\o \gamma^{-2'},h\>\delta^{1}\gamma^{-1'}{}\t \gamma^{-1}.\]
    It is not hard to check that $\<\gamma_{\delta}(h\o , g\o ), f\o \>g\t h\t f\t =g\o h\o f\o \<\gamma_{\delta}(h\t , g\t ), f\t \>$,  from which it follows that $\gamma_\delta$ is of associative type. Next, the 2-cocycle condition  is equivalent to 
    \begin{align}\label{equ. two cocycle on HH*}
        \<\gamma_{\delta}(g\o , f\o ), h\o k\o  \>\<\gamma_{\delta}(h\t , f\t g\t ), k\t \>=\<\gamma_{\delta}(h\o , h\o ), k\o  \>\<\gamma_{\delta}(g\t h\t , f), k\t \>,
    \end{align}
with the left hand side equal to
\[ \delta(h\o k\o , g\o )\delta(g\t h\t k\t , f\o )\delta^{-1}(h\th  k\th  , f\t g\th  )\delta(k\fo  , h\fo  )\delta(h\fiv  k\fiv  , f\th  g\fo  )\delta^{-1}(k\si  , f\fo  g\fiv  h\si  ).\]
The right hand side is equal to $\delta(k\o , h\o )\delta(h\t k\t , g\o )\delta(g\t h\th  k\th  , f\o )\delta^{-1}(k\fo  , f\t g\th  h\fo  )$ and the two are equal by property (1). Finally, $\gamma_{\delta}$ satisfies the rest of the conditions in Proposition~\ref{lemC} by properties (2) and (3). These constructions are inverse. Thus,
\begin{align*}
\<\gamma_{\delta_{\gamma}}(g, f), h\>=\varepsilon(\gamma(h\o , g\o )\gamma(g\t h\t , f\o )\gamma^{-1}(h\th  , f\t g\th  ))=\<\gamma(g, f), h\>,
\end{align*}
where the last step uses that $\gamma$ satisfies (\ref{equ. two cocycle on HH*}) with $k=1$. The other direction $\delta_{\gamma_{\delta}}(g, f)=\varepsilon(\gamma_{\delta}(g, f))=\delta(g, f)$ is clear. 
\end{proof}

We then have

\begin{cor}
    For any $\varrho\in C^1_{\ra, ver}(H,H^*)$, we obtain a 2-cocycle $\partial\varrho\in Z^2_\ra(H,H^*)$ by
    \begin{align}
        \partial\varrho(h, g)=\rho^{-1}(h\o , 1)\rho^{-1}(g\o , h\t )\rho(g\t h\th  , 1)
    \end{align}
for all $h,g\in H$, where $\rho(h,g)=\varrho(h)(g)$. Moreover, if  $\delta\in Z^2_\ra(H,H^*)$ then so is 
\begin{align}
    \delta^\varrho(h, g)=\rho^{-1}(h\o , 1)\rho^{-1}(g\o , h\t )\delta(h\th  , g\t )\rho(g\th  h\fo  , 1),
\end{align}
giving a group action of $C^1_{\ra, ver}(H,H^*)$ with the set of orbits $\CH^2_\ra(H,H^*)$.\end{cor}
\proof This follows by Theorem~\ref{thm:actH2} and the above identifications. \endproof

Note that the way we have presented it, $Z^1_{\ra,ver}(H,H^*)$ sits inside $C^1_{\ra,ver}(H,H^*)$ via the map $i$ in (\ref{ialpha}) and  by  Theorem~\ref{thm: boundary map},  exactly such $i(\alpha)$ are trivial under $\partial$ in the sense $(\partial i(\alpha))(h,g)=\eps(hg)$. The simplest case is  the group Hopf algebra $H=k G$ for a finite group $G$ acting on its functions $k(G)$ by left translation and coacting trivially to give the Weyl Hopf algebroid $k(G)\# k G^{op}$. In this case, all bisections are vertical by Corollary~\ref{corZ1W} and $C^1_{\ra,ver}(k G, k(G))$ consists of invertible functions $\rho\in k(G\times G)$ that are 1 when the first argument is the group identity. The set $Z^2_\ra(k G,k(G))$ consists of invertible such functions $\delta$ that are 1 when either argument is the identity.  The transformed $\delta^\rho$ is $\delta^\rho(h,g)=\delta(h,g)\rho(gh,1)/(\rho(h,1)\rho(g,h))$. The cohomology $\CH^2(k G,k(G))$ is thus rather different from multiplicative group cohomology. 

The Weyl Hopf algebroid example extends to any $k(X)\# k G^{op}$ for a finite group $G$ left acting on a set $X$, regarding $B=k(X)$  with trivial coaction (given that $k G$ is cocommutative). This extends in principle to an algebraic setting $\C[X]\# U(\mathfrak g)^{op}$ where a Lie algebra acts on a commutative coordinate algebra by vector fields,  with appropriate care for the algebraic version.  Based on this experience with action Hopf algebroids, one can expect that if  $M$ is a smooth manifold and $D(M)$ the Hopf algebroid of differential operators \cite{Xu} on $M$ then its bisections  should be similarly expressible and the associated cohomology of interest. This will be looked at elsewhere.

\subsection{Action Hopf algebroid associated to a coquasitriangular Hopf algebra}\label{secquasi}

Recall \cite{Ma:book,Ma:bg} that a coquasitriangular Hopf algebra is a Hopf algebra $H$ together with a convolution invertible linear map $\R: H\ot H\to k$, such that
\[    \R(h\ot gf)=\R(\one{h}\ot f)\R(\two{h}\ot g),\quad \R(hg,f)=\R(h, \one{f})\R(g,\two{f}), \]
 \[  \one{g}\one{h}\R(\two{h}\ot\two{g})=\R(\one{h}\ot\one{g})\two{h}\two{g}\]
for all $f,g,h\in H$. We also recall that if $H$ is a coquasitriangular Hopf algebra, we can define a new associative algebra on the same underlying vector space $H$ (denoted by $\udh$) by \cite{Ma:bg},
\begin{align}
    h\bullet g=\two{h}\three{g}\R(\three{h}\ot S\one{g})\R(\one{h}\ot\two{g}).
\end{align}
This is part of the structure of a braided Hopf algebra in the category of right $H$-comodules, where $H$ coacts by $\Ad_R(h)=h{}_{\scriptscriptstyle{(2)}} \tens (S h{}_{\scriptscriptstyle{(1)}})h{}_{\scriptscriptstyle{(3)}}$. It is also known \cite{BegMa,Ma:hod} that $\underline{H}$ is braided-commutative with respect to the braiding of $\DYH$ when viewed as an object by coaction  $\Ad_R$ and  action
\begin{align}
    h\ra g:=\two{h}\R(\one{g}\ot\one{h})\R(\three{h}\ot\two{g})
\end{align}
for all $h,g\in H$.

\begin{thm}
    Let $H$ be a coquasitriangular Hopf algebra and $(H, \udh, \adr, \ra)$ be the data as above. Then $\UH$ is a braided-commutative algebra in $\DYH$. Hence we have an action Hopf algebroid $\underline{H}\# H^{op}$.
\end{thm}
\proof Most of this is known, we just have verify that $\underline H$ is indeed an $H$-module algebra (its structure maps are already $H$-colinear as part of the braided-Hopf algebra structure). We check,
\begin{align*}
    (h\bullet g)\ra f&=\three{h}\four{g}\R(\one{f}\ot \two{h}\three{g}\R(\one{h}\ot\two{g}))\R(\four{h}\ot\two{f})\R(\five{g}\ot\three{f})\R(\five{h}\ot S\one{g})\\
    &=\three{h}\four{g}\R(\one{f}\ot\one{h})\R(\two{f}\ot\two{g
    })\R(\two{h}\ot\three{g})\R(\four{h}\ot\three{f})\R(\five{g}\ot\four{f})\R(\five{h}\ot S\one{g}),\\
    (h\ra\one{f})\bullet&(g\ra\two{f})\\
    &=\three{h}\four{g}\R(\four{h}\ot S\two{g})\R(\two{h}\ot\three{g})\R(\one{f}\ot\one{h})\R(\five{h}\ot\two{f})\R(\three{f}\ot\one{g})\R(\five{g}\ot\four{f})
\end{align*}
for all $f,g,h\in H$. The two expressions are equal using that
\begin{align*}
  \R(\two{h}\ot & S\one{g})\R(\one{f}\ot\two{g})\R(\one{h}\ot\two{f})=  \R(h\ot (S\one{g})\two{f}\R(\one{f}\ot\two{g}))\\
  &=\R(h\ot \two{(Sg)}\two{f}\R(\one{f}\ot S^{-1}(\one{(Sg)})))=\R(h\ot \two{(Sg)}\two{f}\R^{-1}(\one{f}\ot \one{(Sg)}))\\
    &=\R(h\ot \one{f} \one{(Sg)}\R^{-1}(\two{f}\ot \two{(Sg)}))=\R(h\ot \one{f} \one{(Sg)}\R(\two{f}\ot S^{-1}(\two{(Sg)})))\\
    &=\R(h\ot \one{f} (S\two{g})\R(\two{f}\ot \one{g}))=\R(h\ot \one{f} (S\two{g})\R(\two{f}\ot \one{g}))\\
    &=\R(\one{h}\ot S\two{g})\R(\two{f}\ot\one{g})\R(\two{h}\ot\one{f}).
\end{align*}
 \endproof

Next, for any coquasitriangular Hopf algebra $H$, one has a {\em quantum Killing form}  $\Q:H\to H^*$ given by $\Q(h)(g)=\R(g{}_{\scriptscriptstyle{(1)}} \tens h{}_{\scriptscriptstyle{(1)}})\R(h{}_{\scriptscriptstyle{(2)}} \tens g{}_{\scriptscriptstyle{(2)}})$ for all $h,g\in H$, see \cite{Ma:book}.  In the finite-dimensional case, $H$ is said to be {\em factorisable} if $\Q$ is a linear isomorphism.

\begin{prop} If $H$ is finite-dimensional, there is a morphism of Hopf algebroids  $\underline{H}\# H^{op}\to H^*\# H^{op}$ to the Weyl Hopf algebroid induced by $\Q$ on the base algebra. This is an isomorphism if $H$ is factorisable.
\end{prop}
\proof It is already known \cite{Ma:book} that $\Q:\underline H\to H^*$ is an algebra map (as part of a braided Hopf algebra map where $H^*$ has a transmuted coproduct) and is a comodule map, where we use $\Ad_R$ on the left and the coaction dual to the left adjoint action of $H^*$ on $H^*$. Thus, we only need to check that this map is a module map for the $H$ action part of the Drinfeld-Yetter module. Thus,
\begin{align*}\Q(f\ra h)(g)&=\Q(f\ra h\tens g)=\Q(f{}_{\scriptscriptstyle{(2)}}\tens g)\R(h{}_{\scriptscriptstyle{(1)}}\tens f{}_{\scriptscriptstyle{(1)}})\R(f{}_{\scriptscriptstyle{(3)}}\tens h{}_{\scriptscriptstyle{(2)}})\\
&=\R(g{}_{\scriptscriptstyle{(1)}}\tens f{}_{\scriptscriptstyle{(2)}})\R(f{}_{\scriptscriptstyle{(3)}}\tens g{}_{\scriptscriptstyle{(2)}})\R(h{}_{\scriptscriptstyle{(1)}}\tens f{}_{\scriptscriptstyle{(1)}})\R(f{}_{\scriptscriptstyle{(4)}}\tens h{}_{\scriptscriptstyle{(2)}})\\
&=\R(h{}_{\scriptscriptstyle{(1)}} g{}_{\scriptscriptstyle{(1)}}\tens f{}_{\scriptscriptstyle{(1)}})\R(f{}_{\scriptscriptstyle{(2)}}\tens h{}_{\scriptscriptstyle{(2)}} g{}_{\scriptscriptstyle{(2)}})\\
&=\Q(f\tens hg)=\<\Q(f){}_{\scriptscriptstyle{(1)}},h\>\<\Q(f){}_{\scriptscriptstyle{(2)}},g\>=\<\Q(f)\ra h,g\>
\end{align*}
for all $f,g,h\in H$. Hence $\Q$ is a map of braided-commutative algebras in $\DYH$. The resulting action Hopf algebroids on the two sides depend only on this and hence are isomorphic in the factorisable case. Explicitly $\Q\tens\id$ connects the two braided-Hopf algebras. \endproof

\begin{cor}Let  $G(\underline H)$ be the group of counital invertible elements  of $\underline{H}$. There is a group homomorphism $\varrho: G(\underline H)\to Z^1_\ra(H, \underline{H})$ sending $\beta\in G(\underline H)$ to $\varrho_{\beta}\in Z^1_\ra(H, \underline{H})$ given by  $\varrho_{\beta}(h)=(\beta\ra h)\beta^{-1}$. This is a  group  isomorphism if $H$ is factorisable.\end{cor}
\begin{proof}
First we check that $\varrho$ gives a well-defined bisection,
\begin{align*}
    (\varrho_{\beta}(h)\ra g\o ) \varrho_{\beta}(g\t )=(((\beta\ra h)\beta^{-1})\ra g\o ) (\beta\ra g\t  )\beta^{-1}=((\beta\ra h)\ra g)\beta^{-1}=\varrho_{\beta}(hg).
\end{align*}
Moreover,  $\psi_{\beta}: f\mapsto \zero{f}\varrho_{\beta}(f\o )$ is invertible with inverse  $\psi_{\beta^{-1}}$. Thus,
\begin{align*}
    \psi_{\beta}\circ\psi_{\beta^{-1}}(f)=&\psi_{\beta}(\zero{f}\varrho_{\beta^{-1}}(f\o ))=(\zero{f}\zero{(\beta^{-1}\ra f\t )}\zero{\beta})(\beta\ra(f\o (\beta^{-1}\ra f\t )\o \beta\o ))\beta^{-1}\\=&(\zero{f}(\beta\ra f\o )(\beta^{-1}\ra f\t )\beta)\beta^{-1}=f
\end{align*}
for all $f\in \underline{H}$, and similarly on the other side.  Moreover, if $\alpha$ is another counital invertible element of $\underline{H}$ then
\begin{align*}
   ( \varrho_{\alpha}\bullet \varrho_{\beta})(h)=&\varrho_{\alpha}(h\o )\psi_{\alpha}( \varrho_{\beta}(h\t ))
    =(\alpha\ra h\o )\alpha^{-1}(\zero{(\beta\ra h\t )\beta^{-1})}(\alpha\ra ((\beta\ra h\t )\beta^{-1})\o )\alpha^{-1}\\
    =&(\alpha\ra h\o )(\beta\ra h\t )\beta^{-1}\alpha^{-1}=\varrho_{\alpha\beta}(h)
\end{align*}
so that $\varrho: \beta\mapsto \varrho_{\beta}$ is a group homomorphism.

Next, we show in the factorisable case that $\varrho$ is an isomorphism, by showing that the  diagram 
\[
\begin{tikzcd}
  &\{\alpha\in \udh\ |\ \alpha\,\,\, \textup{is counital and invertible}\} \arrow[d, "\Q"] \arrow[r, "\varrho"] & Z^1_\ra(H,\udh) \arrow[d, "\Q_{*}"] \\
   & \{\beta\in H^{*}\ |\ \beta\,\,\, \textup{is counital and invertible}\}  \arrow[r, "\Psi"] & Z^1_\ra(H,H^*)
\end{tikzcd}
\]
commutes and that the bottom and vertical maps are isomorphisms. Here, $\Psi$ is an isomorphism by Proposition \ref{Zaut}. The left map is the restriction of $\Q$ to the stated subset of $\udh$. Given a  counital and invertible $\alpha\in \udh$, its image $\Q(\alpha)$ is also invertible since  $\Q$ is an algebra map. Moreover, $\Q(\alpha)$ is counital since $\varepsilon(\Q(\alpha))=\Q(\alpha)(1)=\CR(1\ot \alpha\o)\CR(\alpha\t\ot 1)=\varepsilon(\alpha)=1$. So the left map is an isomorphism. The map $\Q_{*}:Z^1_\ra(H,\udh)\to Z^1_\ra(H,H^*)$ is composition by $\Q$, hence is an isomorphism. It remains to show that the diagram commutes. Thus, for $\alpha$ a counital and invertible element in $\udh$, we have
\begin{align*}
    \Psi^{-1}(\Q_{*}(\varrho_{\alpha}))(h)=&\varepsilon((\Q_{*}(\varrho_{\alpha}))(h))=\varepsilon(\Q(\varrho_{\alpha}(h)))=\varepsilon(\Q((\alpha\ra h)\alpha^{-1}))\\
    =&\varepsilon(\Q(\alpha\t\,\CR(h\o\ot \alpha\o)\CR(\alpha\th\ot h\t)\alpha^{-1}))\\
    =&\Q(\alpha\t\,\CR(h\o\ot \alpha\o)\CR(\alpha\th\ot h\t)\alpha^{-1})(1)\\
    =&\CR(h\o\ot \alpha\o)\CR(\alpha\t\ot h\t)=\Q(\alpha)(h),
\end{align*}
where the first step is given by Proposition \ref{Zaut}.
\end{proof}

In the infinite-dimensional essentially factorisable case such as $\CO_q[SL_2]$, we should therefore view $B_q[SL_2]\# \CO_q[SL_2]^{op}$ as essentially the Weyl algebra $U_q(sl_2)\# \CO_q[SL_2]^{op}$ for generic $q$, but with very different classical limit as $q\to 1$. It also has the merit of not requiring a duality pairing as needed for the Weyl Hopf algebroid. To describe its structure explicitly, we use matrix generators $t^i{}_j$ for $\CO_q[SL_2]$ and $u^i{}_j$ for $B_q[SL_2]$, which is the standard description of $\underline{\CO_q[SL_2]}$, see \cite{Ma:book}. Then the cross relations are
\[ t^i{}_j u^k{}_l=     (u^k{}_l\ra t^i{}_m) t^m{}_j   =u^p{}_q  t^m{}_j\R(t^i{}_n\tens t^k{}_p)\R(t^q{}_l\tens t^n{}_m)=u^p{}_q  t^m{}_j R^i{}_n{}^k{}_pR^q{}_l{}^n{}_m \]
in terms of $R$-matrices defined as the value of $\R$ on the generators. We used that $u^i{}_j=t^i{}_j$ when the vector spaces of $B_q[SL_2]$ and $\CO_q[SL_2]$ are identified as part of the transmutation procedure \cite{Ma:book}. Note that the relevant coproduct on $u^i{}_j$ is the one retained from $H=\CO_q[SL_2]$, not the braided one. These relations can be written in compact form as
\[ {\bf t}{}_1{\bf u}_2= R {\bf u}_2 R_{21} {\bf t}_1,\]
where the numerical indices refer to the position of matrix indices in $M_n\tens M_n$ (with values in the algebra). In the same notation, the quadratic relations of the two subalgebras appear as $R{\bf t}_2{\bf t}_1={\bf t}_1{\bf t}_2 R$ for $\CO_q[SL_2]^{op}$ and $R_{21}{\bf u}_1 R{\bf u}_2={\bf u}_2 R_{21}{\bf u}_1 R$ for $B_q[SL_2]$, see \cite{Ma:book}. (There are also $q$-determinant relations on both factors.) The source and target maps are
\begin{align}
    s(u^i{}_j)=u^i{}_j\#1, \qquad t(u^i{}_j)=u^k{}_l\# (St^i{}_k)t^l{}_j
\end{align}
and the coproduct is
\begin{align}
    \Delta({\bf u}_1\#{\bf t}_2)={\bf u}_1\#{\bf t}_2\ot_{B_q[SL_2]}1\#{\bf t}_2.
\end{align}
The precise computation of the bisections and cohomologies for this case of $q$-deformation quantum groups is deferred to further work.

\section{\bf Dual constructions }\label{secdual}

To round off the paper, we construct dual versions of some of the main results above, starting with a preliminary section on duality in the left-finite or right-finite case.

\subsection{Duality of bialgebroids} Following  \cite{schau1}, a left bialgebroid $\cL$ is called left-finite if it is finitely generated projective as a right $B$-module induced by the target map on the left. This means that there are
$x_{i}\in \cL$ and $x^{i}\in \cL^{\vee}:=\Hom_{B}(\cL, B)$ (the collection of right $B$ module map in the sense of (\ref{eq:rbgd.bimod})) such that  $X=\sum_i t(x^{i}(X))x_{i}$  for all $X\in \cL$.  Let $\cL$ be such a left-finite left bialgeboid over $B$. Then $\cL^{\vee}$ is a left bialgebroid over $B$ with product  given by
\begin{align}\label{equ. product on the right dual bialgebroid}
    (\sigma\eta)(X):=\sigma(s(\eta(\one{X}))\two{X})
\end{align}
for all $X\in \cL$ and $\sigma$, $\eta\in \cL^{\vee}$.  We omit writing $\ast$ explicitly for this convolution-type product. The unit is just the counit of $\cL$. The source and target maps are given by 
\begin{align}\label{equ. source and target map on the right dual bialgebroid}
     s(b)(X):=b\varepsilon(X),\quad t(b)(X):=\varepsilon(Xt(b))=\varepsilon(Xs(b))
  \end{align}
for $X\in \cL$ and $b\in B$.
Before we describe the coproduct, we first observe that since $\cL$ is left-finite, there is an isomorphism $\phi:  \cL^{\vee}\ot_{B}\cL^{\vee}\to \Hom_{B}(\cL\ot_{B^{op}}\cL, B)$ given by
\begin{align}\label{equ. isomorphism of right dual}
\phi(\sigma\ot_{B}\eta)(X\ot_{B^{op}}Y):=\sigma(Xt(\eta(Y))),
\end{align}
where $T\in\Hom_{B}(\cL\ot_{B^{op}}\cL, B)$ means  $T(t(b)X\ot_{B^{op}}Y)=T(X\ot_{B^{op}}Y)b$.
The inverse of this map is
\begin{align}\label{equ. inverse isomorphism of right dual}
\phi^{-1}(T):=T((\ )\ot_{B^{op}}x_{i})\ot_{B}x^{i}.
\end{align}
This isomorphism restricts to an isomorphism   $\hat{\phi}:\cL^{\vee}\times_{B}\cL^{\vee}\to \Hom_{B}(\cL\ot_{B^{e}}\cL, B)$,
where the balanced tensor product in $\cL^{\vee}\times_{B}\cL^{\vee}$ is given in Definition \ref{def:right.bgd}~(i). The balanced tensor product in $\Hom_{B}(\cL\ot_{B^{e}}\cL, B)$ is induced by $s$ and $t$ (so $Xs(a)t(b)\ot_{B^{e}}Y=X\ot_{B^{e}}s(a)t(b)Y$, for all $X, Y\in \cL$ and $a, b\in B$). Given $\sigma\in \cL^{\vee}$, define $\tilde{\sigma}\in \Hom_{B}(\cL\ot_{B^{e}}\cL, B)$ by
\begin{align}\label{equ. tilde sigma}
    \tilde{\sigma}(X\ot_{B^{e}}Y)=\sigma(XY).
\end{align}
The coproduct of $\sigma$ can now be given as
\begin{align}
   \Delta(\sigma):=\hat\phi^{-1}(\Tilde{\sigma}),
\end{align}
 or equivalently,
\begin{align}\label{equ. coproduct of the right dual bialgebroid}
    \one{\sigma}(X t\circ\two{\sigma}(Y))=\sigma(XY).
\end{align}
The counit is just the evaluation on the unit of $1_{\cL}$. 

In the following, we  write
\begin{align}
    \lan \sigma | X\ran:=\sigma(X).
\end{align}
for $X\in \cL$ and $\sigma\in \cL^{\vee}$. We also let $\Lambda=\cL^\vee$ and note that it is a right-finite left bialgebroid. Right-finite here means that $\Lambda$ is finitely generated projective as a left $B$ module induced by the source map on the left, meaning that there are
$\rho_{i}\in \Lambda$ and $\rho^{i}\in {}^{\vee}\Lambda$ (the collection of left module maps) such that for any $\sigma\in \Lambda$, $\sigma=s(\rho^{i}(\sigma))\rho_{i}$. 

Given a right-finite left bialgebroid $\Lambda$, we have ${}^\vee\Lambda$ a left-finite left bialgebroid over $B$. The product is given by
\begin{align}\label{equ. product of left dual bialgebroid}
    (XY)(\sigma):=(X t(Y(\two{\sigma})))(\one{\sigma}),
\end{align}
where $X, Y\in {}^{\vee}\Lambda$ and $\sigma\in \Lambda$. The source and target maps are given by
\begin{align}\label{equ. source and target map of left dual}
    s(b)(\sigma)=\varepsilon(\sigma  s(b))=\varepsilon(\sigma  t(b)),\quad t(b)(\sigma)=\varepsilon(\sigma)b
\end{align}
for all $b\in B$. In this context, we similarly write $\lan \sigma | X\ran:=X(\sigma)$  for $X\in {}^{\vee}\Lambda$ and $\sigma\in\Lambda$. Then both (\ref{equ. coproduct of the right dual bialgebroid}) and (\ref{equ. product of left dual bialgebroid}) can be written in the same form as
\begin{align}\label{equ. product on L}
    \lan \sigma |XY\ran=\lan\one{\sigma}|Xt\lan\two{\sigma}|Y\ran\ran=\lan t\lan\two{\sigma}|Y\ran\one{\sigma}|X\ran
\end{align}
for $X,Y,\sigma$ in the relevant spaces. Moreover, the coproduct on $\VL$ is characterised by
\begin{align}\label{equ. coproduct on L}
    \lan\sigma\eta|X\ran=\lan\sigma| s\lan\eta|\one{X}\ran\two{X}\ran=\lan\sigma s\lan\eta|\one{X}\ran| \two{X}\ran.
\end{align}

We can summarise the above constructions as follows:
\begin{thm}\label{thm. bialgebroids dual}{\rm \cite[Thm.~5.13]{schau1}.} If $\cL$  is a left-finite left bialgebroid then $\LV$  is a right-finite left  bialgebroid. If  $\Lambda$ is a right-finite left bialgebroid then  $\VL$ is a left-finite left bialgebroid. Moreover, we can respectively identify ${}^\vee(\cL^\vee)=\cL$ and  $({}^\vee\Lambda)^\vee=\Lambda$.
\end{thm}

Schauenburg also later extended this result to (anti)-left Hopf algebroids in \cite{schau3} by abstract arguments. A  version of the following more explicit statement can also be found in  \cite{K18}, obtained there by constructing the translation map. 

\begin{prop}{\rm \cite{K18,schau3}}. \label{thm. Hopf algebroids dual} If $\cL$  is a left-finite left Hopf algebroid then $\LV$  is a right-finite  anti-left Hopf algebroid. If  $\Lambda$ is a right-finite anti-left Hopf algebroid then  $\VL$ is a left-finite left Hopf algebroid.
\end{prop}
\begin{proof}
We construct the following commutative diagram:
\[
\begin{tikzcd}
  &\LV\ot^{B^{op}}\LV \arrow[d, "\psi"] \arrow[r, "\mu"] & \LV\ot_{B}\LV \arrow[d, "\phi"] \\
   & \Hom_{B}(\cL\ot_{B}\cL, B)  \arrow[r, "\lambda^{*}"] & \Hom_{B}(\cL\ot_{B^{op}}\cL, B),
\end{tikzcd}
\]
where $\mu$ is the map given in Definition~\ref{defHopf}  and we want to show that it is invertible. Here, $\phi$ is the isomorphism given by (\ref{equ. isomorphism of right dual}). The domain  of $\lambda^{*}$ is  a  right $B$-module by $\tilde{T}(X\ot_{B}Y)b=\tilde{T}(X\ot_{B}t(b)Y)$ for $\tilde{T}\in \Hom_{B}(\cL\ot_{B}\cL, B)$. Moreover, $\lambda^*$ is given by the pull back of $\lambda$ in Definition \ref{defHopf}, 
\begin{align}
    \lambda^{*}(\tilde{T})(X\ot_{B^{op}}Y)=&\tilde{T}(\one{X}\ot_{B}\two{X}Y).
\end{align}
As a result, $\lambda^{*}$ is an invertible map with  inverse,
\begin{align}
    \lambda^{*-1}(T)(X\ot_{B}Y)=T(X_{+}\ot_{B^{op}}X_{-}Y).
\end{align}
Finally, we define $\psi$ by
\begin{align}\label{equ. isomorphic of psi}
    \psi(\sigma\ot^{B^{op}}\eta)(X\ot_{B}Y):=\eta(s\circ\sigma(X)Y)=\lan\eta|s\lan\sigma|X\ran Y\ran. 
\end{align}
It is not hard to check that this is well defined with respect to all the balanced tensor products. Using that $\cL$ is left-finite, we write down the inverse of $\psi$ as
\begin{align}
    \psi^{-1}(\tilde{T}):=x^{i}\ot^{B^{op}}\tilde{T}(x_{i}\ot_{B}(\ )),
\end{align}
which is clearly well defined since $\tilde{T}$ is $B^{op}$-linear on the second tensor factor, and check that
\begin{align*}
    \psi\circ\psi^{-1}(\tilde{T})(X\ot_{B}Y)=&\tilde{T}(x_{i}\ot_{B}s(x^{i}(X))Y)=\tilde{T}(t(x^{i}(X))x_{i}\ot_{B}Y)=\tilde{T}(X\ot_{B}Y),\\
    \psi^{-1}\circ\psi(\sigma\ot^{B^{op}}\eta)=&x^{i}\ot^{B^{op}}\eta(s\circ\sigma(x_{i})(\ ))
    =(s\circ\sigma(x_{i})) x^{i}\ot^{B^{op}}\eta
    =\sigma\ot^{B^{op}}\eta.
\end{align*}
The last step uses that
\begin{align*}
    ((s\circ\sigma(x_{i})) x^{i})(X)=\sigma(x_{i})x^{i}(X)=\sigma(t(x^{i}(X))x_{i})=\sigma(X).
\end{align*}
Finally, we check that the diagram commutes. For $\sigma\ot^{B^{op}}\eta\in \LV\ot^{B^{op}}\LV$ and $X\ot_{B^{op}}Y\in \cL\ot_{B^{op}}\cL$, we have
\begin{align*}     \lambda^{*}\circ\psi(\sigma\ot^{B^{op}}\eta)(X\ot_{B^{op}}Y)=&\psi(\sigma\ot^{B^{op}}\eta)(X\o\ot_{B}X\t\,Y)
=\<\eta|s\<\sigma|X\o\>X\t\,Y\>\\
=&\<\eta\o|s\<\sigma|X\o\>X\t\,t\<\eta\t|Y\>\>=\<\eta\o\,\sigma|X\,t\<\eta\t|Y\>\>\\
=&\phi(\eta\o\,\sigma\ot_{B} \eta\t)(X\ot_{B^{op}}Y)=\phi\circ \mu(\sigma\ot^{B^{op}}\eta)(X\ot_{B^{op}}Y).
\end{align*}
It follows that $\mu$ is invertible as the other maps are.

Similarly, if $\Lambda$ is a right-finite anti-left Hopf  algebroid, one has a commutative diagram
\[
\begin{tikzcd}
  &\VL\ot_{B^{op}}\VL \arrow[d, "\tilde{\phi}"] \arrow[r, "\lambda"] & \VL\ot_{B}\VL \arrow[d, "\tilde{\psi}"]\\
   & {}_{B}\Hom(\Lambda\ot_{B}\Lambda, B)  \arrow[r, "\mu^{*}"] & {}_{B}\Hom(\Lambda\ot^{B^{op}}\Lambda, B),
\end{tikzcd}
\]
with all maps isomorphisms. Here, $\mu^{*}$ is the pull back of $\mu$, $F\in {}_{B}\Hom(\Lambda\ot_{B}\Lambda, B)$ means a left $B$-module map i.e.,
\begin{align}
    F(s(b)\sigma\ot_{B}\eta)=bF(\sigma\ot_{B}\eta),
\end{align}
and similarly $\tilde{F}\in {}_{B}\Hom(\Lambda\ot^{B^{op}}\Lambda, B)$ means
\begin{align}
    \tilde{F}(\sigma\ot^{B^{op}}s(b)\eta)=b\tilde{F}(\sigma\ot^{B^{op}}\eta),
\end{align}
for all $b\in B$ and $\sigma, \eta\in \Lambda$. For the vertical maps, we have
\begin{align}
    \tilde{\phi}(X\ot_{B^{op}}Y)(\sigma\ot_{B}\eta):=\lan\sigma|X t\lan\eta|Y\ran\ran=\lan t\lan\eta|Y\ran \sigma|X \ran,\\
    \tilde{\psi}(X\ot_{B}Y)(\sigma\ot^{B^{op}}\eta):=\lan\eta|s\lan\sigma|X\ran Y\ran=\lan\eta  s\lan\sigma|X\ran|Y\ran, 
\end{align}
with inverses
\begin{align}
    \tilde{\phi}^{-1}(F)=F((\ )\ot_{B}\rho_{i})\ot_{B^{op}}\rho^{i},\quad
    \tilde{\psi}^{-1}(\tilde{F})=\rho^{i}\ot_{B}\tilde{F}(\rho_{i}\ot^{B^{op}}(\ )).
\end{align}
\end{proof}

\subsection{Non-Abelian second cohomology $\CH^2(\cL)$}\label{seccoh1}
For a left bialgebroid $\cL$ over an algebra $B$, we define $C^{1}(\cL)$ to be the group of invertible elements $U\in \cL$, such that
\begin{equation}\label{C1L} \varepsilon(U)=1,\quad Us(b)=s(b)U,\quad Ut(b)=t(b)U\end{equation}
for all $b\in B$.
We also define
\[Z^{1}(\cL)=\{U\in C^{1}(\cL)\ | \quad U \quad\textrm{grouplike } \}.\]

\begin{lem}
    If $\cL$ is an anti-left Hopf  algebroid then $Z^{1}(\cL)$ is a group with
    \[U^{-1}=s(\varepsilon(U_{[+]}))U_{[-]}
    \]
    for all $U\in Z^{1}(\cL)$. Similarly, if $\cL$ is a left Hopf algebroid then $Z^{1}(\cL)$ is a group with
    \[U^{-1}=t(\varepsilon(U_{+}))U_{-}
    \]
    for all $U\in Z^{1}(\cL)$.
\end{lem}
\begin{proof}
First, as $U$ commutes with the images of the source and target maps, by applying $\mu^{-1}$ in (\ref{X[+][-]}) on both sides of the latter two parts of (\ref{C1L}), we obtain
\[t(b)U_{[-]}\ot^{B^{op}}U_{[+]}=U_{[-]}t(b)\ot^{B^{op}}U_{[+]},\quad U_{[-]}\ot^{B^{op}}t(b)U_{[+]}=U_{[-]}\ot^{B^{op}}U_{[+]}t(b).\]
 We next see that the formula stated for $U^{-1}$ is well defined,
\[s(\varepsilon(U_{[+]}s(b)))U_{[-]}=s(\varepsilon(U_{[+]}t(b)))U_{[-]}=s(\varepsilon(t(b)U_{[+]}))U_{[-]}=s(\varepsilon(U_{[+]}))s(b)U_{[-]},\]
where the 1st step uses the definition of the counit $\varepsilon$.
 It is not hard to see that $U^{-1}$ commutes with the image of the target map, while it commutes with the image of the source map by 
    \[s(\varepsilon(U_{[+]}))U_{[-]}s(b)=s(\varepsilon(s(b)U_{[+]}))U_{[-]}=s(b)s(\varepsilon(U_{[+]}))U_{[-]}.\]
It remains to check that $U^{-1}$  is indeed the inverse of $U$. On one side, this is
\begin{align*}
    U s(\varepsilon(U_{[+]}))U_{[-]}&=s(\varepsilon(U_{[+]}))UU_{[-]}=s(\varepsilon(U\o{}_{[+]}))U\t U\o{}_{[-]}\\
    &=s(\varepsilon(U_{[+]}\o))U_{[+]}\t U_{[-]}=U_{[+]}U_{[-]}=t(\varepsilon(U))=1.
\end{align*}
Similarly on the other side,
\begin{align*}
    s(\varepsilon(U_{[+]}))U_{[-]}U=s(\varepsilon(U\t{}_{[+]}))U\t{}_{[-]}U\o=s(\varepsilon(U))=1.
\end{align*}
Finally, 
\[\varepsilon(U^{-1})=\varepsilon(U^{-1}t\circ\varepsilon(U))=\varepsilon(U^{-1}U)=1\]
so that $U^{-1}\in Z^1(\CL)$. 
\end{proof}
It is straightforward to see that the above constructions are equivalent in the left-finite or right-finite case to results in Section~\ref{secbisec}:
\begin{cor}\label{lem. dual bisection}
    If $\cL$ is a left-finite   bialgebroid over $B$ and $\Lambda=\cL^{\vee}$ then $Z^{1}(\Lambda)=Z^{1}(\cL, B)$ and $C^{1}(\Lambda)=C^{1}(\cL, B)$. If $\Lambda$ is a right-finite  left bialgebroid over $B$ and ${}^{\vee}\Lambda=\cL$ then $Z^{1}(\Lambda, B)=Z^{1}(\cL)$ and $C^{1}(\Lambda, B)=C^{1}(\cL)$.
\end{cor}

We also treat 2-cocycles on left bialgebroids in the same way. This is similar to 2-cocycles on a right bialgebroids given in \cite{Boehm}.
\begin{defi}\label{def. dual 2-cocycle}
    Let $\cL$ be a left bialgebroid over $B$. A counital  2-cocycle in $\cL$ is an invertible  element $F=F^{\alpha}\ot_{B}F_{\alpha}\in\cL\times_{B}\cL$, such that
    \begin{itemize}
        \item [(i)] $s(b)F^{\alpha}\ot_{B}t(b')F_{\alpha}=F^{\alpha}s(b)\ot_{B}F_{\alpha}t(b')$ for all $b,b'\in B$;
        \item[(ii)] $(F\ot_{B} 1)(\Delta\ot_{B} \id)(F)= (1\ot_{B} F)(\id\ot_{B}\Delta)(F)$;
        \item[(iii)] $(\varepsilon\ot_{B} \id)(F)=1=(\id\ot_{B}\varepsilon)(F)$
    \end{itemize}
    for all $b, b'\in B$. The collection of such counital 2-cocycles in $\cL$ will be denoted $Z^{2}(\cL)$. We will use the notation $F^{-1}=F^{-\alpha}\ot_B F_{-\alpha}$ for the inverses. 
\end{defi}

Moreover, given a 2-cocycle $F$ as above on a left bialgebroid $\CL$, there is a new left bialgebroid $\CL_F$  with coproduct
\begin{align}\label{twistcoprod}
    \Delta^{F}(X):=F\Delta(X)F^{-1},
\end{align}
for all $X\in \cL$ and the original product, unit, $s, t$. This construction is dual to cotwisting by a cocycle on $\CL$ in the sense of Definition~\ref{Lcotwist}.

\begin{lem}\label{lem. dual 2-cocycle}
    Let $\cL$ be a left-finite left bialgebroid and $\Lambda=\cL^{\vee}$. Then $Z^2(\CL,B)\simeq Z^2(\Lambda)$. Given a 2-cocycle $F\in Z^{2}(\Lambda)$, we construct a right-handed 2-cocycle  $\Gamma_{F}\in Z^{2}(\cL, B)$ on $\CL$ by
\[\Gamma_{F}(X\ot_{B^{e}}Y)=\lan F^{\alpha}|X t\lan F_{\alpha}| Y\ran\ran.\]
Similarly, $\Gamma_{F^{-1}}$ is a left handed 2-cocycle.
Moreover, we have $(\cL^{\Gamma_{F}^{-1}})^{\vee}\simeq \Lambda_{F}$.
\end{lem}
\begin{proof}
    The isomorphism is given by $\hat{\phi}: \Lambda\times_{B}\Lambda\to \Hom_{B}(\cL\ot_{B^{e}}\cL, B)$ in the proof of Theorem \ref{thm. Hopf algebroids dual},  according to
    \begin{align*}
        \hat{\phi}(\alpha\ot_{B}\beta)(X\ot_{B^{e}}Y)=\lan\alpha|X t\lan \beta| Y\ran\ran.
    \end{align*}
As  $s(b)F^{\alpha}\ot_{B}t(b')F_{\alpha}=F^{\alpha}s(b)\ot_{B}F_{\alpha}t(b')$, we have that $\Gamma_{F}$ is left $B$-linear and satisfies $\Gamma_{F}(X, Ys(b))=\Gamma_{F}(X, Yt(b))$. To show the 2-cocycle condition on $\Gamma_{F}$, we first observe that there is an isomorphism $\hat{\Phi}: \Lambda\times_{B}\Lambda\times_{B}\Lambda\to {}_{B}\Hom_{B}(\cL\otimes_{B^e}\cL\otimes_{B^e}\cL, B)$ given by
    \[\hat{\Phi}(\alpha\ot_{B}\beta\ot_{B}\gamma)(X\ot_{B^{e}}Y\ot_{B^{e}}Z)=\lan\alpha|X t\lan \beta|Y t\lan \gamma|Z\ran\ran\ran.\]
    We have on the one hand,
    \begin{align*}
        \hat{\Phi}(F^{\beta}F^{\alpha}\o&\ot_{B}F_{\beta}F^{\alpha}\t\ot_{B}F_{\alpha})(X\ot_{B^{e}}Y\ot_{B^{e}}Z)\\
    =&\lan F^{\beta}|s\lan F^{\alpha}\o|X\o\ran X\t t\lan F_{\beta}| s\lan F^{\alpha}\t|Y\o\ran Y\t t\lan F_{\alpha}|Z\ran\ran\ran\\
    =&\lan F^{\beta}|s\lan F^{\alpha}\o|X\o t\lan F^{\alpha}\t|Y\o\ran\ran X\t t\lan F_{\beta}|  Y\t t\lan F_{\alpha}|Z\ran\ran\ran\\
    =&\lan F^{\beta}|s\lan F^{\alpha}\o|X\o t\lan F^{\alpha}\t|Y\o t\lan F_{\alpha}|Z\ran\ran\ran X\t t\lan F_{\beta}|  Y\t \ran\ran\\
    =&\lan F^{\alpha}\o|X\o t\lan F^{\alpha}\t|Y\o t\lan F_{\alpha}|Z\ran\ran\ran\lan F^{\beta}| X\t t\lan F_{\beta}|  Y\t \ran\ran\\
    =&\lan F^{\alpha}|X\o Y\o t\lan F_{\alpha}|Z\ran\ran \lan F^{\beta}| X\t t\lan F_{\beta}|  Y\t \ran\ran\\    =&\Gamma(t(\Gamma(\two{X}, \two{Y}))\one{X}\one{Y}, Z),
    \end{align*}
and on the other hand, by similar computations,
    \[
    \hat{\Phi}(F^{\alpha}\ot_{B}F_{\beta}F_{\alpha}\o\ot_{B}F_{\beta}F_{\alpha}\t)(X\ot_{B^{e}}Y\ot_{B^{e}}Z)=\Gamma(X, t(\Gamma(\two{Y}, \two{Z}))\one{Y}\one{Z}).
    \]
To show  that $(\cL^{\Gamma_{F}^{-1}})^{\vee}\simeq \Lambda_{F}$, we have for any $\sigma\in \cL^{\vee}$,
   \begin{align*}
       \hat{\phi}(\Delta^{F}(\sigma))(X\ot_{B^{e}}Y)=& \lan F^{\alpha}| s\lan \sigma\o| s\lan F^{-\alpha}|X\o\ran X\t\ran X\th t\lan F_{\alpha}| s\lan \sigma\t| s\lan F_{-\alpha}|Y\o\ran Y\t\ran|Y\th\ran\ran\\
       =&\lan \sigma\o| s\lan F^{-\alpha}|X\o\ran X\t\ran \lan F^{\alpha}|X\th t\lan F_{\alpha}| s\lan \sigma\t| s\lan F_{-\alpha}|Y\o\ran Y\t\ran|Y\th\ran\ran\\
       =&\lan F^{-\alpha}|X\o t\lan F_{-\alpha}|Y\o\ran\ran \lan \sigma\o|X\t t\lan \sigma\t|Y\t\ran\ran \lan F^{\beta}|X\th t\lan F_{\beta}|Y\th\ran\ran\\
       =&\Gamma_{F}^{-1}(X\o, Y\o) \sigma(X\t Y\t)\Gamma_{F}(X\th, Y\th)\\
       =&\lan \sigma|X\cdot_{\Gamma^{-1}_{F}}Y\ran.
\end{align*}
\end{proof}
We now have a parallel to  Theorem~\ref{thm: boundary map}.
\begin{thm}\label{thm: boundary map1}
For any $U\in C^{1}(\cL)$, we can construct a 2-cocycle $\partial U\in Z^2(\CL)$ by
\begin{align}
    \partial U:=UU^{-1}\o\ot_{B}UU^{-1}\t,
\end{align}
where $U^{-1}$ is the inverse of $U$. Moreover, if $F\in Z^2(\CL)$ and $U\in C^{1}(\cL)$ then
\begin{align}   F_{U}=UF^{\alpha}U^{-1}\o\ot_{B}UF_{\alpha}U^{-1}\t
\end{align}
is another 2-cocycle in $Z^{2}(\cL)$ with  inverse given by
\begin{align}
    (F_{U})^{-1}=U\o F^{-\alpha}U^{-1}\ot_{B}U\t F_{-\alpha}U^{-1}.
\end{align}
This gives a group action of $C^1(\CL)$ on $Z^{2}(\cL)$ with the set of orbits denoted $\H^{2}(\cL)$.
\end{thm}
\begin{proof}We omit the proof  that  $F_U$ obeys the 2-cocycle condition, but check  that
 $F_{U}$ is well defined (and hence also $\partial U$). Indeed, to show that different representatives of  $U^{-1}\o\ot_{B}U^{-1}\t=t(b)\,\tilde U^{-1}\o\ot_{B}\tilde U^{-1}\t= \tilde U^{-1}\o\ot_{B}s(b)\,\tilde U^{-1}\t$ give the same $F_U$, we check that
\begin{align*}
    UF^{\alpha}\,t(b)\,\tilde U^{-1}\o\ot_{B}UF_{\alpha}\tilde U^{-1}\t=&U\,t(b)\,F^{\alpha}\tilde U^{-1}\o\ot_{B}UF_{\alpha}\tilde U^{-1}\t=\,t(b)\,UF^{\alpha}\tilde U^{-1}\o\ot_{B}UF_{\alpha}\tilde U^{-1}\t\\
    =&UF^{\alpha}\tilde U^{-1}\o\ot_{B}s(b)\,UF_{\alpha}\tilde U^{-1}\t=UF^{\alpha}\tilde U^{-1}\o\ot_{B}U\,s(b)\,F_{\alpha}\tilde U^{-1}\t\\
    =&UF^{\alpha}\tilde U^{-1}\o\ot_{B}UF_{\alpha}\,s(b)\,\tilde U^{-1}\t,
\end{align*}
using that the images of $s,t$ commute with $U$. That the construction of $F_U$ also factors through the balanced tensor product of $F^\alpha\ot_{B} F_{\alpha}$ is similar. We also check that $F_U$ satisfies Definition \ref{def. dual 2-cocycle} (i). Applying $\Delta$ to both sides of  $s(b)U^{-1}=U^{-1}s(b)$, one can show that $s(b)\,U^{-1}\o\ot_{B}U^{-1}\t=U^{-1}\o\,s(b)\ot_{B}U^{-1}\t$. Hence,
\begin{align*}
    UF^{\alpha}U^{-1}\o\,s(b)\,\ot_{B}UF_{\alpha}U^{-1}\t=&UF^{\alpha}\,s(b)\,U^{-1}\o\ot_{B}UF_{\alpha}U^{-1}\t=\,s(b)\,UF^{\alpha}U^{-1}\o\ot_{B}UF_{\alpha}U^{-1}\t.
\end{align*}
By the same method, we also have $UF^{\alpha}U^{-1}\o\ot_{B}UF_{\alpha}U^{-1}\t\,t(b)\,=UF^{\alpha}U^{-1}\o\ot_{B}\,t(b)\,UF_{\alpha}U^{-1}\t$.
 Similarly for the formula for $(F_U)^{-1}$.
\end{proof}

The connection back to Theorem~\ref{thm: boundary map} is:
\begin{prop}
    Let $\cL$ be a left-finite left bialgebroid and $\Lambda=\cL^{\vee}$. Then $\mathcal{H}^2(\CL,B)\simeq\mathcal{H}^2(\Lambda)$.
\end{prop}
\begin{proof}
  We use Corollary~\ref{lem. dual bisection} and Lemma~\ref{lem. dual 2-cocycle}, and let $F\in Z^2(\Lambda)$,  $U\in C^{1}(\Lambda)=C^{1}(\cL, B)$. We  check that 
    \begin{align*}
        \Gamma_{(F_{U})^{-1}}(X\ot_{B^{e}}Y)=&\lan U\o F^{-\alpha} U^{-1}|X t\lan U\t F_{-\alpha}U^{-1}|Y\ran\ran\\
        =&U^{-1}(\one{X}t(U^{-1}(\one{Y})))\Gamma_{F^{-1}}(\two{X}, \two{Y})U(\three{X}\three{Y})=(\Gamma_{F^{-1}})^{U}(X\ot_{B^{e}}Y)
    \end{align*}
   for all $X,Y\in \CL$. 
\end{proof}

Finally, analogously to 2-cocycles in a bialgebra \cite{Ma:book}, if $F$ is a 2-cocycle then we also have the useful identities 
\begin{align}
    F^{\alpha}\o F^{-\beta}\ot_{B}F^{\alpha}\t F_{-\beta}\o \ot_{B} F_{\alpha}F_{-\beta}\t&=F^{-\beta}\ot_{B}F_{-\beta}F^{\alpha}\ot_{B}F_{\alpha},\\
    F^{\alpha}F^{-\beta}\o\ot_{B}F_{\alpha}\o F^{-\beta}\t\ot_{B}F_{\alpha}\t F_{\beta}&=F^{\alpha}\ot_{B} F^{-\beta}F_{\alpha}\ot_{B}F_{-\beta}.
\end{align}
We then have the following dual version of Theorem~\ref{thm. 2 cocycle twist1}, which in the left-finite or right-finite case would be equivalent by duality, but which is quite hard to prove in general as we do here.

\begin{thm}\label{thm. 2 cocycle twist}
Let $\cL$ be an anti-left Hopf  algebroid and $F\in Z^{2}(\cL)$. Then the left bialgebroid $\cL_{F}$ with twisted coproduct (\ref{twistcoprod}) is an anti-left Hopf  algebroid with
\begin{align*}
   (\mu_{F})^{-1}(Y&\ot_{B}X)
   =F_{\gamma}(F_{-\alpha}F^{\beta}X\o F^{-\delta} F^{\gamma})_{[-]}F^{-\alpha}Y\ot^{B^{op}}s(\varepsilon((F_{-\alpha}F^{\beta}X\o F^{-\delta} F^{\gamma})_{[+]}))F_{\beta}X\t F_{-\delta}.
\end{align*}
Similarly, if $\cL$ is a left Hopf algebroid and $F\in Z^{2}(\cL)$ then the left bialgebroid $\cL_{F}$ with twisted coproduct (\ref{twistcoprod}) is a left Hopf algebroid with
\begin{align*}
(\lambda_{F})^{-1}(X&\ot_{B}Y)
   =
    t(\varepsilon((F^{-\alpha}F_{\beta}X\t F_{-\delta}F_{\gamma})_{+}))F^{\beta}X\o F^{-\delta}\ot_{B^{op}}F^{\gamma}(F^{-\alpha}F_{\beta}X\t F_{-\delta}F_{\gamma})_{-}F_{-\alpha}Y.
\end{align*}
\end{thm}

\begin{proof}
  One can check that  $\mu_{F}$ is well defined. To verify the inverse on one side,  it suffices to check that
    \begin{align*}
        (\mu_{F}\circ \mu_{F}^{-1})&(1\ot_{B}X)\\
        =&s\circ \varepsilon((F_{-\alpha}F^{\beta}X\o F^{-\delta}F^{\gamma})_{[+]})F^{\rho}F_{\beta}\o X\t F_{-\delta}\o F^{-\lambda}F_{\gamma}(F_{-\alpha}F^{\beta}X\o F^{-\delta}F^{\gamma})_{[-]}F^{-\alpha}\\
        &\quad \ot_{B}F_{\rho}F_{\beta}\t X\th F_{-\delta}\t F_{-\lambda}\\
         =&s\circ \varepsilon((F_{-\alpha}F^{\rho}F^{\beta}\o X\o F^{-\delta}\o F^{-\lambda}F^{\gamma})_{[+]})F_{\rho}F^{\beta}\t X\t F^{-\delta}\t F_{-\lambda}F_{\gamma}\\
        &\quad (F_{-\alpha}F^{\rho}F^{\beta}\o X\o F^{-\delta}\o F^{-\lambda}F^{\gamma})_{[-]}F^{-\alpha}\ot_{B}F_{\beta} X\th F_{-\delta}\\
        =&s\circ \varepsilon((F_{-\alpha}F^{\rho}F^{\beta}\o X\o F^{-\delta}\o )_{[+]})F_{\rho}F^{\beta}\t X\t F^{-\delta}\t \\
        &\quad (F_{-\alpha}F^{\rho}F^{\beta}\o X\o F^{-\delta}\o )_{[-]}F^{-\alpha}\ot_{B}F_{\beta} X\th F_{-\delta}\\
        =&s\circ \varepsilon((F^{\alpha}\t F_{-\rho}\o F^{\beta}\o X\o F^{-\delta}\o )_{[+]})F_{\alpha}F_{-\rho}\t F^{\beta}\t X\t F^{-\delta}\t \\
        &\quad (F^{\alpha}\t F_{-\rho}\o F^{\beta}\o X\o F^{-\delta}\o )_{[-]}F^{\alpha}\o F^{-\rho}\ot_{B}F_{\beta} X\th F_{-\delta}\\
        =&s\circ \varepsilon(F^{\alpha}\t{}_{[+]} (F_{-\rho}\o F^{\beta}\o X\o F^{-\delta}\o )_{[+]})F_{\alpha}F_{-\rho}\t F^{\beta}\t X\t F^{-\delta}\t \\
        &\quad ( F_{-\rho}\o F^{\beta}\o X\o F^{-\delta}\o )_{[-]}F^{\alpha}\t{}_{[-]}F^{\alpha}\o F^{-\rho}\ot_{B}F_{\beta} X\th F_{-\delta}\\
        =&s\circ \varepsilon(F^{\alpha}(F_{-\rho}\o F^{\beta}\o X\o F^{-\delta}\o )_{[+]})F_{\alpha}F_{-\rho}\t F^{\beta}\t X\t F^{-\delta}\t \\
        &\quad ( F_{-\rho}\o F^{\beta}\o X\o F^{-\delta}\o )_{[-]} F^{-\rho}\ot_{B}F_{\beta} X\th F_{-\delta}\\
        =&s\circ \varepsilon(F^{\alpha}(F_{-\rho}\o F^{\beta}\o )_{[+]} X\o{}_{[+]} t\circ\varepsilon(F^{-\delta}\o {}_{[+]}))F_{\alpha}F_{-\rho}\t F^{\beta}\t X\t F^{-\delta}\t  F^{-\delta}\o {}_{[-]}X\o{}_{[-]}\\
        &\quad ( F_{-\rho}\o F^{\beta}\o   )_{[-]} F^{-\rho}\ot_{B}F_{\beta} X\th F_{-\delta}\\
        =&s\circ \varepsilon(F^{\alpha}(F_{-\rho}\o F^{\beta}\o )_{[+]} \big(X\o{} t\circ\varepsilon(F^{-\delta}\o {}_{[+]})\big)_{[+]})F_{\alpha}F_{-\rho}\t F^{\beta}\t X\t F^{-\delta}\t  F^{-\delta}\o {}_{[-]}\\
        &\quad\big(X\o{} t\circ\varepsilon(F^{-\delta}\o {}_{[+]})\big)_{[-]} ( F_{-\rho}\o F^{\beta}\o   )_{[-]} F^{-\rho}\ot_{B}F_{\beta} X\th F_{-\delta}\\
        =&s\circ \varepsilon(F^{\alpha}(F_{-\rho}\o F^{\beta}\o )_{[+]} X\o{}_{[+]} )F_{\alpha}F_{-\rho}\t F^{\beta}\t X\t s\circ\varepsilon(F^{-\delta}\o {}_{[+]})F^{-\delta}\t  F^{-\delta}\o {}_{[-]}X\o{}_{[-]}\\
        &\quad ( F_{-\rho}\o F^{\beta}\o   )_{[-]} F^{-\rho}\ot_{B}F_{\beta} X\th F_{-\delta}\\
        =&s\circ \varepsilon(F^{\alpha}(F_{-\rho}\o F^{\beta}\o )_{[+]} X\o{}_{[+]} )F_{\alpha}F_{-\rho}\t F^{\beta}\t X\t s\circ\varepsilon(F^{-\delta}{}_{[+]}\o )F^{-\delta}{}_{[+]}\t  F^{-\delta}{}_{[-]}X\o{}_{[-]}\\
        &\quad ( F_{-\rho}\o F^{\beta}\o   )_{[-]} F^{-\rho}\ot_{B}F_{\beta} X\th F_{-\delta}\\
        =&s\circ \varepsilon(F^{\alpha}(F_{-\rho}\o F^{\beta}\o )_{[+]} X\o{}_{[+]} )F_{\alpha}F_{-\rho}\t F^{\beta}\t X\t X\o{}_{[-]}\\
        &\quad ( F_{-\rho}\o F^{\beta}\o   )_{[-]} F^{-\rho}\ot_{B}F_{\beta} X\th \\
        =&s\circ \varepsilon(F^{\alpha}F_{-\rho}\o{}_{[+]} F^{\beta}\o {}_{[+]} t\circ\varepsilon(X\o{}_{[+]} ))F_{\alpha}F_{-\rho}\t F^{\beta}\t X\t X\o{}_{[-]}\\
        &\quad F^{\beta}\o   {}_{[-]} F_{-\rho}\o{}_{[-]}  F^{-\rho}\ot_{B}F_{\beta} X\th \\
        =&s\circ \varepsilon(F^{\alpha}F_{-\rho}\o{}_{[+]} F^{\beta}\o {}_{[+]} )F_{\alpha}F_{-\rho}\t F^{\beta}\t t\circ\varepsilon(X\o) F^{\beta}\o   {}_{[-]} F_{-\rho}\o{}_{[-]} F^{-\rho}\ot_{B}F_{\beta} X\t \\
        =&s\circ \varepsilon(F^{\alpha}F_{-\rho}\o{}_{[+]} F^{\beta}\o {}_{[+]} )F_{\alpha}F_{-\rho}\t F^{\beta}\t  F^{\beta}\o   {}_{[-]} F_{-\rho}\o{}_{[-]} F^{-\rho}\ot_{B}F_{\beta} s\circ\varepsilon(X\o)X\t \\
         =&s\circ \varepsilon(F^{\alpha}F_{-\rho}\o{}_{[+]}  )F_{\alpha}F_{-\rho}\t  F_{-\rho}\o{}_{[-]} F^{-\rho}\ot_{B} X F_{-\delta}\\
         =&s\circ \varepsilon(F^{\alpha} )F_{\alpha}  \ot_{B} X =1\ot_{B}X
    \end{align*}
for all $X\in \CL$. Similarly on the other side,
\begin{align*}
    (\mu_{F}^{-1}\circ \mu_{F})&
    (1\ot^{B^{op}}X)\\
    =&F_{\alpha}(F_{-\beta}F^{\lambda}F^{\gamma}\o X\t F_{-\delta}\o F^{-\rho}F^{\alpha})_{[-]}F^{-\beta}F^{\gamma}X\o F^{-\delta}\\
    &\quad \ot^{B^{op}} s\circ\varepsilon((F_{-\beta}F^{\lambda}F_{\gamma}\o X\t F_{-\delta}\o F^{-\rho}F^{\alpha})_{[+]})F_{\lambda}F_{\gamma}\t X\th F_{-\delta}\t F_{-\rho}\\
    =&F_{\alpha}(F_{-\beta}F_{\lambda}F^{\gamma}\t X\t F^{-\delta}\t F_{-\rho}F^{\alpha})_{[-]}F^{-\beta}F^{\lambda}F^{\gamma}\o X\o F^{-\delta}\o F^{-\rho}\\
    &\quad \ot^{B^{op}} s\circ\varepsilon((F_{-\beta}F_{\lambda}F^{\gamma}\t X\t F^{-\delta}\t F_{-\rho}F^{\alpha})_{[+]})F_{\gamma} X\th F_{-\delta}\\
    =&F_{\alpha}(F^{\gamma}\t X\t F^{-\delta}\t F_{-\rho}F^{\alpha})_{[-]}F^{\gamma}\o X\o F^{-\delta}\o F^{-\rho}\\
    &\quad \ot^{B^{op}} s\circ\varepsilon((F^{\gamma}\t X\t F^{-\delta}\t F_{-\rho}F^{\alpha})_{[+]})F_{\gamma} X\th F_{-\delta}\\
    =&F_{\alpha}( X\t F^{-\delta}\t F_{-\rho}F^{\alpha})_{[-]} X\o F^{-\delta}\o F^{-\rho}\\
    &\quad \ot^{B^{op}} s\circ\varepsilon((X\t F^{-\delta}\t F_{-\rho}F^{\alpha})_{[+]}) X\th F_{-\delta}\\
    =&F_{\alpha}(  F^{-\delta}\t F_{-\rho}F^{\alpha})_{[-]}  F^{-\delta}\o F^{-\rho} \ot^{B^{op}} s\circ\varepsilon(X\o (F^{-\delta}\t F_{-\rho}F^{\alpha})_{[+]}) X\t F_{-\delta}\\
    =&F_{\alpha}(   F_{-\rho}F^{\alpha})_{[-]}   F^{-\rho} \ot^{B^{op}} s\circ\varepsilon(X\o F^{-\delta} (F_{-\rho}F^{\alpha})_{[+]}) X\t F_{-\delta}\\
    =&F_{\rho}F_{-\alpha}\t(   F^{\rho}\t F_{-\alpha}\o)_{[-]}   F^{\rho}\o F^{-\alpha} \ot^{B^{op}} s\circ\varepsilon(X\o F^{-\delta} (F^{\rho}\t F_{-\alpha}\o)_{[+]}) X\t F_{-\delta}\\
    =&F_{\rho}F_{-\alpha}\t   F_{-\alpha}\o{}_{[-]}    F^{-\alpha} \ot^{B^{op}} s\circ\varepsilon(X\o F^{-\delta} F^{\rho} (F_{-\alpha}\o)_{[+]}) X\t F_{-\delta}\\
    =&F_{\rho}F_{-\alpha}\t   F_{-\alpha}\o{}_{[-]}    F^{-\alpha} \ot^{B^{op}} s\circ\varepsilon(X\o F^{-\delta} F^{\rho} F_{-\alpha}\o{}_{[+]}) X\t F_{-\delta}\\
    =&F_{\rho}F_{-\alpha}{}_{[+]}\t   F_{-\alpha}{}_{[-]}    F^{-\alpha} \ot^{B^{op}} s\circ\varepsilon(X\o F^{-\delta} F^{\rho} F_{-\alpha}{}_{[+]}\o) X\t F_{-\delta}\\
    =&F_{\rho}s(\varepsilon(F_{-\alpha}{}_{[+]}\o))F_{-\alpha}{}_{[+]}\t   F_{-\alpha}{}_{[-]}    F^{-\alpha} \ot^{B^{op}} s\circ\varepsilon(X\o F^{-\delta} F^{\rho} ) X\t F_{-\delta}\\
    =&F_{\rho}\ot^{B^{op}} s\circ\varepsilon(X\o F^{-\delta} F^{\rho} ) X\t F_{-\delta} =1\ot^{B^{op}}X
\end{align*}
for all $X\in\CL$. \end{proof}

\subsection*{Acknowledgements} We are particularly grateful to one of the referees for many detailed comments. We also thank G. Landi for helpful discussions relating to Section~\ref{secbisec}, which indeed builds on the joint work \cite{HL20}. XH was supported by Marie Curie Fellowship HADG - 101027463 agreed between QMUL  and the  European Commission.

\medskip\noindent{\bf Data availability}  Data sharing is not applicable as no data sets were generated or analysed during the current
 study.

\section*{Declarations}

{\bf Conflict of interest} On behalf of all authors, the corresponding author states that there is no conflict of interests.

\renewcommand\refname

\end{document}